\documentclass[12pt]{article}
\usepackage{amsmath, amsfonts, amssymb, amsthm, hyperref}
\usepackage[bb=boondox]{mathalfa}
\usepackage[active]{srcltx}
\newtheorem{thm}{Theorem}[section]
\newtheorem{proposition}[thm]{Proposition}
\newtheorem{lemma}[thm]{Lemma}

\newtheorem{corollary}[thm]{Corollary}
\newtheorem{definition}[thm]{Definition}

\numberwithin{equation}{section}
\begin{document}

\begin{center}
\begin{huge} 
The inverse Willmore flow \\
\end{huge} 
\vspace{15cm}
\begin{large}
Martin Mayer \\
diploma thesis\\
Mathematisches Institut der Universit\"at T\"ubingen \\
Prof. Dr. Reiner Sch\"atzle\\
2009
\thispagestyle{empty}
\end{large}
\end{center}
\newpage
\tableofcontents
\newpage
\begin{abstract}
Instead of investigating the Willmore flow for two-dimensional, closed immersed surfaces directly we turn to its inversion. We give a lower bound on the lifespan of this inverse Willmore flow, depending on the concentration of curvature in space and the extension of the initial surface, as  well as a characterization of its breakdown, which might occur in finite time.
\end{abstract}
\section{Introduction}
The Willmore functional of a closed, immersed and smooth surface
\\
$f_{0}:\Sigma \longrightarrow \mathbb{R}^{n}$
with induced area measure
$d\mu=d\mu_{\Sigma, f_{0}}$ is given by 
\begin{gather*}
W(f_{0}):=\frac{1}{4} \int_{\Sigma} \vert H_{f_{0}} \vert^{2} d\mu, \end{gather*}
where $H_{f_{0}}=g^{i,j}A_{i,j}$ denotes the mean curvature. Then the Willmore flow is the $L^{2}-$gradient flow of the Willmore functional, i.e. a solution
\begin{gather*}
\begin{split}
&f:\Sigma \times [0,T) \rightarrow \mathbb{R}^{n}\quad \text{ with} \quad 
f(\cdot,0)=f_{0}\quad \text{ of} \\ & \partial_{t}f=-\text{grad}_{L^{2}}W(f)=-\frac{1}{2}(\Delta H + Q(A^{0})H)
\end{split}
\end{gather*}
for some $T \in (0,\infty]$ and an arbitrary smooth initial surface $f_{0}$, see below.\\
If additionally $f:\Sigma \times [0,T) \longrightarrow \mathbb{R}^{n}\setminus \lbrace \mathbb{0} \rbrace,$ we may consider
the inverse flow 
\begin{gather*}
I_{\sharp}f:=\frac{f}{\Vert f \Vert^{2}}. 
\end{gather*}
Due to the fact, that the Willmore functional for surfaces is invariant under inversion,  
\begin{gather*}
W(f)=W(I_{\sharp}f),
\end{gather*}
 cf. proposition \ref{invariance}, the inverse flow satisfies 
\begin{gather*}
\partial_{t} f=-\frac{\Vert f \Vert^{8}}{2} (\Delta H + Q(A^{0})H).
\end{gather*}
This suggests to apply the techniques used by E. Kuwert and R. Sch\"atzle  \cite{ref1},  which we strongly recommend to read beforehand, since therein all the basic ideas and methods are performed with less calculations. \\
As our main result we give a lower bound on its lifespan, which depends on the concentration of curvature and the extension of the initial surface.

\begin{thm}\label{thm1.1}
For $0<R,\alpha<\infty$ and $n \in \mathbb{N}$ there exist 
\begin{gather*}
\delta=\delta(n),\; k=k(n),\;c=c(n,R,\alpha)>0
\end{gather*}
such that, if
\quad
$f_{0}:\Sigma \longrightarrow \mathbb{R}^{n} \setminus \lbrace \mathbb{0} \rbrace$ 
\quad
is a closed immersed surface 
satisfying
\begin{gather*}
\begin{split}
&
\sup_{x \in \mathbb{R}^{n}} \int_{B_{\rho}(x)} \vert A \vert^{2}d\mu \leq\delta, \\
&
\sup_{x \in \mathbb{R^{n}}}\rho^{4}\int_{B_{\rho}(x)} \vert \nabla^{2} A \vert^{2}d\mu  
\leq 
\alpha, \\
&
\rho^{-1} \Vert f_{0} \Vert_{L^{\infty}_{\mu}(\Sigma)}\leq R
\end{split}
\end{gather*}
for some $\rho>0$,
there exists an inverse Willmore flow
\begin{gather*}
f:\Sigma \times [0,T)\longrightarrow \mathbb{R}^{n} \setminus \lbrace \mathbb{0} \rbrace,\quad f(\cdot,0)=f_{0}
\end{gather*}
with $T > \rho^{-4}c$.
Moreover we have the estimates
\begin{gather*}
\begin{split}
&
\sup_{0 \leq t \leq\rho^{-4}c} \rho^{-1}\Vert f \Vert_{L^{\infty}_{\mu}(\Sigma)}
\leq 2R
\quad \text{and} \quad
\sup_{0 \leq t \leq \rho^{-4}c,\, x \in \mathbb{R}^{n}} \int_{B_{\rho}(x)} \vert A \vert^{2} d\mu\leq k  \delta.
\end{split}
\end{gather*}
\end{thm} 

\begin{thm} \label{thm1.2}
For $n \in \mathbb{N}$ there exist 
\begin{gather*}
\delta=\delta(n),\; k=k(n),\;c=c(n)>0
\end{gather*}
such that, if 
\quad 
$ f:\Sigma \longrightarrow \mathbb{R}^{n}\setminus \lbrace \mathbb{0} \rbrace
$
\quad 
is a closed immersed surface,
\begin{gather*}
T_{f_{0},\mathbb{0}}>0
\end{gather*}
maximal with respect to the existence of an inverse Willmore flow
\begin{gather*}
f:\Sigma \times [0,T)\longrightarrow \mathbb{R}^{n} \setminus \lbrace \mathbb{0} \rbrace,\quad f(\cdot,0)=f_{0}
\end{gather*} 
and 
\begin{gather*}
\begin{split}
&
\sup_{x \in \mathbb{R}^{n}} \int_{B_{\rho}(x)} \vert A \vert^{2}d\mu\lfloor_{t=0} \;\leq \delta 
\quad \text{as well as} \;\;
\sup_{0 \leq t < \min \lbrace T_{f_{0},\mathbb{0}},c\frac{\rho^{-4}}{R^{8}+R^{4}} \rbrace} \rho^{-1}\Vert f \Vert_{L^{\infty}_{\mu}(\Sigma)}\leq R
\end{split}
\end{gather*}
for some $R,\rho>0$, we have $T_{f_{0},\mathbb{0}} > c\frac{\rho^{-4}}{R^{8}+R^{4}}$ and in addition
\begin{gather*}
\begin{split}
&
\sup_{0 \leq t \leq \frac{\rho^{-4}}{R^{8}+R^{4}},\, x \in \mathbb{R}^{n}} \int_{B_{\rho}(x)} \vert A \vert^{2} d\mu \leq k  \delta.
\end{split}
\end{gather*}
\end{thm} \hspace{-22pt}
Therefore, if an inverse Willmore flow breaks down in finite time $T<\infty,$ then it diverges or  a quantum $\delta$ of the curvature concentrates in space, i.e.
\begin{gather*}
r(\tau):=\sup \lbrace \rho>0 \mid \sup_{x\in \mathbb{R}^{n}} \int_{B_{\rho}(x)} \vert A \vert^{2}d\mu \lfloor_{t=\tau} \leq \delta \rbrace \longrightarrow 0 \quad \text{as} \quad \tau \nearrow T.
\end{gather*}
\section{Basics}
Consider a smooth variation  with normal velocity, i.e.
\begin{gather*}
f: \Sigma \times 
[0,T) \rightarrow \mathbb{R}^{n} \;\text{ and }V=V^{\perp}=\partial_{t}f.
\end{gather*}
Let $X,Y$ be tangential vectorfields on $\Sigma$ independent of $t$ and 
$\Phi$ a normal valued $l$-form along $f$.
We collect some formulae.
\begin{equation} \label{formula1}
 (\nabla_{X}A)(Y,Z)=(\nabla_{Y}A)(X,Z), \quad \nabla H=-\nabla^{*}A = -2 \nabla^{*} A^{0}
\end{equation} 
\begin{equation} \label{formula2}
 K=\frac{1}{2}(\vert H \vert^{2}- \vert A \vert^{2})= \frac{1}{4}\vert H \vert^{2}-\frac{1}{2}\vert A^{0} \vert^{2}
\end{equation} 
\begin{equation*} 
\hspace{-10pt}
 R^{\perp}(X,Y)\Phi= A(e_{i},X)\langle A(e_{i}, Y), \Phi \rangle - A(e_{i},Y )\langle A(e_{i},X), \Phi \rangle 
\end{equation*}
\begin{equation} \label{formula3}
 \hspace{91pt}
 =A^{0}(e_{i},X)\langle A^{0}(e_{i}, Y), \Phi \rangle - A^{0}(e_{i},Y )\langle A^{0}(e_{i},X), \Phi \rangle
\end{equation} 
\begin{equation} \label{formula4}
 \partial_{t}P= -\partial_{t}P^{\perp}
= 
\langle \nabla_{e_{i}}V,\cdot \rangle e_{i}
+
\langle e_{i},\cdot \rangle \nabla_{e_{i}}V
\end{equation} 
\begin{equation} \label{formula5}
 P(\partial_{t}\Phi)
= 
-\langle \nabla_{e_{i}}V,\Phi\rangle e_{i}
\end{equation}
\begin{equation} \label{formula6}
\partial^{\perp}_{t} \nabla_{X} \Phi 
= 
\nabla_{X} \partial^{\perp}_{t} \Phi
+
A(X,e_{i}) \langle \nabla_{e_{i}}V,\Phi \rangle  
+ 
\nabla_{e_{i}}V \langle A(X,e_{i}),\Phi\rangle 
\end{equation} 
\begin{equation} \label{formula7}
 (\partial_{t}g)(X,Y)=-2 \langle A(X,Y),V \rangle
\end{equation} 
\begin{equation} \label{formula8}
  \partial_{t}(d\mu)=-\langle H,V \rangle d\mu
\end{equation}  
\begin{equation*} 
\partial_{t}(\nabla_{X}Y)
=
-\langle (\nabla_{e_{i}}A)(X,Y),V \rangle e_{i}
+
\langle A(X,Y), \nabla_{e_{i}}V \rangle e_{i}
\end{equation*}  
\begin{equation} \label{formula9} 
 \hspace{49pt}
 -\langle A(X,e_{i}),\nabla_{Y}V\rangle e_{i} 
 -\langle A(Y,e_{i}),\nabla_{X}V\rangle e_{i}
\end{equation} 
\begin{equation} \label{formula10}
 \partial^{\perp}_{t}A(X,Y)= \nabla^{2}_{X,Y}V-A(e_{i},X) \langle A(Y,e_{i}),V \rangle
\end{equation}  
\begin{equation} \label{formula11}
 \partial^{\perp}_{t}H = \Delta V + Q(A^{0})V + \frac{1}{2}H \langle H,V \rangle,
\end{equation}  
where by definition
\begin{equation} \label{formula12}
 A^{0}:=A-\frac{1}{2}gH 
\end{equation} 
is the tracefree part of the second fundamental form $A$,
\begin{equation} \label{formula13}
 Q(A_{0}) \Phi :=A^{0}(e_{i},e_{j}) \langle \ A^{0}(e_{i},e_{j}), \Phi \rangle
\end{equation} 
and $\nabla^{*}$ is the adjoint operator to the normal derivative $\nabla=D^{\perp}$, i.e.
\begin{equation*}
\nabla^{*}\Phi
=
-(\nabla_{e_{i}}\Phi)(e_{i},\ldots)
=
-g^{i,j}\nabla_{i}\Phi_{j,k_{2},\ldots,k_{l}}
\end{equation*}
for any normal valued $l$-form along $f$.
\begin{definition}
 Let $\Phi,\Psi$ be normal valued multilinear forms along $f$. \\
 Then let $\Phi * \Psi$ denote any normal or real valued multilinear form along $f$,\\
 depending only on $\Phi$ and $\Psi$ in a 
 universal, bi-linear way. 

 For a normal valued multilinear form $\Phi$ along $f$ we denote by
 $P^{m}_{r}(\Phi)$ 

any term of the type
 \begin{equation*}
  P^{m}_{r}(\Phi)=\sum_{i_{1}+\ldots+i_{r}=m}\nabla^{i_{1}}\Phi*\ldots*\nabla^{i_{r}}\Phi.
 \end{equation*}
\end{definition}
\section{The corresponding flow}
Let $f:\Sigma \times [0,T) \longrightarrow \mathbb{R}^{n} \setminus \lbrace \mathbb{0} \rbrace$ be a Willmore flow. \\
By (\ref{formula8}) and (\ref{formula11}) we derive for the variation of the Willmore functional
\begin{gather*}
\begin{split}
\delta W(g)(u)
=
\partial_{t}W(g+tu) \lfloor_{t=0}
=
\frac{1}{2}\int_{\Sigma} \langle \Delta H_{g} + Q(A_{g}^{0})H_{g},u \rangle d\mu_{\Sigma,g}.
\end{split}
\end{gather*}
Since for the inversion 
$I:\mathbb{R}^{n} \setminus \lbrace \mathbb{0} \rbrace \longrightarrow \mathbb{R}^{n} \setminus \lbrace \mathbb{0} \rbrace: x \longrightarrow \frac{x}{\Vert x \Vert^{2}}$ 
\begin{gather*}
\begin{split}
dI(x) \cdot dI^{T}(x)
= &
(\frac{1}{\Vert x \Vert^{2}}(\delta_{i}^{j}    -2\frac{x_{i}\delta^{j,n}x_{n}}{\Vert x \Vert^{2}}))\delta_{j,m}
(\frac{1}{\Vert x \Vert^{2}}(\delta_{k}^{m} -2\frac{x_{k}\delta^{m,r}x_{r}}{\Vert x \Vert^{2}}))\delta^{k,l} \\
= &
\frac{1}{\Vert x \Vert^{4}} \delta_{i}^{l} 
- \frac{2}{\Vert x \Vert^{6}}(\delta^{r}_{i} x_{k} \delta^{k,l}x_{r}
+ x_{i}x_{n}\delta^{n,l})
+\frac{4}{\Vert x \Vert^{8}} x_{i} x_{n} x_{r} x_{k} \delta^{n,r} \delta^{k,l}
\end{split}
\end{gather*}
\begin{equation} \label{eq0}
 \hspace{-197pt}
 =\frac{1}{\Vert x \Vert^{4}} id,
\end{equation} 
we obtain for the Jacobian of $I_{\sharp}f$, where $I_{\sharp}f(x):=I(f(x))$ for all $x \in \Sigma$, 
\begin{gather*}
 J(I_{\sharp}f) = \sqrt{ \text{det} (d(I_{\sharp}f)^{T}d(I_{\sharp}f))}=\sqrt{\text{det} \frac{df^{T}df}{\Vert f \Vert^{4}}}
 =\frac{1}{\Vert f \Vert^{4}} Jf, \\ \quad \text{i.e.} \quad
 d\mu_{I_{\sharp}f} = \Vert f \Vert^{-4} d\mu_{f}
\end{gather*}
and, since $I^{2}=id$, i.e. $id=d(I^{2})(x)=dI(I(x))\cdot dI(x)$ for all $x \in \mathbb{R}^{n}\setminus \lbrace \mathbb{0} \rbrace,$
\begin{gather*}
\Vert x \Vert^{-4} dI(I(x))=\Vert x \Vert^{-4} dI^{-1}(x)=dI^{T}(x).
\end{gather*}
By this and the  invariance under inversion of the Willmore functional,\\ cf. proposition \ref{invariance} we conclude for arbitrary
$u \in C^{\infty}_{0}(\Sigma,\mathbb{R}^{n})$
\begin{gather*}
\begin{split} 
\langle \partial_{t} I_{\sharp}f,u \rangle_{L^{2}_{\mu_{f}}(\Sigma)} 
= &
\langle dI(f)\cdot\partial_{t}f,u\rangle_{L^{2}_{\mu_{f}}(\Sigma)}
=
\langle \partial_{t} f,dI^{T}(f)u\rangle_{L^{2}_{\mu_{f}}(\Sigma)}\\
= &
\langle -grad_{L^{2}_{\mu_{f}}(\Sigma)}W(f),dI^{T}(f)u\rangle_{L^{2}_{\mu_{f}}(\Sigma)}
=
-\delta W(f)(dI^{T}(f)u) \\
= &
-\delta W(I^{2}_{\sharp}f)(\Vert f \Vert^{-4}dI(I_{\sharp} f)\cdot u)\\
= & 
-\partial_{t}W(I_{\sharp}(I_{\sharp}f+t\Vert f \Vert^{-4}u))\lfloor_{t=0} \\
= &
-\partial_{t}W(I_{\sharp}f+t\Vert f \Vert^{-4}u)\lfloor_{t=0}
= 
-\delta W(I_{\sharp}f)(\Vert f \Vert^{-4}u) \\
= &
-\frac{1}{2}\int_{\Sigma} \langle \Delta H_{I_{\sharp}f} + Q(A_{I_{\sharp}f}^{0})H_{I_{\sharp}f},\Vert f \Vert^{-4}u \rangle d\mu_{\Sigma,I_{\sharp}f}\\
= &
\langle -\frac{1}{2}\Vert I_{\sharp}f \Vert^{8} ( \Delta H_{I_{\sharp}f} + Q(A_{I_{\sharp}f}^{0})H_{I_{\sharp}f}),u \rangle_{L^{2}_{\mu_{f}}(\Sigma)}.
\end{split}
\end{gather*}
Passing from $I_{\sharp}f$ to $f$ we have to solve the quasi-linear, parabolic \\ forth-order evolution equation
\begin{equation} \label{eq1}
\partial_{t}f = - \frac{\Vert f \Vert^{8}}{2}(\Delta H + Q(A^{0})H),
\end{equation}
to whose solutions we will refer considering the inverse Willmore flow.
\section{Evolution of the curvature}
For any normal valued $l$-form $\Phi$ we define the curvature
\begin{equation*}
 R^{l}(X,Y) \Phi := R^{\perp}(X,Y)\Phi(X_{1},\ldots,X_{l})-\sum^{l}_{k=1}\Phi(X_{1},\ldots,R(X,Y)X_{k},\ldots,X_{l}).
\end{equation*} 
Using (\ref{formula3}) and the Gauss equation we deduce 
\begin{equation} \label{123}
 R^{l}(X,Y) \Phi=A*A*\Phi.
\end{equation} 
\begin{proposition}
 Let $\Phi$ be a normal valued l-form along $f$. Then we have
\begin{equation} \label{eq2}
 (\nabla \nabla^{*} - \nabla^{*} \nabla) \Phi = A * A * \Phi - \nabla^{*}T,
\end{equation} 
where $T(X_{0},\ldots,X_{l})= (\nabla_{X_{0}}\Phi)(X_{1},X_{2},\ldots,X_{l}) - (\nabla_{X_{1}}\Phi)(X_{0},X_{2},\ldots,X_{l}).$ 
\end{proposition} 
\begin{proof}
\begin{gather*}
\begin{split}
((\nabla \nabla^{*}-\nabla^{*}\nabla)\Phi)_{i_{1},\ldots,i_{l}}
= &
-\nabla_{i_{1}}(g^{i,j}\nabla_{i}\Phi_{j,i_{2},\ldots,i_{l}})
+
g^{i,j}\nabla_{i}\nabla_{j}\Phi_{i_{1},\ldots,i_{l}}\\
= &
g^{i,j}\nabla_{i}\nabla_{j}\Phi_{i_{1},\ldots,i_{l}}
-
g^{i,j}\nabla_{i}\nabla_{i_{1}}\Phi_{j,i_{2},\ldots,i_{l}} \\
& +
g^{i,j}\nabla_{i}\nabla_{i_{1}}\Phi_{j,i_{2},\ldots,i_{l}}
-
g^{i,j}\nabla_{i_{1}}\nabla_{i}\Phi_{j,i_{2},\ldots,i_{l}}\\
= &
g^{i,j}\nabla_{i}
(
\nabla_{j}\Phi_{i_{1},i_{2},\ldots,i_{l}}
-
\nabla_{i_{1}}\Phi_{j,i_{2},\ldots,i_{l}}
)
+
g^{i,j}R^{l}_{i,i_{1}}\Phi_{j,i_{2},\ldots,i_{l}}\\
= &
-(\nabla^{*}T)_{i_{1},\ldots,i_{l}}+(A*A*\Phi)_{i_{1},\ldots,i_{l}},
\end{split}
\end{gather*}
where the last equality follows by (\ref{123}).
\end{proof}
Taking $\Phi=A$ in (\ref{eq2}), we get by (\ref{formula1})
\begin{equation} \label{eq3}
 \Delta A= \nabla^{2}H + A * A * A \quad (\Delta:= -\nabla^{*} \nabla).
\end{equation} 

Considering  $\nabla \Phi,$ where $\Phi$ is a normal valued ($l$-1)-form in (\ref{eq2}), we get
\begin{equation} \label{eq4}
 \Delta(\nabla \Phi) - \nabla (\Delta \Phi) = 
 (\nabla \nabla^{*} - \nabla^{*} \nabla) \nabla \Phi = A * A * \nabla \Phi + A * \nabla A * \Phi.
\end{equation} 
For inductive reasons we need the following
\begin{proposition}
 Let $\Phi$ be a normal valued (l-1)-form along a variation \\
 $f:\Sigma \times [0,T) \rightarrow \mathbb{R}^{n}$
 with normal velocity $V=\partial_{t}f$ and 
 $\theta:\Sigma \times [0,T) \rightarrow \mathbb{R}.$ 
 \\If $\partial^{\perp}_{t} \Phi + \theta\Delta^{2} \Phi = Y,$ 
 then $\Psi = \nabla \Phi$ satisfies 
\begin{equation*}
 \partial^{\perp}_{t} \Psi + \theta \Delta^{2} \Psi 
 = \nabla Y - \nabla \theta \Delta^{2}\Phi 
 +  \theta \sum_{i+j+k=3}\nabla^{i}A *\nabla^{j}A*\nabla^{k}\Phi 
\end{equation*}
\begin{equation} \label{inductionII}
 \hspace{3pt} 
 +\nabla A*V*\Phi + A*\nabla V*\Phi.
\end{equation}
\end{proposition}
\begin{proof}
We use a local frame $e_{1},e_{2} \in T(0,1)$  independent of $t,$ for which \\
$\nabla e_{1}=\nabla e_{2}=0$ at a given point of $\Sigma \times [0,T).$ 
There we have for  $X_{k} \in \lbrace e_{1},e_{2} \rbrace$
\begin{gather*}
 \hspace{-135pt}
 (\partial^{\perp}_{t}\Psi)(X_{1},\ldots,X_{l})
 = \partial^{\perp}_{t}((\nabla_{X_{1}} \Phi)(X_{2},\ldots,X_{l})) \\
 \hspace{0pt}
 =\partial^{\perp}_{t} \nabla_{X_{1}}(\Phi(X_{2},\ldots,X_{l}))
 -\partial^{\perp}_{t} \sum^{l}_{k=2}\Phi(X_{2},\ldots,\nabla_{X_{1}}X_{k},\ldots,X_{l}).
\end{gather*}
For the first term we use (\ref{formula6}) to infer
\begin{gather*}
\begin{split}
\partial^{\perp}_{t} \nabla_{X_{1}}(\Phi(X_{2},\ldots,X_{l})) 
= &
\nabla_{X_{1}}\partial^{\perp}_{t} (\Phi (X_{2},\ldots,X_{l}))
+
(A \hspace{-0.3pt}*\nabla V \hspace{-0.3pt}* \Phi )(X_{1},\ldots,X_{l}) \\
= &
(\nabla \partial^{\perp}_{t} \Phi + A*\nabla V*\Phi)(X_{1},\ldots,X_{l}),
\end{split}
\end{gather*}
for the second one (\ref{formula9})
 \begin{gather*}
\begin{split}
\partial^{\perp}_{t} \hspace{-0.6pt}
\sum^{l}_{k=2}\Phi(X_{2},\ldots,\nabla_{X_{1}}X_{k},\ldots,X_{l}) 
= &
\sum^{l}_{k=2} P^{\perp} \Phi(X_{2},\ldots,\partial_{t} \nabla_{X_{1}}X_{k},\ldots,X_{l}) \\
& + 
\sum^{l}_{k=2} (\partial^{\perp}_{t} \Phi)(X_{2},\ldots,\nabla_{X_{1}}X_{k},\ldots,X_{l}) \\
= &
(\nabla A*V*\Phi + A*\nabla V*\Phi) (X_{1},\ldots,X_{l}).
\end{split}
 \end{gather*}
Therefore 
\begin{gather*}
  \theta \Delta^{2} \Psi +\partial^{\perp}_{t}\Psi - \nabla Y 
 =\theta \Delta^{2} (\nabla \Phi) - \nabla (\theta \Delta^{2} \Phi)
 +\nabla A*V*\Phi + A*\nabla V*\Phi.
\end{gather*}
By (\ref{eq4}) we have
\begin{gather*}
\begin{split}
\theta \Delta^{2} (\nabla \Phi)  
& - 
\nabla (\theta \Delta^{2} \Phi) \\
= &
\theta \Delta^{2} (\nabla \Phi) 
- 
\theta \nabla ( \Delta^{2} \Phi) 
- 
\nabla \theta   \Delta^{2} \Phi \\
= &\theta 
[ 
\Delta 
[
\Delta (\nabla \Phi) 
- 
\nabla (\Delta \Phi)
]
+
\Delta (\nabla (\Delta \Phi)) 
- 
\nabla (\Delta (\Delta \Phi))
] 
- 
\nabla \theta \Delta^{2} \Phi \\
= &
\theta 
[ 
\Delta 
(A * A * \nabla \Phi 
+ 
A * \nabla A * \Phi
) \\
&  \hspace{8pt} 
+
A * A * \nabla( \Delta \Phi) 
+ 
A * \nabla A *(\Delta \Phi) 
] 
-
\nabla \theta \Delta^{2} \Phi \\
= &
\theta \sum_{i+j+k=3}\nabla^{i}A* \nabla^{j}A*\nabla^{k}\Phi 
-
\nabla \theta \Delta^{2}\Phi.
\end{split}
\end{gather*}
Collecting terms the proposition follows.
\end{proof}
To state the main result of this section, we define abbreviatively
\begin{equation} \label{I(m)=}
 (i,j,k) \in I(m) :\Longleftrightarrow k \in \lbrace 1,3,5 \rbrace,\quad i \leq m+2 , \quad i+j+k=m+5.
\end{equation} 

\begin{proposition}
 Under the inverse Willmore flow, i.e. solutions of 
 $(\ref{eq1})$, we have the evolution equations
 \begin{equation} \label{evolutioneq}
 \partial^{\perp}_{t}(\nabla^{m}A) +\frac{\Vert f \Vert^{8}}{2} \Delta^{2}(\nabla^{m}A)
 = \sum_{(i,j,k) \in I(m),\, j<m+4} \nabla^{i}\Vert f \Vert^{8}* P^{j}_{k}(A).
\end{equation} 
\end{proposition}
\begin{proof}
For $m=0$  we obtain by (\ref{formula10}) and (\ref{eq1})
\begin{gather*}
\begin{split}
\partial^{\perp}_{t}A 
=&
-\nabla^{2}(\frac{\Vert f \Vert^{8}}{2} (\Delta H + Q(A^{0})H)) -A*A*(\frac{\Vert f \Vert^{8}}{2} (\Delta H + Q(A^{0})H)) \\
=&
-\frac{\Vert f \Vert^{8}}{2} \nabla^{2}(\Delta H) 
  + \Vert f \Vert^{8} P^{2}_{3}(A)
  + \nabla \Vert f \Vert^{8}*[P^{3}_{1}(A)+ P^{1}_{3}(A)] \\
&
  +\nabla^{2}\Vert f \Vert^{8}*[P^{2}_{1}(A)+P^{0}_{3}(A)]
  +\Vert f \Vert^{8}[P^{2}_{3}(A) + P^{0}_{5}(A)]. 
\end{split}
\end{gather*}
Using (\ref{eq3}) and (\ref{eq4}) we derive
\begin{gather*}
\begin{split}
\nabla^{2}(\Delta H) 
=& 
\nabla (\nabla (\Delta H)) 
= \nabla [ \Delta (\nabla H) + A*A*\nabla H +A*\nabla A*H] \\
=&
\Delta (\nabla^{2}H) + A*A*\nabla^{2}H + A*\nabla A*\nabla H + P^{2}_{3}(A) \\
=&
\Delta (\Delta A  + A*A*A) + P^{2}_{3}(A) =\Delta^{2}A +P^{2}_{3}(A).
\end{split}
\end{gather*}
Consequently
\begin{gather*}
\begin{split}
\partial^{\perp}_{t}A 
=& 
-\frac{\Vert f \Vert^{8}}{2}\Delta^{2}A
+ \Vert f \Vert^{8} [P^{2}_{3}(A) + P^{0}_{5}(A)] \\
&+\nabla  \Vert f \Vert^{8}*[P^{3}_{1}(A) + P^{1}_{3}(A)] \\
&+\nabla^{2} \Vert f \Vert^{8}*[P^{2}_{1}(A) + P^{0}_{3}(A)],
\end{split}
\end{gather*}
which is the required form for $m=0.$ \\
By (\ref{eq1}) and (\ref{inductionII}) applied to 
$\Phi:=\nabla^{m}A, \quad \theta:=\frac{\Vert f \Vert^{8}}{2} $ we conclude inductively
\begin{gather*}
\begin{split}
\partial^{\perp}_{t}\nabla^{m+1}A & + \frac{\Vert f \Vert^{8}}{2} \Delta^{2}(\nabla^{m+1} A) \\
= &
\nabla [\sum_{\,(i,j,k) \in I(m),\,j<m+4}  \hspace{-5pt}\nabla^{i}\Vert f \Vert^{8}* P^{j}_{k}(A) \; ] 
-
\nabla \frac{\Vert f \Vert^{8}}{2}*\Delta^{2}(\nabla^{m}A) \\
& +
\frac{\Vert f \Vert^{8}}{2} \sum_{i+j+k=3}\nabla^{i}A*\nabla^{j}A*\nabla^{k}(\nabla^{m}A) \\
& +
\nabla A*( \frac{\Vert f \Vert^{8}}{2} (\Delta H + Q(A^{0})H))*\nabla^{m}A\\
& +
A*\nabla(\frac{\Vert f \Vert^{8}}{2} (\Delta H + Q(A^{0})H))*\nabla^{m}A \\ 
=&
\sum_{(i,j,k) \in I(m), \,j < m+4}  \nabla^{i+1}\Vert f \Vert^{8}*P^{j}_{k}(A)\\  
& +
\sum_{(i,j,k) \in I(m),\,j < m+4}  \nabla^{i}\Vert f \Vert^{8}* P^{j+1}_{k}(A) \\
& +
\Vert f \Vert^{8} [P^{m+3}_{3}(A) + P^{m+1}_{5}(A)] \\
& +
\nabla \Vert f \Vert^{8}*[P^{m+4}_{1}(A) + P^{m+2}_{3}(A) + P^{m}_{5}(A)] \\ 
=&
\sum_{(i,j,k) \in I(m+1),\,j < m+5}  \nabla^{i}\Vert f \Vert^{8}*P^{j}_{k}(A).
\end{split}
\end{gather*}
\end{proof}

\section{Energy type inequalities}

\begin{proposition}
Let 
\quad 
$
f:\Sigma \times [0,T) \rightarrow \mathbb{R}^{n}
\quad \text{with} \quad 
V=V^{\perp}=\partial_{t}f
$, \\
$\theta:\Sigma \times [0,T) \rightarrow \mathbb{R}$ \;and 
$\Phi$ a normal valued $l$-form along $f$, which satisfies 
\begin{gather*}
\partial^{\perp}_{t}\Phi + \theta\Delta^{2}\Phi = Y.
\end{gather*}
Then for any $\eta:\Sigma \times [0,T) \rightarrow \mathbb{R}$ we have
\begin{gather*}
\hspace{-105pt}
\frac{d}{dt}\int_{\Sigma}\frac{1}{2}  \eta \vert \Phi \vert^{2}d\mu 
+
\int_{\Sigma}\langle \Delta \Phi , \Delta (\theta \eta \Phi) \rangle d\mu 
-
\int_{\Sigma} \langle Y , \eta \Phi \rangle d\mu \\
\hspace{19pt}
=
\int_{\Sigma}\eta \sum^{l}_{k=1} \langle A(e_{i_{k}},e_{j}) , V \rangle 
\langle \Phi(e_{i_{1}},\ldots ,e_{i_{k}},\ldots ,e_{i_{l}}) , 
\Phi(e_{i_{1}},\ldots , e_{j},\ldots ,e_{i_{l}}) \rangle d\mu
\end{gather*} 
\begin{equation} \label{Lemma3.1}
\hspace{-127pt}
-
\int_{\Sigma} \frac{1}{2} \vert \Phi \vert^{2} \langle H , V \rangle \eta d\mu 
+
\int_{\Sigma} \frac{1}{2} \vert \Phi \vert^{2} \partial_{t} \eta d\mu.
\end{equation}
\end{proposition}
\begin{proof}
Recalling  (\ref{formula8}) we have, using summation over repeated indices,
\begin{gather*}
\begin{split}
\partial_{t} \int_{\Sigma} & \frac{1}{2}\eta \vert \Phi \vert^{2}  d\mu
= 
\partial_{t} \int_{\Sigma}\frac{1}{2}\eta 
\langle \Phi(e_{i_{1}},\ldots,e_{i_{l}}),\Phi(e_{i_{l}},\ldots,e_{i_{l}}) \rangle d\mu \\
= &
\int_{\Sigma}\frac{1}{2} \partial_{t}\eta \vert \Phi \vert^{2} d\mu
- 
\int_{\Sigma} \frac{1}{2} \eta \vert \Phi \vert^{2} \langle H,V \rangle d\mu \\
& +
\int_{\Sigma}\eta 
\langle \Phi(e_{i_{1}},\ldots,e_{i_{l}})
,
\partial_{t}(\Phi(e_{i_{1}},\ldots,e_{i_{l}})) \rangle 
d\mu \\
= &
\int_{\Sigma} \frac{1}{2}(\partial_{t}\eta - \eta\langle H,V \rangle) \vert \Phi \vert^{2} d\mu \\
& +
\int_{\Sigma}\eta \langle \Phi(e_{i_{1}},\ldots,e_{i_{l}})
,
(\partial^{\perp}_{t}\Phi)(e_{i_{1}},\ldots,e_{i_{l}}) \rangle 
d\mu  \\
& +
\int_{\Sigma}\eta \sum^{l}_{k=1}
\langle \Phi(e_{i_{1}},\ldots,e_{i_{k}},\ldots,e_{i_{l}})
,
\Phi(e_{i_{1}},\ldots,\partial_{t}e_{i_{k}},\ldots,e_{i_{l}}) \rangle 
d\mu \\
= &
\int_{\Sigma} \frac{1}{2}(\partial_{t}\eta - \eta\langle H,V \rangle) \vert \Phi \vert^{2} d\mu 
+
\int_{\Sigma}\eta \langle \Phi ,- \theta \Delta^{2}\Phi +Y  \rangle d\mu \\  
& +
\int_{\Sigma}\eta \sum^{l}_{k=1} g(\partial_{t}e_{i_{k}},e_{j}) 
\langle \Phi(e_{i_{1}},\ldots,e_{i_{k}},\ldots,e_{i_{l}})
,
\Phi(e_{i_{1}},\ldots,e_{j},\ldots,e_{i_{l}}) \rangle 
d\mu.
\end{split}
\end{gather*}
The claim follows from (\ref{formula7}) and symmetry in $i_{k},j$.
\end{proof}
\newpage
\begin{proposition} \label{4.2}
Under the assumptions of the previous proposition let
\begin{gather*}
\theta=\varphi^{4},\;
\eta=\gamma^{s} 
\quad \text{with} \quad 
0 \leq \varphi,
\gamma :\Sigma \times [0,T) \rightarrow \mathbb{R}
\quad \text{and}\quad s \geq 4.
\end{gather*}
Then  we have for some $c>0$
\begin{gather*}
\begin{split}
\frac{d}{dt} \int_{\Sigma} \vert \Phi & \vert^{2} \gamma^{s} d\mu
+
\int_{\Sigma }\varphi^{4} \vert \nabla^{2} \Phi \vert^{2} \gamma^{s} d\mu
-
2\int_{\Sigma}  \langle Y , \Phi \rangle \gamma^{s} d\mu \\
\leq & 
\int_{\Sigma} \langle A*\Phi *\Phi , V \rangle \gamma^{s} d\mu
+
\int_{\Sigma} \vert \Phi \vert^{2} s \gamma^{s-1} \partial_{t} \gamma d\mu \\
& +
c\int_{\Sigma} \varphi^{4} 
(
s^{4}\vert \nabla \gamma \vert^{4} 
+
s^{2}\gamma^{2}\vert \nabla^{2} \gamma \vert^{2}
+\gamma^{4} 
[\;
\vert A \vert^{4} 
+ 
\vert \nabla A \vert ^{2}
\;]
\,)
\gamma^{s-4}\vert \Phi \vert^{2}d\mu \\
& +
c \int_{\Sigma} \vert \nabla \varphi \vert^{2}
(
\vert \nabla \varphi \vert^{2} \gamma^{2}  
+ 
s^{2}\varphi^{2} \vert \nabla \gamma \vert^{2}
)
\gamma^{s-2} \vert \Phi \vert^{2} d\mu 
\end{split}
\end{gather*}
\begin{equation} \label{Lemma 3.2}
 \hspace{-155pt}
 +c\int_{\Sigma} \vert \nabla^{2} \varphi \vert^{2} \varphi^{2} \gamma^{s} \vert \Phi \vert^{2} d\mu.
\end{equation} 
\end{proposition}
\begin{proof}
 From the previous proposition, (\ref{Lemma3.1}), we infer
\begin{gather*}
\begin{split}
\frac{d}{dt} \int_{\Sigma}\vert \Phi \vert^{2} \gamma^{s} d\mu 
+ 
2\int_{\Sigma}\langle \Delta \Phi , & \Delta (\varphi^{4} \gamma^{s} \Phi) \rangle d\mu 
- 
2\int_{\Sigma} \langle Y , \gamma^{s} \Phi \rangle d\mu \\
\leq & 
\int_{\Sigma}  \langle A* \Phi *\Phi , V \rangle \gamma^{s} d\mu 
+
\int_{\Sigma} \vert \Phi \vert^{2} s \gamma^{s-1} \partial_{t} \gamma d\mu,
\end{split}
\end{gather*}
so that it will be sufficient to prove
\begin{gather*}
\hspace{-105pt} 
\begin{split}
\int_{\Sigma }\varphi^{4} & \vert \nabla^{2} \Phi \vert^{2} \gamma^{s} d\mu \\
\leq & 
2 \int_{\Sigma}\langle \Delta \Phi , \Delta (\varphi^{4} \gamma^{s} \Phi) \rangle d\mu\\
& +
c\int_{\Sigma} \varphi^{4}
(
s^{4}\vert \nabla \gamma \vert^{4} 
+
s^{2}\gamma^{2}\vert \nabla^{2} \gamma \vert^{2}
)
\gamma^{s-4}\vert \Phi \vert^{2}d\mu \\
& +
c \int_{\Sigma} \vert \nabla \varphi \vert^{2}
(
\vert \nabla \varphi \vert^{2} \gamma^{2}  
+ 
s^{2}\varphi^{2} \vert \nabla \gamma \vert^{2}
)
\gamma^{s-2} \vert \Phi \vert^{2} d\mu 
\end{split}
\end{gather*}
\begin{equation} \label{eq5}
\hspace{16pt}
+
c\int_{\Sigma} \vert \nabla^{2} \varphi \vert^{2} \varphi^{2} \gamma^{s} \vert \Phi \vert^{2} d\mu
+
c\int_{\Sigma} \varphi^{4} 
(\;
\vert A \vert^{4} 
+
 \vert \nabla A \vert^{2}
\; ) 
 \gamma^{s} \vert \Phi \vert^{2} d\mu.   
\end{equation}
Using Young's inequality we get
\begin{gather*}
\begin{split}
\int_{\Sigma} \varphi^{4} \vert & \nabla^{2}  \Phi \vert^{2} \gamma^{s} d\mu 
\leq
\int_{\Sigma} 
\langle \nabla^{2} \Phi 
, 
\nabla^{2} (\varphi^{4} \gamma^{s} \Phi) \rangle 
d\mu \\
& \quad \quad \;+
2\int_{\Sigma} \vert \nabla^{2} \Phi \vert \vert \nabla (\varphi^{4} \gamma^{s}) \vert \vert \nabla \Phi \vert d\mu
+
\int_{\Sigma} \vert \nabla^{2} \Phi \vert \vert \Phi \vert \vert \nabla^{2} (\varphi^{4} \gamma^{s}) \vert d\mu \\
\leq &
\int_{\Sigma} 
\langle \nabla^{2} \Phi 
, 
\nabla^{2} (\varphi^{4} \gamma^{s} \Phi ) \rangle 
d\mu 
+
c\int_{\Sigma} \vert \nabla^{2} \Phi \vert \vert \nabla \varphi \vert \varphi^{3}\gamma^{s} \vert \nabla \Phi \vert d\mu \\
& +
cs\int_{\Sigma} \vert \nabla^{2} \Phi \vert \varphi^{4}  \vert \nabla \gamma \vert \gamma^{s-1} \vert \nabla \Phi \vert d\mu 
+
c\int_{\Sigma} \vert \nabla^{2} \Phi \vert \vert \Phi \vert \vert \nabla^{2} \varphi \vert \varphi^{3} \gamma^{s} d\mu \\
& +
c\int_{\Sigma} \vert \nabla^{2} \Phi \vert \vert \Phi \vert \vert \nabla \varphi \vert^{2} \varphi^{2} \gamma^{s} d\mu 
+
cs\int_{\Sigma} \vert \nabla^{2} \Phi \vert \vert \Phi \vert \vert \nabla \varphi \vert \varphi^{3} \vert \nabla \gamma \vert \gamma^{s-1}d\mu \\
& +
cs^{2}\int_{\Sigma} \vert \nabla^{2} \Phi \vert \vert \Phi \vert \varphi^{4} \vert \nabla \gamma \vert^{2} \gamma^{s-2} d\mu 
+
cs\int_{\Sigma} \vert \nabla^{2} \Phi \vert \vert \Phi \vert 
\varphi^{4} \vert \nabla^{2}  \gamma \vert \gamma^{s-1}d\mu \\ 
\leq&
\int_{\Sigma} \langle \nabla^{2} \Phi , \nabla^{2} (\varphi^{4} \gamma^{s} \Phi ) \rangle d\mu \\
& +
\varepsilon \int_{\Sigma} \varphi^{4} \vert \nabla^{2} \Phi \vert^{2} \gamma^{s} d\mu 
+
c(\varepsilon) \int_{\Sigma} \vert \nabla \varphi \vert^{2} \varphi^{2} \gamma^{s} \vert \nabla \Phi \vert^{2} d\mu \\
& +
c(\varepsilon)s^{2} \int_{\Sigma} \varphi^{4} \vert \nabla \gamma \vert^{2} \gamma^{s-2}\vert \nabla \Phi \vert^{2} d\mu 
+
c(\varepsilon) \int_{\Sigma} \vert \nabla^{2} \varphi \vert^{2} \varphi^{2} \gamma^{s} \vert \Phi \vert^{2} d\mu \\
&+
c(\varepsilon) \int_{\Sigma} \vert \nabla \varphi \vert^{4} \gamma^{s} \vert \Phi \vert^{2} d\mu 
+
c(\varepsilon) s^{2}\int_{\Sigma} \vert \nabla \varphi \vert^{2} \varphi^{2} \vert \nabla \gamma \vert^{2} \gamma^{s-2} \vert \Phi \vert^{2} d\mu \\
& +
c(\varepsilon)s^{4} \int_{\Sigma} \varphi^{4} \vert \nabla \gamma \vert^{4} \gamma^{s-4} \vert \Phi \vert^{2} d\mu 
+
c(\varepsilon) s^{2}\int_{\Sigma} \varphi^{4} \vert \nabla^{2} \gamma \vert^{2} \gamma^{s-2} \vert \Phi \vert^{2} d\mu.
\end{split}
\end{gather*}
Hence by absorption with $\varepsilon=\frac{1}{3}$
\begin{gather*}
\begin{split}
\int_{\Sigma} \varphi^{4} & \vert \nabla^{2} \Phi \vert^{2} \gamma^{s} d\mu \\ 
\leq &
\frac{3}{2}\int_{\Sigma} \langle \nabla^{2} \Phi , \nabla^{2} (\varphi^{4} \gamma^{s} \Phi ) \rangle d\mu 
+ 
c\int_{\Sigma} \varphi^{4} 
 (
 s^{4}\vert \nabla \gamma \vert^{4} 
 +
 s^{2}\gamma^{2}\vert \nabla^{2} \gamma \vert^{2}
 )
\gamma^{s-4}\vert \Phi \vert^{2}d\mu \\
& +
c \hspace{-1pt}\int_{\Sigma} \vert \nabla \varphi \vert^{2}
(
\vert \nabla \varphi \vert^{2} \gamma^{2}  
+ 
s^{2}\varphi^{2} \vert \nabla \gamma \vert^{2}
)
\gamma^{s-2} \vert \Phi \vert^{2} d\mu 
+
c
\hspace{-1pt}\int_{\Sigma} \vert \nabla^{2} \varphi \vert^{2} \varphi^{2} \gamma^{s} \vert \Phi \vert^{2} d\mu 
\end{split}
\end{gather*}
\begin{equation} \label{eq6}
\hspace{-100pt}
+
c\int_{\Sigma}\varphi^{2}
(
\vert \nabla \varphi \vert^{2}  \gamma^{2} 
+
s^{2}\varphi^{2} \vert \nabla \gamma \vert^{2}
) 
\gamma^{s-2} \vert \nabla \Phi \vert^{2} d\mu.
\end{equation} 
\hspace{-22pt}
In the last term we integrate by parts to get
\begin{gather*}
\hspace{-40pt}
\begin{split}
\int_{\Sigma}\varphi^{2} &
(
\vert \nabla \varphi \vert^{2}  \gamma ^{2} 
+ 
s^{2}\varphi^{2} \vert \nabla \gamma \vert^{2}
) 
\gamma^{s-2} \vert \nabla \Phi \vert^{2} d\mu \\
\leq &
-\int_{\Sigma}\varphi^{2}
(
\vert \nabla \varphi \vert^{2}  \gamma^{2} 
+
s^{2}\varphi^{2} \vert \nabla \gamma \vert^{2}
) 
\gamma^{s-2} \langle \Phi , \Delta \Phi \rangle d\mu \\
& + 
c\int_{\Sigma} \vert \nabla \Phi \vert \vert \Phi \vert
[
\varphi \vert \nabla \varphi \vert^{3} \gamma^{s} 
+
\varphi^{2} \vert \nabla \varphi \vert \vert \nabla^{2} \varphi \vert \gamma^{s} \\
& \hspace{78pt}
+
s\varphi^{2} \vert \nabla \varphi \vert^{2} \vert \nabla \gamma \vert \gamma^{s-1} 
+
s^{2}\varphi^{3} \vert \nabla \varphi \vert \vert \nabla \gamma \vert^{2} \gamma^{s-2} \\
& \hspace{78pt}
+
s^{2}\varphi^{4} \vert \nabla \gamma \vert \vert \nabla^{2} \gamma \vert \gamma^{s-2}
+
s^{3}\varphi^{4} \vert \nabla \gamma \vert^{3} \gamma^{s-3} 
]\;
d\mu \\
\leq &
\varepsilon \int_{\Sigma} \varphi^{4} \vert \nabla^{2} \Phi \vert^{2} \gamma^{s} d\mu
+
c(\varepsilon) \int_{\Sigma}
(
\vert \nabla \varphi \vert^{2}  \gamma^{2} 
+
s^{2}\varphi^{2} \vert \nabla \gamma \vert^{2})^{2} \gamma^{s-4} \vert \Phi \vert^{2} d\mu\\
& +
\varepsilon  \int_{\Sigma}\varphi^{2}(\vert \nabla \varphi \vert^{2}  \gamma^{2} 
+
s^{2}\varphi^{2} \vert \nabla \gamma \vert^{2}
) 
\gamma^{s-2}\vert \nabla \Phi \vert^{2} d\mu \\
& +
c(\varepsilon)
\hspace{-1pt}
\int_{\Sigma}  \vert \Phi \vert^{2}
[\,
\vert \nabla \varphi \vert^{4} \gamma^{s} 
+ 
\varphi^{2} \vert \nabla^{2} \varphi \vert^{2} \gamma^{s} 
+ 
s^{2}\varphi^{2} \vert \nabla \varphi \vert^{2} \vert \nabla \gamma \vert ^{2} \gamma^{s-2} \\
& \hspace{84pt}+ 
s^{2}\varphi^{4} \vert \nabla^{2} \gamma \vert^{2} \gamma^{s-2} 
+
s^{4}\varphi^{4} \vert \nabla \gamma \vert^{4} \gamma^{s-4} 
\,]\;
d\mu,
\end{split}
\end{gather*} 
so that we have by absorption
\begin{gather*}
\begin{split}
 \int_{\Sigma}\varphi^{2}
(
\vert \nabla \varphi \vert^{2}  \gamma^{2} 
& +
s^{2}\varphi^{2} \vert \nabla \gamma \vert^{2}
) 
\gamma^{s-2} \vert \nabla \Phi \vert^{2} d\mu \\
\leq &
\varepsilon \int_{\Sigma} \varphi^{4} \vert \nabla^{2} \Phi \vert^{2} \gamma^{s} d\mu  \\
& +
c(\varepsilon)\int_{\Sigma} \varphi^{4}  
(
s^{4}\vert \nabla \gamma \vert^{4} 
+
s^{2}\gamma^{2}\vert \nabla^{2} \gamma \vert^{2}
)
\gamma^{s-4}\vert \Phi \vert^{2}d\mu \\ 
& + 
c(\varepsilon) \int_{\Sigma} \vert \nabla \varphi \vert^{2}
(
\vert \nabla \varphi \vert^{2} \gamma^{2}  
+ 
s^{2}\varphi^{2} \vert \nabla \gamma \vert^{2}
)
\gamma^{s-2} \vert \Phi \vert^{2} d\mu
\end{split}
\end{gather*}
\begin{equation} \label{eq7}
\hspace{-32pt}
 +c(\varepsilon)\int_{\Sigma} \vert \nabla^{2} \varphi \vert^{2} \varphi^{2} 
\gamma^{s} \vert \Phi \vert^{2} d\mu. 
\end{equation}
Plugging (\ref{eq7}) into (\ref{eq6}) yields by absorption with  $\varepsilon=\frac{1}{7}$ yields

\begin{gather*}
\hspace{-40pt}
\begin{split}
\int_{\Sigma} \varphi^{4} \vert \nabla^{2} \Phi \vert^{2} \gamma^{s} d\mu  
\leq & 
\frac{7}{4}\int_{\Sigma} \langle \nabla^{2} \Phi , \nabla^{2}  (\varphi^{4} \gamma^{s} \Phi ) \rangle d\mu \\
&  +
c\int_{\Sigma} \varphi^{4} 
(
s^{4}\vert \nabla \gamma \vert^{4} 
+
s^{2}\gamma^{2}\vert \nabla^{2} \gamma \vert^{2}
)
\gamma^{s-4}\vert \Phi \vert^{2}d\mu \\
& +
c \int_{\Sigma} \vert \nabla \varphi \vert^{2}
(
\vert \nabla \varphi \vert^{2} \gamma^{2}  
+ 
s^{2}\varphi^{2} \vert \nabla \gamma \vert^{2}
)
\gamma^{s-2} \vert \Phi \vert^{2} d\mu
\end{split}
\end{gather*}
\begin{equation} \label{eq8}
\hspace{-46pt}
 +c\int_{\Sigma} \vert \nabla^{2} \varphi \vert^{2} \varphi^{2} \gamma^{s} \vert \Phi \vert^{2} d\mu. 
\end{equation} 
Next  we compute using (\ref{eq4})
\begin{gather*}
\hspace{-7pt}
\begin{split}
\int_{\Sigma} \langle   \nabla^{2} &  \Phi ,  \nabla^{2} (\varphi^{4} \gamma^{s} \Phi ) \rangle d\mu \\
= &
\int_{\Sigma} \langle \Delta \Phi , \Delta (\varphi^{4} \gamma^{s} \Phi) \rangle d\mu 
+
\int_{\Sigma} \langle A*A*\nabla \Phi + A*\nabla A*\Phi , \nabla (\varphi^{4} \gamma^{s} \Phi) \rangle d\mu \\
\leq &
\int_{\Sigma} \langle \Delta \Phi , \Delta (\varphi^{4} \gamma^{s} \Phi) \rangle d\mu \\
& +
c\int_{\Sigma} \varphi^{4} \vert A \vert^{2} \gamma^{s} \vert \nabla \Phi \vert^{2} 
+
cs\int_{\Sigma} \varphi^{4} \vert A \vert^{2} \vert \nabla \gamma \vert \gamma^{s-1} 
\vert \nabla \Phi \vert \vert \Phi \vert d\mu \\
& +
c\int_{\Sigma}\varphi^{3} \vert \nabla \varphi \vert \vert A \vert^{2} \gamma^{s} 
\vert \nabla \Phi \vert \vert \Phi \vert d\mu
+
c\int_{\Sigma} \varphi^{4} \vert A \vert \vert \nabla A \vert \gamma^{s} 
\vert \nabla \Phi \vert \vert \Phi \vert d\mu 
\end{split}
\end{gather*}
\begin{equation} \label{eq9}
\hspace{28pt}
 +cs\int_{\Sigma}\varphi^{4} \vert A \vert \vert \nabla A \vert \vert \nabla \gamma \vert \gamma^{s-1} 
  \vert \Phi \vert^{2} d\mu
 +c\int_{\Sigma}\varphi^{3} \vert \nabla \varphi \vert \vert A \vert \vert \nabla A \vert \gamma^{s} 
  \vert \Phi \vert^{2} d\mu.
\end{equation}
Treating the second summand by integration by parts 
\begin{gather*}
\begin{split}
\int_{\Sigma} \varphi^{4} \vert A \vert^{2} & \gamma^{s}  \vert \nabla \Phi \vert^{2} \\
\leq &
-\int_{\Sigma} \varphi^{4} \vert A \vert^{2} \gamma^{s} \langle \Phi , \Delta \Phi \rangle d\mu 
+
c\int_{\Sigma} \varphi^{4} \vert A \vert \vert \nabla A \vert \gamma^{s}
\vert \nabla \Phi \vert \vert \Phi \vert d\mu \\ 
& +
cs\int_{\Sigma} \varphi^{4} \vert A \vert^{2} \vert \nabla \gamma \vert \gamma^{s-1} 
\vert \nabla \Phi \vert \vert \Phi \vert d\mu 
+c\int_{\Sigma} \varphi^{3} \vert \nabla \varphi \vert \vert A \vert^{2} \gamma^{s} 
\vert \nabla \Phi \vert \vert \Phi \vert d\mu \\
\end{split}
\end{gather*}
\begin{gather*}
\begin{split}
\leq & 
\varepsilon \int_{\Sigma} \varphi^{4} \vert \nabla^{2} \Phi \vert^{2} \gamma^{s} d\mu
+
\varepsilon \int_{\Sigma} \varphi^{4} \vert A \vert^{2} \gamma^{s} \vert \nabla \Phi \vert^{2} d\mu\\
& +
c(\varepsilon) \int_{\Sigma} \varphi^{2}
(\vert \nabla \varphi \vert^{2} \gamma^{2} 
+
s^{2}\varphi^{2} \vert \nabla \gamma \vert^{2}
) 
\gamma^{s-2}\vert \nabla \Phi \vert^{2} d\mu \\
&+
c(\varepsilon) \int_{\Sigma} \varphi^{4} \vert A \vert^{4}  \gamma^{s} \vert \Phi \vert^{2} d\mu
+
c(\varepsilon)\int_{\Sigma} \varphi^{4} \vert \nabla A \vert^{2} \gamma^{s} \vert \Phi \vert^{2} d\mu
\end{split}
\end{gather*} 
we obtain by absorption
\begin{gather*}
\hspace{-20pt}
\begin{split}
\int_{\Sigma} \varphi^{4} \vert A \vert^{2} \gamma^{s} \vert \nabla \Phi \vert^{2} 
\leq & 
\varepsilon \int_{\Sigma} \varphi^{4} \vert \nabla^{2} \Phi \vert^{2} \gamma^{s} d\mu \\
& +
c(\varepsilon) \int_{\Sigma} \varphi^{2}
(
\vert \nabla \varphi \vert^{2}  \gamma^{2} 
+
s^{2}\varphi^{2} \vert \nabla \gamma \vert^{2}
) 
\gamma^{s-2} \vert \nabla \Phi \vert^{2} d\mu 
\end{split}
\end{gather*}
\begin{equation} \label{eq10}
\hspace{26pt}
+
c(\varepsilon) \int_{\Sigma} \varphi^{4}
[\;
\vert A \vert^{4}  
+
\vert \nabla A \vert^{2}
\;] 
\gamma^{s} \vert \Phi \vert^{2} d\mu. 
\end{equation} 
The remaining terms in (\ref{eq9}) are estimated to
\begin{gather*}
\begin{split}
s\int_{\Sigma} &  \varphi^{4} \vert A \vert^{2} \vert \nabla \gamma \vert \gamma^{s-1} 
\vert \nabla \Phi \vert \vert \Phi \vert d\mu 
+
\int_{\Sigma}\varphi^{3} \vert \nabla \varphi \vert \vert A \vert^{2} \gamma^{s} 
\vert \nabla \Phi \vert \vert \Phi \vert d\mu \\
&+
\int_{\Sigma} \varphi^{4} \vert A \vert \vert \nabla A \vert \gamma^{s } 
\vert \nabla \Phi \vert \vert \Phi \vert d\mu 
+
s\int_{\Sigma}\varphi^{4} \vert A \vert \vert \nabla A \vert \vert \nabla \gamma \vert \gamma^{s-1} 
\vert \Phi \vert^{2} d\mu\\
& +
\int_{\Sigma}\varphi^{3} \vert \nabla \varphi \vert \vert A \vert \vert \nabla A \vert \gamma^{s} 
\vert \Phi \vert^{2} d\mu \\
\leq &
\int_{\Sigma} \varphi^{4} \vert A \vert^{2} \gamma^{s} \vert \nabla \Phi \vert^{2} \\
& +
c \int_{\Sigma} \varphi^{2}
(
\vert \nabla \varphi \vert^{2}  \gamma^{2} 
+
s^{2}\varphi^{2} \vert \nabla \gamma \vert^{2}
) 
\gamma^{s-2}\vert \nabla \Phi \vert^{2} d\mu \\
& +
cs^{4}\int_{\Sigma}\varphi^{4} \vert \nabla \gamma \vert^{4} \gamma^{s-4} \vert \Phi \vert^{2} d\mu 
+
c\int_{\Sigma} \vert \nabla \varphi \vert^{4} \gamma^{s} \vert \Phi \vert^{2} d\mu 
\end{split} 
\end{gather*}
\begin{equation} \label{eq11} 
\hspace{-122pt}
+
c\int_{\Sigma} \varphi^{4} 
[\;
\vert A \vert^{4}
+
\vert \nabla A \vert^{2}
\;]
\gamma^{s} \vert \Phi \vert^{2} d\mu.
\end{equation} 
Now plugging (\ref{eq11}) into (\ref{eq9}), then using (\ref{eq10}) yields
\begin{gather*}
\begin{split}
\int_{\Sigma} \langle \nabla^{2} \Phi ,   \nabla^{2}  & (\varphi^{4} \gamma^{s} \Phi )  \rangle d\mu  \leq 
\varepsilon \int_{\Sigma} \varphi^{4} \vert \nabla^{2} \Phi \vert^{2} \gamma^{s} d\mu 
+ 
\int_{\Sigma} \langle \Delta \Phi , \Delta (\varphi^{4} \gamma^{s} \Phi) \rangle d\mu \\
& +
c(\varepsilon) \int_{\Sigma} \varphi^{2}
(
\vert \nabla \varphi \vert^{2}  \gamma^{2} 
+
s^{2}\varphi^{2} \vert \nabla \gamma \vert^{2}
) 
\gamma^{s-2} \vert \nabla \Phi \vert^{2} d\mu \\
& +
c\int_{\Sigma}\varphi^{4} s^{4}\vert \nabla \gamma \vert^{4} \gamma^{s-4} \vert \Phi \vert^{2} d\mu 
+
c\int_{\Sigma} \vert \nabla \varphi \vert^{4} \gamma^{s} \vert \Phi \vert^{2} d\mu
\\
& +
c(\varepsilon) \int_{\Sigma} \varphi^{4} 
(\;
\vert A \vert^{4}
+ 
\vert \nabla A \vert^{2} 
\;)
\gamma^{s} \vert \Phi \vert^{2} d\mu  \\
\leq &
\varepsilon \int_{\Sigma} \varphi^{4} \vert \nabla^{2} \Phi \vert^{2} \gamma^{s} d\mu 
+ 
\int_{\Sigma} \langle \Delta \Phi , \Delta (\varphi^{4} \gamma^{s} \Phi) \rangle d\mu \\
& +
c(\varepsilon)\int_{\Sigma} \varphi^{4} 
(
s^{4}\vert \nabla \gamma \vert^{4} 
+
s^{2}\gamma^{2}\vert \nabla^{2} \gamma \vert^{2}
)
\gamma^{s-4}\vert \Phi \vert^{2}d\mu \\
& +
c(\varepsilon) \int_{\Sigma} \vert \nabla \varphi \vert^{2}
(
\vert \nabla \varphi \vert^{2} \gamma^{2}  
+
s^{2} \varphi^{2} \vert \nabla \gamma \vert^{2}
)
\gamma^{s-2} \vert \Phi \vert^{2} d\mu\\
& +
c(\varepsilon)\int_{\Sigma} \vert \nabla^{2} \varphi \vert^{2} \varphi^{2} 
\gamma^{s} \vert \Phi \vert^{2} d\mu
\end{split}
\end{gather*}
\begin{equation} \label{eq12}
\hspace{-40pt}
+c(\varepsilon)
\int_{\Sigma} \varphi^{4} 
(\;
\vert A \vert^{4}
+
\vert \nabla A \vert^{2}
\;)
\gamma^{s} \vert \Phi \vert^{2} d\mu,
\end{equation} 
where (\ref{eq7}) was used in the last step.\\
Inserting the above inequality (\ref{eq12}) into (\ref{eq8}) verifies by absorption with $\varepsilon=\frac{1}{14}$ equation (\ref{eq5}), what completes the proof.
\end{proof} 
We  now apply the data of the inverse Willmore flow to this proposition. \\
Therefore let
$\gamma=\widetilde{\gamma} \circ f,$ 
where 
$0 \leq \widetilde{\gamma} \leq 1$ 
and 
$\Vert \widetilde{\gamma} \Vert_{C^{2}(\mathbb{R}^{n})} \leq \Lambda.$ \\ 
This implies 
$ \nabla \gamma = d \widetilde{\gamma}(f) \cdot \nabla f$ 
and
$ \nabla^{2} \gamma = d^{2}\widetilde{\gamma}(f) \cdot (\nabla f \otimes \nabla f) + d\widetilde{\gamma}(f)\cdot A$, \\
so that we have by this and (\ref{eq1}) the (in-)equalities 
\begin{gather*} 
\hspace{-30pt}
\vert \gamma \vert \leq \Lambda,\quad \vert \nabla \gamma \vert \leq \Lambda, \quad \vert \nabla^{2}\gamma \vert \leq \Lambda(1 + \vert A \vert) \;\text{ for }\; \Vert \widetilde{\gamma}\Vert_{C^{2}(\mathbb{R}^{n})}\leq \Lambda, 
\end{gather*}
\begin{equation} \label{eq13} \hspace{-12pt}
 V=\partial_{t} f = - \frac{\Vert f \Vert^{8}}{2}(\Delta H + Q(A^{0})H) = \Vert f \Vert^{8} (P^{2}_{1}(A) + P^{0}_{3}(A)),                        
\end{equation} 
\begin{gather*}
\hspace{5pt}
\partial_{t}\gamma=
d \widetilde{\gamma}(f)(- \frac{\Vert f \Vert^{8}}{2}(\Delta H + Q(A^{0})H))
=- \frac{\Vert f \Vert^{8}}{2}d \widetilde{\gamma}(f) (\Delta H + P^{0}_{3}(A)),
\end{gather*} 
to which we will refer as (\ref{eq13}).
\newpage
\begin{proposition} \label{4.3}
For $n,m,\Lambda>0$ and $s \geq 4$ there exists
\begin{gather*}
 c=c(n,m,s,\Lambda)>0
\end{gather*}
such that, if
\;
$f:\Sigma \times [0,T) \rightarrow \mathbb{R}^{n}\setminus \lbrace \mathbb{0} \rbrace$ \;
is an inverse Willmore flow and 
$\gamma=\widetilde{\gamma} \circ f$ as in (\ref{eq13}), 
we have
\begin{gather*}
\hspace{-60pt}
\begin{split}
\frac{d}{dt} \int_{\Sigma} \vert \nabla^{m} A \vert^{2} & \gamma^{s} d\mu
 + 
\int_{\Sigma} \frac{\Vert f \Vert^{8}}{4}
\vert \nabla^{m+2} A \vert^{2} \gamma^{s} d\mu \\
\leq & 
\int_{\Sigma}    \sum_{(i,j,k) \in I(m),\, j<m+4} \nabla^{i}
\Vert f \Vert^{8}* P^{j}_{k}(A)*\nabla^{m} A \gamma^{s} d\mu
\end{split}
\end{gather*}
\begin{equation} \label{eq14}
\hspace{-18pt}
+c
\int_{[\gamma>0]} 
[\;
\Vert f \Vert^{8} \gamma^{s-4} 
+
\Vert f \Vert^{4}\gamma^{s}
\;]\,
\vert \nabla^{m} A \vert^{2} d\mu.
\end{equation} 
\end{proposition}
\begin{proof}
Consider (\ref{Lemma 3.2}) with $\Phi=\nabla^{m}A.$
From (\ref{evolutioneq}) we infer 
\begin{gather*}
\varphi=2^{-\frac{1}{4}}\Vert f \Vert^{2}
\text{ and }
Y=\sum_{(i,j,k) \in I(m),\,j<m+4} \nabla^{i}\Vert f \Vert^{8}* P^{j}_{k}(A).
\end{gather*}
By (\ref{eq13}) we then have
\vspace{-0pt}
\begin{gather*}
\begin{split}
& \int_{\Sigma}  2 \langle  Y ,  \Phi \rangle  \gamma^{s} d\mu 
 + 
\int_{\Sigma} \langle A*\Phi *\Phi , V \rangle \gamma^{s} d\mu 
+
\int_{\Sigma} \vert \Phi \vert^{2} s \gamma^{s-1} \partial_{t} \gamma d\mu \\
& +
c\int_{\Sigma} \varphi^{4} 
(
s^{4}\vert \nabla \gamma \vert^{4} 
+s^{2}\gamma^{2}\vert \nabla^{2} \gamma \vert^{2}
+
\gamma^{4} [\,\vert A \vert^{4} + \vert \nabla A \vert ^{2}]\,
)
\gamma^{s-4}\vert \Phi \vert^{2}d\mu \\
& +
c \int_{\Sigma} \vert \nabla \varphi \vert^{2}
(
\vert \nabla \varphi \vert^{2} \gamma^{2}  
+
s^{2}\varphi^{2} \vert \nabla \gamma \vert^{2}
)
\gamma^{s-2} \vert \Phi \vert^{2} d\mu  
+
c\int_{\Sigma} \vert \nabla^{2} \varphi \vert^{2} \varphi^{2} \gamma^{s} \vert \Phi \vert^{2} d\mu\\
\leq &
\int_{\Sigma}\sum_{(i,j,k) \in I(m),\, j<m+4} 
\nabla^{i}\Vert f \Vert^{8}* P^{j}_{k}(A)*\Phi \gamma^{s} d\mu \\
& +
\int_{\Sigma} \Vert f \Vert^{8} A*\nabla^{m}A*(P^{2}_{1}(A) 
+
P^{0}_{3}(A))*\Phi \gamma^{s}d\mu \\
& -
\frac{s}{2}\int_{\Sigma} \Vert f \Vert^{8} \vert \Phi \vert^{2} \gamma^{s-1}d \widetilde{\gamma}(f)
(\Delta H + P^{0}_{3}(A)) d\mu \\
& + 
c\int_{[\gamma>0]} \Vert f \Vert^{8}
(
s^{4}\Lambda^{4}
+
s^{2} \Lambda^{2} \gamma^{2}(1 + \vert A \vert)^{2} 
+ 
\gamma^{4}(\vert A \vert^{4} + \vert \nabla A \vert^{2})
) 
\gamma^{s-4} \vert \Phi \vert^{2} d\mu \\
& +
c\int_{\Sigma} \vert \nabla \Vert f \Vert^{2} \vert^{2} 
(
\vert \nabla \Vert f \Vert^{2} \vert^{2} \gamma^{2} 
+ 
s^{2}\Lambda^{2}\Vert f \Vert^{4}
)
\gamma^{s-2} \vert \phi \vert^{2} d\mu \\
& +
c\int_{\Sigma} \Vert f \Vert^{4}\vert \nabla^{2} \Vert f \Vert^{2}    \vert^{2} \gamma^{s}\vert \Phi \vert^{2} d\mu 
\end{split}
\end{gather*}
\begin{gather*}
\begin{split}
\leq &
\int_{\Sigma}    \sum_{(i,j,k) \in I(m),\, j<m+4} \nabla^{i}\Vert f \Vert^{8}*   P^{j}_{k}(A)*\Phi \gamma^{s} d\mu \\
& \hspace{-2pt}-
\frac{s}{2}
\hspace{-2pt}
\int_{\Sigma} 
\hspace{-2pt}\Vert f \Vert^{8} \vert \Phi \vert^{2} \gamma^{s-1}d \widetilde{\gamma} (f) (\Delta H)\,d\mu  
\hspace{-2pt}
+
\hspace{-2pt}
c\hspace{-3pt}\int_{[\gamma>0]} 
\hspace{-1pt}
[\;
s^{4}\Lambda^{4}\Vert f \Vert^{8} \gamma^{s-4} 
\hspace{-2pt}
+
\hspace{-2pt}
\Vert f \Vert^{4} \gamma^{s}
\;]
\vert \Phi \vert^{2} d\mu,
\end{split}
\end{gather*}
where we made use of Young's inequality with $p=\frac{4}{3},\,q=\frac{1}{4}$ to estimate
\begin{gather*}
\begin{split}
-\frac{s}{2}\int_{\Sigma} \Vert f \Vert^{8} \vert \Phi \vert^{2} \gamma^{s-1} & d \widetilde{\gamma}(f)
( P^{0}_{3}(A)) d\mu
\leq
cs\Lambda\int_{\Sigma}\Vert f \Vert^{8}\vert A \vert^{3} \gamma^{s-1}\vert \Phi \vert^{2}d\mu \\
\leq &\int_{\Sigma}
\Vert f \Vert^{8}P^{m}_{5}(A)*\Phi \gamma^{s}d\mu
+cs^{4}\Lambda^{4}\int_{[\gamma>0]}\Vert f \Vert^{8}\gamma^{s-4}\vert \Phi \vert^{2}d\mu.
\end{split}
\end{gather*} 
For the second term we get by integration by parts and Young's inequality
\begin{gather*}
\begin{split}
-\frac{s}{2}\int _{\Sigma} \Vert f \Vert^{8} \vert &  \Phi \vert^{2} \gamma^{s-1} d \widetilde{\gamma}(f) (\Delta H) d\mu \\
\leq&
cs\Lambda \int_{\Sigma}\vert \nabla (\,\Vert f \Vert^{8}\vert \Phi \vert^{2} \gamma^{s-1}\,) \vert \vert \nabla A \vert d\mu
+
cs\Lambda\int_{\Sigma} \Vert f \Vert^{8}\vert \Phi \vert^{2} \gamma^{s-1}  \vert \nabla A \vert d\mu\\
& +
cs\Lambda\int_{\Sigma} \Vert f \Vert^{8}\vert \Phi \vert^{2} \gamma^{s-1}   \vert A \vert \vert \nabla A \vert d\mu\\
\leq&
cs\Lambda\int_{\Sigma} \Vert f \Vert^{7}\gamma^{s-1}\vert \nabla A \vert \vert \Phi \vert^{2}d\mu
+
cs\Lambda\int_{\Sigma}\Vert f\Vert^{8}\gamma^{s-1}\vert \nabla A \vert \vert \Phi \vert
\vert \nabla \Phi \vert d\mu\\
& +
cs^{2}\Lambda^{2}\int_{\Sigma}\Vert f\Vert^{8}\gamma^{s-2}\vert \nabla A \vert \vert \Phi\vert^{2}d\mu
+
cs\Lambda \int_{\Sigma} \Vert f \Vert^{8}\gamma^{s-1}\vert \nabla A \vert \vert \Phi \vert^{2}d\mu\\
& +
cs\Lambda\int_{\Sigma}\Vert f \Vert^{8}\gamma^{s-1}\vert A \vert \vert \nabla A \vert \vert \Phi \vert^{2} d\mu \\
\leq&
\int_{\Sigma}    \sum_{(i,j,k) \in I(m),\, j<m+4} \nabla^{i}\Vert f \Vert^{8}* P^{j}_{k}(A)*\Phi \gamma^{s} d\mu \\
& +
cs^{4}\Lambda^{4}\int_{[\gamma>0]} \Vert f \Vert^{8} \gamma^{s-4} \vert \Phi \vert^{2} d\mu \\
& + 
cs^{2}\Lambda^{2}\int_{\Sigma} \Vert f \Vert^{6} \gamma^{s-2} \vert \Phi \vert^{2} d\mu \\
&
+ 
cs^{2}\Lambda^{2}\int_{\Sigma } \Vert f \Vert^{8} \gamma^{s-2} \vert \Phi \vert^{2} d\mu \\
& +
cs^{2}\Lambda^{2}\int_{\Sigma} \Vert f \Vert^{8} \gamma^{s-2} \vert \nabla \Phi \vert^{2} d\mu ,
\end{split}
\end{gather*}
since for $\varphi:=\Vert f \Vert^{8}\vert \Phi \vert^{2} \gamma^{s-1} \geq 0$
\begin{gather*}
\begin{split}
0=&\int_{\Sigma}g^{i,j}\nabla_{i}(\varphi d\widetilde{\gamma}(f)(\nabla_{j}H))d\mu\\
=&
\int_{\Sigma}g^{i,j}\nabla_{i}\varphi\cdot d\widetilde{\gamma}(f)(\nabla_{j}H)d\mu
+\int_{\Sigma}g^{i,j}\varphi\nabla_{i}(d\widetilde{\gamma}(f)(\nabla_{j}H))d\mu\\
=&
\int_{\Sigma}g^{i,j}\nabla_{i}\varphi\cdot d\widetilde{\gamma}(f)(\nabla_{j}H)d\mu
+\int_{\Sigma}g^{i,j}\varphi\partial_{i}
(d\widetilde{\gamma}(f)(\nabla_{j}H))d\mu\\
&-\int_{\Sigma}g^{i,j}\varphi \Gamma^{k}_{i,j}
d\widetilde{\gamma}(f)(\nabla_{k}H)d\mu\\
=&
\int_{\Sigma}g^{i,j}\nabla_{i}\varphi\cdot d\widetilde{\gamma}(f)(\nabla_{j}H)d\mu
+\int_{\Sigma}g^{i,j}\varphi\, d^{2}\widetilde{\gamma}(f)(\partial_{i}f,\nabla_{j}H)d\mu\\
&-\int_{\Sigma}g^{i,j}\varphi 
d\widetilde{\gamma}(f)(\Gamma^{k}_{i,j}\nabla_{k}H)d\mu
+\int_{\Sigma}g^{i,j}\varphi \,
d\widetilde{\gamma}(f)(\partial_{i}\nabla_{j}H)d\mu,
\end{split}
 \end{gather*}
i.e. with $\nabla_{i}\nabla_{j}H=\partial_{i}\nabla_{j}H 
+ g^{n,m}\langle \nabla_{j}H,A_{i,n}\rangle \partial_{m}f-\Gamma^{k}_{i,j}\nabla_{k}H$ 
\begin{gather*}
\begin{split}
0 
=&
\int_{\Sigma}g^{i,j}\nabla_{i}\varphi\cdot d\widetilde{\gamma}(f)(\nabla_{j}H)d\mu
+\int_{\Sigma}g^{i,j}\varphi\, d^{2}\widetilde{\gamma}(f)(\partial_{i}f,\nabla_{j}H)d\mu\\
&+\int_{\Sigma}g^{i,j}\varphi \,
d\widetilde{\gamma}(f)(\nabla_{i}\nabla_{j}H)d\mu
-\int_{\Sigma}g^{i,j}\varphi \,
d\widetilde{\gamma}(f)(g^{n,m}\langle \nabla_{j} H ,A_{i,n}\rangle \partial_{m}f)d\mu,
\end{split}
\end{gather*}
and hence
\begin{gather*}
\begin{split}
-\int_{\Sigma}\varphi\,d\widetilde{\gamma}(f)(\Delta H)d\mu
\leq &
c \Lambda\int_{\Sigma}\vert \nabla \varphi \vert \vert \nabla A \vert d\mu
+
c\Lambda\int_{\Sigma} \varphi  \vert \nabla A \vert d\mu\\
& +
c\Lambda\int_{\Sigma}\varphi \vert \nabla A \vert \vert A \vert d\mu.
\end{split}
\end{gather*}
By proposition \ref{4.2}, $0 \leq \gamma \leq 1 \leq \Lambda $ and  Young's inequality it follows
\begin{gather*}
\begin{split}
\frac{d}{dt} \int_{\Sigma}  \vert \Phi  \vert & ^{2}   \gamma^{s} d\mu
+
\int_{\Sigma} \frac{\Vert f \Vert^{8}}{2}
\vert \nabla^{2} \Phi \vert^{2} \gamma^{s} d\mu \\
\leq &
\int_{\Sigma} \sum_{(i,j,k) \in I(m),\, j<m+4} 
\nabla^{i}\Vert f \Vert^{8}* P^{j}_{k}(A)*\Phi \gamma^{s} d\mu \\
& +
cs^{2}\Lambda^{2}
\hspace{-0.5pt}\int_{\Sigma} \Vert f \Vert^{8} \gamma^{s-2} 
\vert \nabla \Phi \vert^{2} d\mu
 +
c\int_{\Sigma}
[
s^{4}\Lambda^{4}\Vert f \Vert^{8} \gamma^{s-4}
+ 
\Vert f \Vert^{4} \gamma^{s}
]
\vert \Phi \vert^{2} d\mu.
\end{split}
\end{gather*}
By integration by parts one obtains
\begin{gather*}
\begin{split}
cs^{2}\Lambda^{2}&\int_{\Sigma}  \Vert f \Vert^{8}  \gamma^{s-2} \vert \nabla \Phi \vert^{2} d\mu \\
\leq &
\frac{1}{4}\int_{\Sigma} \Vert f \Vert^{8}\vert \nabla^{2} \Phi \vert^{2} \gamma^{s} d\mu 
+
c\int_{[\gamma>0]}
[\;
s^{4}\Lambda^{4}\Vert f \Vert^{8}\gamma^{s-4}
+
s^{2}\Lambda^{2}\Vert f \Vert^{6} \gamma^{s-2}
\;]
\vert \Phi \vert^{2}d\mu.
\end{split}
\end{gather*}
Plugging in and absorbing we conclude 
\begin{gather*}
\begin{split}
\frac{d}{dt} \int_{\Sigma} \vert \Phi \vert^{2} \gamma^{s} d\mu
+
\frac{1}{4} & \int_{\Sigma} 
\Vert f \Vert^{8}\vert \nabla^{2} \Phi \vert^{2} \gamma^{s} d\mu \\
\leq &
\int_{\Sigma} \sum_{(i,j,k) \in I(m),\, j<m+4} 
\nabla^{i}\Vert f \Vert^{8}* P^{j}_{k}(A)*\Phi \gamma^{s} d\mu \\
& +
c\int_{\Sigma} 
[\;
s^{4} \Lambda^{4} \Vert f \Vert^{8} \gamma^{s-4} 
+
s^{2}\Lambda^{2}\Vert f \Vert^{6} \gamma^{s-2} 
+
\Vert f \Vert^{4} \gamma^{s}
\;]
\vert \Phi \vert^{2} d\mu. 
\end{split}
\end{gather*}
The claim follows by Young's inequality.
\end{proof} 
\newpage
\section{Estimates on the right hand side}
In this section we estimate the second term on the right-hand side of (\ref{eq14}). This is unfortunately necessary, since
$\int_{[\gamma>0]}\Vert f \Vert^{8} \vert \nabla^{m} A \vert^{2} \gamma^{s-4}d\mu$
can not be controlled trivially, when using Gronwall's inequality.
\begin{proposition} \label{5.2}
For $\varepsilon,n,m,\Lambda > 0$ and $s \geq 2m+4$ there exists 
\begin{gather*}
c=c(\varepsilon,m,s,\Lambda)>0
\end{gather*}
such that, if \;
$f:\Sigma \rightarrow \mathbb{R}^{n}$ \;is a immersion and 
\; $\gamma=\widetilde{\gamma} \circ f$ \;as in (\ref{eq13}), \\
we have 
\begin{gather*}
\begin{split}
\int_{[\gamma>0]}
\Vert f \Vert^{8}
\vert \nabla^{m} & A \vert^{2}\gamma^{s-4}d\mu \\
\leq &
\varepsilon \int_{\Sigma} \Vert f \Vert^{8} \vert \nabla^{m+2} A \vert^{2} \gamma^{s} d\mu  
+
\varepsilon\sum^{m}_{i=0}\int_{\Sigma} \Vert f \Vert^{4} 
\vert \nabla^{i} A \vert^{2} \gamma^{s-2(m-i)} d\mu 
\end{split}
\end{gather*}
\begin{equation} \label{eq18} 
\hspace{-72pt}
+
c
\int_{[\gamma>0]}
\Vert f \Vert^{8} 
\vert A \vert^{2} \gamma^{s-2m-4}d\mu.
\end{equation}
\end{proposition}
\begin{proof}
The case $m=0$ is trivial. 
Let $m \geq 1$.
Applying corollary \ref{A7} with
$p=2,\,k=8,\;u=0,\;v=1$ we obtain 
\begin{gather*}
\begin{split}
\int_{\Sigma} \Vert f \Vert^{8}\vert \nabla \Phi \vert^{2} \gamma^{s-4}d\mu 
\leq &
\varepsilon \int_{\Sigma} \Vert f \Vert^{8} \vert \nabla^{2} \Phi \vert^{2} \gamma^{s-2}d\mu \\
& +
c(\varepsilon,s, \Lambda)\int_{[\gamma>0]}
(
\Vert f \Vert^{6}\vert \Phi \vert^{2} \gamma^{s-4}
+
\Vert f \Vert^{8} \vert \Phi \vert^{2} \gamma^{s-6}
)d\mu \\
\leq &
\varepsilon \int_{\Sigma} \Vert f \Vert^{8} \vert \nabla^{3} A \vert^{2} \gamma^{s}d\mu \\
& +
\varepsilon \int_{\Sigma} (\Vert f \Vert^{6}\vert \nabla \Phi \vert^{2}\gamma^{s-2}
+
\int_{\Sigma} \Vert f \Vert^{8} \vert \nabla \Phi \vert^{2} \gamma^{s-4}
)
d\mu \\
& +
c(\varepsilon,s, \Lambda)\int_{[\gamma>0]}
(
\Vert f \Vert^{6}\vert \Phi \vert^{2} \gamma^{s-4}
+
\Vert f \Vert^{8} \vert \Phi \vert^{2} \gamma^{s-6}
)d\mu.
\end{split}
\end{gather*}
By Young's inequality and absorption we derive
\begin{gather*}
\begin{split}
\int_{\Sigma} \Vert f \Vert^{8} \vert \nabla\Phi \vert^{2} \gamma^{s-4}d\mu 
\leq &
\varepsilon \int_{\Sigma} \Vert f \Vert^{8} \vert \nabla^{3} \Phi \vert^{2} \gamma^{s}d\mu \\
& +
\varepsilon \int_{\Sigma} 
[\;
\Vert f \Vert^{4} \vert \nabla \Phi \vert ^{2} \gamma^{s}
+
\Vert f \Vert^{4} \vert \Phi \vert^{2} \gamma^{s-2}
\;]
d\mu \\
& +
c(\varepsilon,s, \Lambda)\int_{[\gamma>0]}\Vert f \Vert^{8} \vert \Phi \vert^{2} \gamma^{s-6}d\mu.
\end{split}
\end{gather*}
For $m=1$ choose $\Phi = A.$ For $m \geq 2 $ we argue by induction via
\begin{gather*}
\begin{split}
\int_{\Sigma}\Vert f \Vert^{8} \vert \nabla^{m} A \vert^{2} \gamma^{s-4}d\mu
\leq &
\varepsilon \int_{\Sigma} \Vert f \Vert^{8} \vert \nabla^{m+2} A \vert^{2} \gamma^{s}d\mu \\
& +
\varepsilon \int_{\Sigma}
[\;
\Vert f \Vert^{4} \vert \nabla^{m} A \vert^{2} \gamma^{s} 
+
\Vert f \Vert^{4} \vert \nabla^{m-1} A \vert^{2} \gamma^{s-2}
\;]
d\mu \\
& +
c(\varepsilon,s,\Lambda)\int_{\Sigma} \Vert f \Vert^{8} \vert \nabla^{m-1} A \vert^{2} \gamma^{s-6}d\mu\\
\leq &
\varepsilon \int_{\Sigma} \hspace{-0.8pt} \Vert f \Vert^{8} \vert \nabla^{m+2} A \vert^{2} \gamma^{s}d\mu
+
\varepsilon \int_{\Sigma} \Vert f \Vert^{8} \vert \nabla^{m+1} A \vert^{2} \gamma^{s-2} d\mu \\
& +
\varepsilon \sum^{m}_{i=0}\int_{\Sigma} \Vert f \Vert^{4} \vert \nabla^{i} A \vert^{2} \gamma^{s-2(m-1)}d\mu\\
& +
c(\varepsilon,s,\Lambda)\int_{[\gamma>0]} \Vert f \Vert^{8} \vert A \vert^{2}\gamma^{s-2m-4}d\mu.
\end{split}
\end{gather*}
Applying  lemma \ref{A7} we derive
\begin{gather*}
\begin{split}
\int_{\Sigma} \Vert f \Vert^{8} \vert \nabla^{m+1} A \vert^{2} \gamma^{s-2}d\mu 
\leq &
\int_{\Sigma} \Vert f \Vert^{8} \vert \nabla^{m+2} A \vert^{2} \gamma^{s}d\mu \\
& +
c(s,\Lambda)\int_{\Sigma}
[\;
\Vert f \Vert^{6}\gamma^{s-2}
+
\Vert f \Vert^{8}\gamma^{s-4}
\;]
\vert \nabla^{m} A \vert^{2} d\mu \\
\leq &
\int_{\Sigma} \Vert f \Vert^{8} \vert \nabla^{m+2} A \vert^{2} \gamma^{s}d\mu \\
& +
c(s,\Lambda)\int_{\Sigma} 
[\;
\Vert f \Vert^{8}  \gamma^{s-4}  
+
\Vert f \Vert^{4}  \gamma^{s}
\;] \vert \nabla^{m} A \vert^{2}d\mu.
\end{split}
\end{gather*}
The claim follows.
\end{proof}
\section{Estimates by smallness assumption, m=0}
\begin{proposition}\label{5.1}
For $n,\Lambda>0$ and $s\geq 4$ there exist 
\begin{gather*}
c'=c'(n)\, ,\, c=c(n,s,\Lambda)>0 
\end{gather*}
such that, if \;
$f: \Sigma \times [0,T) \rightarrow \mathbb{R}^{n}\setminus \lbrace \mathbb{0} \rbrace$ \;is an inverse Willmore flow and \;$\gamma=\widetilde{\gamma} \circ f$\;  as in $(\ref{eq13})$, we have 
\begin{gather*}
\begin{split}
\frac{d}{dt} \int_{\Sigma} & \vert A \vert^{2} \gamma^{s} d\mu 
+
\int_{\Sigma}\frac{\Vert f \Vert^{8}}{8} \vert \nabla^{2} A \vert^{2} \gamma^{s} d\mu \\
& \hspace{-19pt}\leq 
c'\int_{\Sigma}\Vert f \Vert^{8} \vert  A \vert^{6} \gamma^{s}d\mu
+
c \int_{[\gamma>0]} [\;\Vert f \Vert^{8}\gamma^{s-4} 
+
\Vert f \Vert^{4}\gamma^{s}\;]\vert A \vert^{2} d\mu.
\end{split}
\end{gather*}
\end{proposition}
\begin{proof}
Recalling (\ref{I(m)=}), (\ref{eq13}) and observing
\begin{gather*}
\begin{split}
\int_{\Sigma}\nabla \Vert f \Vert^{8}*\nabla^{3}A*A\gamma^{s}d\mu
\leq \;&
c(n)\int_{\Sigma}\vert \nabla^{2}\Vert f \Vert^{8}\vert \vert \nabla^{2} A \vert \vert A \vert \gamma^{s}d\mu\\
&+
c(n)\int_{\Sigma}\vert \nabla \Vert f\Vert^{8}\vert \vert \nabla^{2}A\vert\vert \nabla A \vert \gamma^{s}d\mu \\
& +
c(n)s\Lambda\int_{\Sigma}\vert \nabla \Vert f\Vert^{8}\vert \vert \nabla^{2}A\vert\vert  A \vert \gamma^{s-1}d\mu
\end{split}
\end{gather*}
we obtain 
\begin{gather*}
\hspace{-68pt}
\begin{split}
\int_{\Sigma} \sum_{(i,j,k) \in I(0),\,j<4} &\nabla^{i} \Vert f \Vert^{8} *P^{j}_{k}(A) * A \gamma^{s}d\mu\\
= &
\int_{\Sigma}\Vert f \Vert^{8}
[
\nabla A * \nabla A * A * A 
+
\nabla^{2}A*A*A*A \\
& \hspace{44pt} +
A*A*A*A*A*A
]
\gamma^{s} \\
& \quad +
\nabla \Vert f \Vert^{8}*[\nabla^{3}A*A 
+
\nabla A*A*A*A] \gamma^{s} \\
& \quad +
\nabla^{2} \Vert f \Vert^{8}*
[
\nabla^{2}A*A 
+
A*A*A*A
]
\gamma^{s}\;d\mu \\
\leq & \;
c(n) \int_{\Sigma}
\vert \nabla^{2} A \vert 
[\;
\Vert f \Vert^{8} \gamma^{s} \vert A \vert^{3} 
+
\Vert f \Vert^{7} \gamma^{s} \vert \nabla A \vert  \\
& \hspace{54pt} +
s \Lambda\Vert f \Vert^{7} \gamma^{s-1} \vert  A \vert 
+
\Vert f \Vert^{7} \gamma^{s} \vert  A \vert^{2} 
+\Vert f \Vert^{6} \gamma^{s} \vert  A \vert 
\;]  \\
& \quad + 
\vert \nabla A \vert^{2} \Vert f \Vert^{8} \gamma^{s} 
\vert A \vert^{2} 
+ 
\vert \nabla A \vert \Vert f \Vert^{7} \gamma^{s} \vert A \vert^{3}
\end{split}
\end{gather*}
\begin{equation} \label{eq16}
\hspace{-87pt} 
+\Vert f \Vert^{8} \gamma^{s} \vert A \vert^{6} + \Vert f \Vert^{7} \gamma^{s}\vert A \vert^{5} 
+ \Vert f \Vert^{6} \gamma^{s} \vert A \vert^{4} \;d\mu.
\end{equation} 
For the first summand in (\ref{eq16}) Young's inequality yields
\begin{gather*}
\hspace{-30pt}
\begin{split}
\int_{\Sigma}\vert \nabla^{2} A \vert 
[ & \;
\Vert f \Vert^{8} \gamma^{s} \vert A \vert^{3} 
 +
\Vert f \Vert^{7} \gamma^{s} \vert \nabla A \vert \\
& +
s \Lambda\Vert f \Vert^{7} \gamma^{s-1} \vert  A \vert 
+
\Vert f \Vert^{7} \gamma^{s} \vert  A \vert^{2} 
+
\Vert f \Vert^{6} \gamma^{s} \vert  A \vert 
\;]\,d\mu \\
\leq &
\varepsilon \int_{\Sigma} \Vert f \Vert^{8}  \vert \nabla^{2} A \vert^{2} \gamma^{s} d\mu \\
& +
c(\varepsilon) \int_{\Sigma} 
\vert \nabla A \vert^{2} \Vert f \Vert^{6} \gamma^{s} d\mu \\
& +
c(\varepsilon) \int_{\Sigma} 
[\;
\Vert f \Vert^{8} \gamma^{s} \vert A \vert^{6} 
+
s^{2}\Lambda^{2}\Vert f \Vert^{6} \gamma^{s-2} \vert A \vert^{2} \\ 
& \hspace{122.5pt} + 
\Vert f \Vert^{6} \gamma^{s} \vert A \vert^{4} 
+ \Vert f \Vert^{4} \gamma^{s} \vert A \vert^{2}\,]\, d\mu.
\end{split}
\end{gather*}
By analogous estimate of the third summand, i.e.
\begin{gather*}
 \int_{\Sigma} \vert \nabla A \vert \Vert f \Vert^{7} \gamma^{s} \vert A \vert^{3}  d\mu
 \leq
 c \int_{\Sigma} \vert \nabla A \vert^{2} \Vert f \Vert^{6} \gamma^{s} d\mu
 +c \int_{\Sigma} \Vert f \Vert^{8} \gamma^{s} \vert A \vert^{6} d\mu,
\end{gather*}
and inserting in (\ref{eq16}) we obtain
\begin{gather*}
\begin{split}
\int_{\Sigma}\sum_{(i,j,k) \in I(0),\,j<4} 
\hspace{-13pt} & \hspace{13pt} 
\nabla^{i} \Vert f \Vert^{8}* P^{j}_{k}(A) * A \gamma^{s} d\mu\\
\leq &
\varepsilon \int_{\Sigma} \Vert f \Vert^{8}  
\vert \nabla^{2} A \vert^{2} \gamma^{s} d\mu \\
& +
c(n,\varepsilon)\int_{\Sigma} \vert \nabla A \vert^{2} 
[\; 
\Vert f \Vert^{6} \gamma^{s} 
+ 
\Vert f\Vert^{8} \gamma^{s} \vert A \vert^{2}
\;] 
d\mu \\
& +
c(n,\varepsilon) \int_{\Sigma} 
[\;
\Vert f \Vert^{8} \gamma^{s} \vert A \vert^{6} 
+
\Vert f \Vert^{7} \gamma^{s} \vert A \vert^{5}
+
s^{2}\Lambda^{2}\Vert f \Vert^{6} \gamma^{s-2} \vert A \vert^{2} \\
& \hspace{54pt}+
\Vert f \Vert^{6} \gamma^{s} \vert A \vert^{4} 
+
\Vert f \Vert^{4} \gamma^{s} \vert A \vert^{2} 
\;]
d\mu \\
\leq &
\varepsilon \int_{\Sigma} \Vert f \Vert^{8}  
\vert \nabla^{2} A \vert^{2} \gamma^{s} d\mu 
+
c(n,\varepsilon) \int_{\Sigma}  \Vert f \Vert^{8}  \gamma^{s} \vert A \vert^{6} d\mu \\
& +
c(n,\varepsilon ,s,\Lambda)\int_{[\gamma>0]} 
[\;
\Vert f \Vert^{8}\gamma^{s-4} 
+
\Vert f \Vert^{4} \gamma^{s}
\;]
\vert A \vert^{2} d\mu
\end{split}
\end{gather*}
\begin{equation} \label{eq17}
\hspace{-8pt} +
c(n,\varepsilon)\int_{\Sigma} \vert \nabla A \vert^{2} 
[\;
\Vert f \Vert^{6} \gamma^{s} 
+
\Vert f\Vert^{8} \gamma^{s} \vert A \vert^{2}
] 
d\mu .
\end{equation}
Next we come to estimate the fourth summand above. 
\begin{gather*}
\begin{split}
\int_{\Sigma} \vert \nabla A \vert^{2} &
[\;
\Vert f \Vert^{6} \gamma^{s} + \Vert f\Vert^{8} \gamma^{s} \vert A \vert^{2}] d\mu \\
= & 
-\int_{\Sigma} \langle  \nabla A , A \rangle *
\nabla  
[\; 
\Vert f \Vert^{6} \gamma^{s} 
+
\Vert f\Vert^{8} \gamma^{s} \vert A \vert^{2}] d\mu \\
& -
\int_{\Sigma} \langle  \Delta A , A \rangle   
[\; 
\Vert f \Vert^{6} \gamma^{s} 
+ 
\Vert f\Vert^{8} \gamma^{s} \vert A \vert^{2}
] 
d\mu \\
\leq &
c \int_{\Sigma} \vert \nabla A \vert \vert A \vert 
[\;
\Vert f \Vert^{5} \gamma^{s} 
+
s \Lambda \Vert f \Vert^{6} \gamma^{s-1} \\
& \hspace{105pt}+
\Vert f \Vert^{7} \gamma^{s} \vert A \vert^{2}
+ 
s \Lambda\Vert f \Vert^{8} \gamma^{s-1}\vert A \vert^{2}]  d\mu \\
& -
2\int_{\Sigma} \Vert f \Vert^{8} \gamma^{s} \langle \nabla A , A \rangle^{2} d\mu \\
& +
\int_{\Sigma} \vert \nabla^{2} A \vert \vert A \vert 
[\;
\Vert f \Vert^{6} \gamma^{s} 
+
\Vert f\Vert^{8} \gamma^{s} \vert A \vert^{2}
]
d\mu 
\\
\leq &
\varepsilon \int_{\Sigma} \Vert f \Vert^{8} 
\vert \nabla^{2} A \vert^{2} \gamma^{s} d\mu 
+ 
\varepsilon \int_{\Sigma} \vert \nabla A \vert^{2}
[\;
\Vert f \Vert^{6} \gamma^{s} 
+
\Vert f\Vert^{8} \gamma^{s} \vert A \vert^{2}
] 
d\mu \\
& +
c(\varepsilon) \int_{\Sigma} 
[\;
\Vert f \Vert^{4} \gamma^{s} \vert A \vert^{2} 
+
s^{2}\Lambda^{2}\Vert f \Vert^{6} \gamma^{s-2} \vert A \vert^{2} \\
& \hspace{50pt} +
\Vert f \Vert^{8} \gamma^{s} \vert A \vert^{6} 
+
s^{2}\Lambda^{2}\Vert f \Vert^{8} \gamma^{s-2} \vert A \vert^{4}
]
d\mu \\
\leq &
\varepsilon \int_{\Sigma} \Vert f \Vert^{8} 
\vert \nabla^{2} A \vert^{2} \gamma^{s} d\mu 
+ 
\varepsilon \int_{\Sigma} \vert \nabla A \vert^{2}
[\;
\Vert f \Vert^{6} \gamma^{s} 
+ 
\Vert f\Vert^{8} \gamma^{s} \vert A \vert^{2}
] 
d\mu \\
& +
c(\varepsilon) 
\hspace{-2pt}
\int_{\Sigma} \Vert f \Vert^{8}  \vert A \vert^{6} \gamma^{s}d\mu \\
& +
c(\varepsilon,s,\Lambda) 
\int_{[\gamma>0]} 
[\,
\Vert f \Vert^{8}\gamma^{s-4} 
+
\Vert f \Vert^{4}\gamma^{s}
]
\vert A \vert^{2} d\mu.
\end{split}
\end{gather*}
Absorbing and plugging into (\ref{eq17}) we conclude
\begin{gather*}
\begin{split}
\int_{\Sigma}\sum_{(i,j,k) \in I(0),\,j<4} &\nabla^{i} \Vert f \Vert^{8}* P^{j}_{k}(A) * A \gamma^{s} d\mu\\
\leq &
\varepsilon \int_{\Sigma} \Vert f \Vert^{8}  \vert \nabla^{2} A \vert^{2} \gamma^{s} d\mu
+
c(n,\varepsilon) \int_{\Sigma}  \Vert f \Vert^{8}  \gamma^{s} \vert A \vert^{6} d\mu\\ & +
c(n,\varepsilon,s,\Lambda) \int_{[\gamma>0]} 
[\,
\Vert f \Vert^{8}\gamma^{s-4} +\Vert f \Vert^{4}\gamma^{s}
]
\vert A \vert^{2} d\mu.
\end{split}
\end{gather*}
Applying this inequality to proposition \ref{4.3}, (\ref{eq14})  proves the claim.
\end{proof}

\begin{proposition} \label{6.1}
For $n,\Lambda>0$ and $s \geq 4$ there exist 
\begin{gather*}
\varepsilon_{0}=\varepsilon_{0}(n),\;c_{0}=c_{0}(n,s,\Lambda)>0 
\end{gather*}
such that, if
\quad $f: \Sigma \times [0,T) \rightarrow \mathbb{R}^{n}\setminus \lbrace \mathbb{0} \rbrace$ \quad
is an inverse Willmore flow, \;$\gamma=\widetilde{\gamma} \circ f$\; as in $(\ref{eq13})$ and 
\begin{gather*}
\sup_{0\leq t < T}\int_{[\gamma>0]}\vert A \vert^{2}d\mu\leq \varepsilon_{0},
\end{gather*}
we have 
\begin{gather*}
\begin{split}
\int_{\Sigma}\vert A \vert^{2} \gamma^{s}d\mu
& + 
\int^{t}_{0} 
\int_{\Sigma}
\frac{\Vert f \Vert^{8}}{16}\vert \nabla^{2} A \vert^{2}\gamma^{s}d\mu \;dt 
\leq 
\int_{\Sigma}\vert A \vert^{2} \gamma^{s} d\mu\lfloor_{t=0} \\
& + 
c_{0}
\sup_{0 \leq t <T}
(
\Vert f \Vert^{8}_{L^{\infty}_{\mu}([\gamma>0])}
+
\Vert f \Vert^{4}_{L^{\infty}_{\mu}([\gamma>0])}
)
\cdot t .
\end{split}
\end{gather*}
\end{proposition}
\begin{proof}
We estimate by lemma \ref{A4} 
\begin{gather*}
\hspace{-66pt}
\begin{split}
\int_{\Sigma}  \Vert  f \Vert^{8} &\vert A \vert^{6}  \gamma^{s} d\mu
= 
\int_{\Sigma}(\Vert f \Vert^{4}\vert A \vert^{3}\gamma^{\frac{s}{2}})^{2}d\mu\\
\leq &
c[\int_{\Sigma}
\Vert f \Vert^{3}\vert A \vert^{3}\gamma^{\frac{s}{2}}d\mu
+
\int_{\Sigma}\Vert f \Vert^{4} \vert A \vert^{2}\vert \nabla A \vert \gamma^{\frac{s}{2}} d\mu\\
& \quad \quad +
s \Lambda \int_{\Sigma}\Vert f \Vert^{4}\vert A \vert^{3}\gamma^{\frac{s-2}{2}}d\mu
+
\int_{\Sigma}\Vert f \Vert^{4}\vert A \vert^{4}\gamma^{\frac{s}{2}}\;d\mu
\;]^{2}\\
\leq &
c\int_{[\gamma>0]}\vert A \vert^{2}d\mu
[\int_{\Sigma}\Vert f \Vert^{6}\vert A \vert^{4}\gamma^{s} d\mu
+
\int_{\Sigma} \Vert f \Vert^{8} \vert A \vert^{2}\vert \nabla A \vert^{2} \gamma^{s} d\mu\\
\end{split}
\end{gather*}
\begin{equation} \label{eq19}
\hspace{100pt}
 +
s^{2} \Lambda^{2}\int_{\Sigma}\Vert f \Vert^{8}\vert A \vert^{4}\gamma^{s-2} d\mu
+
\int_{\Sigma} \Vert f \Vert^{8}\vert A \vert^{6}\gamma^{s}d\mu\,].
\end{equation}
By integration by parts we get using $\Vert f \Vert^{8}\langle A,\nabla A \rangle^{2}\gamma^{s}\geq 0,$ 

\begin{gather*}
\begin{split}
\int_{\Sigma}\Vert f \Vert^{8}\vert A \vert^{2}\vert \nabla A \vert^{2}\gamma^{s}d\mu 
\leq &
c\int_{\Sigma} \Vert f \Vert^{7}\vert A \vert^{3}\vert \nabla A \vert \gamma^{s}d\mu
+
c\int_{\Sigma}\Vert f \Vert^{8} \vert A \vert^{3}\vert \nabla^{2} A \vert \gamma^{s}d\mu \\
& +
cs\Lambda\int_{\Sigma}\Vert f \Vert^{8}\vert A \vert^{3}\vert \nabla A \vert \gamma^{s-1}d\mu \\
\leq &
\varepsilon \int_{\Sigma}\Vert f \Vert^{8}\vert A \vert^{2}\vert \nabla A \vert^{2}\gamma^{s}d\mu
+
c(\varepsilon) \int_{\Sigma}\Vert f \Vert^{6}\vert A \vert^{4}\gamma^{s}d\mu \\
& +
c\int_{\Sigma}\Vert f \Vert^{8} \vert \nabla^{2} A \vert^{2} \gamma^{s}d\mu
+
c\int_{\Sigma} \Vert f \Vert^{8} \vert A \vert^{6}\gamma^{s}d\mu \\
& +
c(\varepsilon)s^{2}\Lambda^{2}\int_{\Sigma}\Vert f \Vert^{8} \vert A \vert^{4}\gamma^{s-2}d\mu,
\end{split}
\end{gather*}
i.e. by absorption and Young's inequality
\begin{gather*}
\begin{split}
\int_{\Sigma}\Vert f \Vert^{8}\vert A \vert^{2}\vert \nabla A \vert^{2}\gamma^{s}d\mu 
\leq &
c \int_{\Sigma} \Vert f \Vert^{8} \vert \nabla^{2} A \vert^{2} \gamma^{s} d\mu
+
c \int_{\Sigma} \Vert f \Vert^{8} \vert A \vert^{6} \gamma^{s} d\mu\\
& +
c(s,\Lambda)
\int_{[\gamma>0]}
[\;\Vert f \Vert^{8} \gamma^{s-4}
+
\Vert f \Vert^{4}\gamma^{s}
\;]\vert A \vert^{2}d\mu.
\end{split}
\end{gather*}
For the remaining summands in (\ref{eq19}) we have
\begin{gather*}
\begin{split}
\int_{\Sigma}\Vert f \Vert^{6} \vert A \vert^{4} & \gamma^{s} d\mu 
 +
s^{2} \Lambda^{2}\int_{\Sigma} \Vert f \Vert^{8} \vert A \vert^{4} \gamma^{s-2}d\mu 
+
\int_{\Sigma} \Vert f \Vert^{8} \vert A \vert^{6} \gamma^{s} d\mu\\
\leq &
c \int_{\Sigma} \Vert f \Vert^{8} \vert A \vert^{6}\gamma^{s}d\mu 
+
c(s,\Lambda)
\int_{[\gamma>0]}
[\;\Vert f \Vert^{8} \gamma^{s-4}
+
\Vert f \Vert^{4}\gamma^{s}
\;]\vert A \vert^{2}d\mu.
\end{split}
\end{gather*}
Applying these inequalities to (\ref{eq19}) we derive
\begin{gather*}
\begin{split}
\int_{\Sigma}\Vert f \Vert^{8}  \vert A \vert^{6}\gamma^{s}d\mu 
\leq &
c\int_{[\gamma>0]}\vert A \vert^{2}d\mu\\
&
[\,
\int_{\Sigma}\Vert f \Vert^{8} \vert \nabla^{2} A \vert^{2} \gamma^{s}d\mu
+
\int_{\Sigma} \Vert f \Vert^{8} \vert A \vert^{6}\gamma^{s}d\mu \\
& \hspace{2pt}+
c(s,\Lambda)
\int_{[\gamma>0]}
[\;\Vert f \Vert^{8} \gamma^{s-4}
+
\Vert f \Vert^{4}\gamma^{s}
\;]\vert A \vert^{2}d\mu
\,].
\end{split}
\end{gather*}
Hence 
\begin{gather*}
\hspace{-168pt}
\begin{split}
\int_{\Sigma}\Vert f \Vert^{8}  \vert A \vert^{6}\gamma^{s}d\mu 
\leq &
\varepsilon \int_{\Sigma}\Vert f \Vert^{8} \vert \nabla^{2} A \vert^{2} \gamma^{s}d\mu
\end{split}
\end{gather*}
\begin{equation} \label{eq20}
\hspace{78pt}
+
c(\varepsilon,s,\Lambda)
\int_{[\gamma>0]}
[\;\Vert f \Vert^{8} \gamma^{s-4}
+
\Vert f \Vert^{4}\gamma^{s}
\;]
\vert A \vert^{2}d\mu,
\end{equation}
provided, that $\int_{[\gamma>0]}\vert A \vert^{2}d\mu \leq \delta(\varepsilon)$ is small enough. \\ 
Applying this inequality to proposition \ref{5.1} with $\varepsilon=\frac{1}{16c'}$ we conclude
\begin{gather*}
\begin{split}
\frac{d}{dt} \int_{\Sigma} \vert A \vert^{2} \gamma^{s} d\mu 
& +
\int_{\Sigma}\frac{\Vert f \Vert^{8}}{16} \vert \nabla^{2} A \vert^{2} \gamma^{s} d\mu\\
\leq &
c(n,s,\Lambda)
\int_{[\gamma>0]}
[\;\Vert f \Vert^{8} \gamma^{s-4}
+
\Vert f \Vert^{4}\gamma^{s}
\;]\vert A \vert^{2}d\mu.
\end{split}
\end{gather*}
The proposition follows by integration.
\end{proof}

\section{Estimates by smallness assumption, m=1}
\begin{proposition} \label{7.1}
For $n,\Lambda,R,d,\tau >0$ and $s\geq 6$ there exist 
\begin{gather*}
\varepsilon_{1}=\varepsilon_{1}(n),\;c_{1}=c_{1}(n,s,\Lambda,R,d,\tau )>0,
\end{gather*}
such that, if \quad
$
f: \Sigma \times [0,T) \rightarrow \mathbb{R}^{n}\setminus \lbrace \mathbb{0} \rbrace,\quad 0<T\leq \tau,
$ \\
is an inverse Willmore flow,
$\gamma=\widetilde{\gamma} \circ f$ as in $(\ref{eq13})$ and
\begin{gather*}
\begin{split}
&\sup_{0\leq t < T}\int_{[\gamma>0]} \vert A \vert^{2}d\mu \leq \varepsilon_{1},\\
&
\sup_{0 \leq t <T}\Vert f \Vert_{L^{\infty}_{\mu}([\gamma>0])} \leq R,\\
&
\int^{T}_{0}\int_{[\gamma>0]} \Vert f \Vert^{8} \vert \nabla^{2} A \vert^{2} d\mu\,dt 
\leq 
d, \\
&
\int^{T}_{0}\Vert \; \Vert f \Vert^{4} A \Vert^{4}_{L^{\infty}_{\mu}([\gamma>0])}dt 
\leq
d,
\end{split}
\end{gather*}
we have 
\begin{gather*}
\begin{split}
\sup_{0 \leq t < T}\int_{\Sigma}\vert \nabla A \vert^{2} \gamma^{s}d\mu
+ 
\int^{T}_{0} \hspace{-4pt}
\int_{\Sigma}
\Vert f \Vert^{8}\vert \nabla^{3} A \vert^{2}\gamma^{s}d\mu \;dt 
\leq 
c_{1}(1+\int_{\Sigma} \vert \nabla A \vert^{2} \gamma^{s} d\mu \lfloor_{t=0}).
\end{split}
\end{gather*} 
\end{proposition}
\begin{proof}
For abbreviative reasons let
$c_{1}=c_{1}(n,s,\Lambda,R)$.
We will show 
\begin{gather*} 
\begin{split}
\sum_{(i,j,k) \in I(1),\,j<5}  \int_{\Sigma}  \nabla^{i}\Vert f \Vert^{8} & *P^{j}_{k}(A)*\nabla A \,\gamma^{s}d\mu \\ 
\leq & 
\int_{\Sigma} \frac{\Vert f \Vert^{8}}{16} \vert \nabla^{3} A \vert^{2}d\mu \\
& +
c_{1}
(
1
+
\int_{[\gamma>0]}\Vert f \Vert^{8}\vert \nabla^{2} A \vert^{2} d\mu
+
\Vert\,\Vert f \Vert^{4} A \Vert^{4}_{L^{\infty}_{\mu}([\gamma>0])}
)
\end{split}
\end{gather*}
\begin{equation} \label{m=1,0}
\hspace{10pt}
(1+\int_{\Sigma} \vert \nabla A \vert^{2}\gamma^{s} d\mu).
\end{equation}
Applying this inequality and proposition \ref{5.2} with $\varepsilon=\frac{1}{8c}$ to proposition \ref{4.3} we obtain
\begin{gather*}
\begin{split}
\frac{d}{dt} \int_{\Sigma} \vert & \nabla A  \vert^{2}  \gamma^{s}d\mu 
+
\int_{\Sigma}\frac{\Vert f \Vert^{8}}{16} \vert \nabla^{3} A \vert^{2} \gamma^{s}d\mu \\
\leq  &
c_{1}
(
1
+
\int_{[\gamma>0]}\Vert f \Vert^{8}\vert \nabla^{2} A \vert^{2} d\mu
+
\Vert\,\Vert f \Vert^{4} A \Vert^{4}_{L^{\infty}_{\mu}([\gamma>0])}
)
(1+\int_{\Sigma} \vert \nabla A \vert^{2}\gamma^{s} d\mu),
\end{split}
\end{gather*}
what proves the proposition using Gronwall's inequality, cf. lemma \ref{Gronwall's inequality}.\newpage
\hspace{-22pt}
For the reader's convenience we perform this argument.
First we have
\begin{gather*}
\begin{split}
\frac{d}{dt}(1 + & \int_{\Sigma}\vert \nabla A \vert^{2}\gamma^{s}d\mu) \\
\leq &
c_{1}
(
1
+
\int_{[\gamma>0]}\Vert f \Vert^{8}\vert \nabla^{2} A \vert^{2} d\mu
+
\Vert\,\Vert f \Vert^{4} A \Vert^{4}_{L^{\infty}_{\mu}([\gamma>0])}
)
(1+\int_{\Sigma} \vert \nabla A \vert^{2}\gamma^{s} d\mu).
\end{split}
\end{gather*}
Integrating and applying  Gronwall's inequality we obtain
\begin{gather*}
\begin{split}
1 +  \int_{\Sigma}\vert & \nabla  A \vert^{2}\gamma^{s}d\mu \\
& \leq 
(1 + \int_{\Sigma}\vert \nabla A \vert^{2}\gamma^{s}d\mu \lfloor_{t=0}) 
e^
{
c_{1}
\int^{t}_{0}
(
1
+ \hspace{-1pt}
\int_{[\gamma>0]}\Vert f \Vert^{8}\vert \nabla^{2} A \vert^{2} d\mu
+
\Vert\,\Vert f \Vert^{4} A \Vert^{4}_{L^{\infty}_{\mu}([\gamma>0])}
)
dt
}
\\
& \leq 
(1 + \int_{\Sigma}\vert \nabla A \vert^{2}\gamma^{s}d\mu \lfloor_{t=0}) 
e^
{
c_{1}(\tau+2d)
}.
\end{split}
\end{gather*}
It follows
\begin{gather*}
\sup_{0 \leq t < T}  \int_{\Sigma}\vert  \nabla  A \vert^{2}\gamma^{s}d\mu 
\leq 
c(n,s,\Lambda,R,d,\tau)
(1 + \int_{\Sigma}\vert \nabla A \vert^{2}\gamma^{s}d\mu \lfloor_{t=0}).
\end{gather*}
Secondly we derive by integration and the already proven estimate
\begin{gather*}
\begin{split}
\int^{T}_{0}\int_{\Sigma}&\frac{\Vert f \Vert^{8}}{16} \vert \nabla^{3}  A \vert^{2}  \gamma^{s}d\mu \,dt
\leq 
\int_{\Sigma}\vert \nabla A \vert^{2}\gamma^{s}d\mu \lfloor_{t=0} \\
& +
c_{1}(1+\sup_{0 \leq t < T}  \int_{\Sigma}\vert \nabla  A \vert^{2}\gamma^{s}d\mu) \\
& \hspace{10pt}\int^{T}_{0}
(
1
+
\int_{[\gamma>0]}\Vert f \Vert^{8}\vert \nabla^{2} A \vert^{2} d\mu
+
\Vert\,\Vert f \Vert^{4} A \Vert^{4}_{L^{\infty}_{\mu}([\gamma>0])}
)
\,dt\\
\leq &
\int_{\Sigma}\vert \nabla A \vert^{2}\gamma^{s}d\mu \lfloor_{t=0}
+
c(n,s,\Lambda,R)
c(n,s,\Lambda,R,d,\tau)\\
& \hspace{106pt}
(1 + \int_{\Sigma}\vert \nabla A \vert^{2}\gamma^{s}d\mu \lfloor_{t=0})(\tau+2d) \\
\leq &
c(n,s,\Lambda,R,d,\tau)
(1 + \int_{\Sigma}\vert \nabla A \vert^{2}\gamma^{s}d\mu \lfloor_{t=0}).
\end{split}
\end{gather*}
This establishes the required estimate, provided, that (\ref{m=1,0})
holds true.

To show (\ref{m=1,0}) we give adequate estimates for each term of the sum on

the left hand side in (\ref{m=1,0}).
\begin{itemize}
\item 
$k=5,j=1,i=0$
\end{itemize}
For this case we start estimating by lemma \ref{A4} and H\"older's inequality
\begin{gather*} 
\hspace{-26pt}
\begin{split}
\int_{\Sigma} \Vert f&  \Vert^{8}  
\vert A \vert^{4} \vert \nabla A \vert^{2} \gamma^{s} d\mu
=  \int_{\Sigma} (\Vert f \Vert^{4} \vert A \vert^{2} \vert \nabla A \vert \gamma^{\frac{s}{2}})^{2}d\mu \\
\leq &
c
[
\int_{\Sigma} \Vert f \Vert^{3} \vert A \vert^{2} \vert \nabla A \vert \gamma^{\frac{s}{2}}d\mu
+
\int_{\Sigma} \Vert f \Vert^{4} \vert A \vert \vert \nabla A \vert^{2} \gamma^{\frac{s}{2}} d\mu \\
& \hspace{22pt}+
\int_{\Sigma} \Vert f \Vert^{4} \vert A \vert^{2} \vert \nabla^{2} A \vert \gamma^{\frac{s}{2}}d\mu
+
s \Lambda \int_{\Sigma} \Vert f \Vert^{4} \vert A \vert^{2} \vert \nabla A \vert \gamma^{\frac{s-2}{2}}d\mu \\
& \hspace{22pt}+
\int_{\Sigma} \Vert f \Vert^{4} \vert A \vert^{3} \vert \nabla A \vert \gamma^{\frac{s}{2}}d\mu\;
]^{2} \\
\leq &
c\int_{[\gamma>0]} \vert A \vert^{2}d\mu\\
&[
\int_{\Sigma} \Vert f \Vert^{6} \vert A \vert^{2} \vert \nabla A \vert^{2} \gamma^{s}d\mu
+
\int_{\Sigma}\Vert f \Vert^{8} \vert \nabla A \vert^{4} \gamma^{s}d\mu \\
& +
\int_{\Sigma} \Vert f \Vert^{8} \vert A \vert^{2} \vert \nabla^{2} A \vert^{2} \gamma^{s} d\mu
+
s^{2}\Lambda^{2}\int_{\Sigma}\Vert f \Vert^{8} \vert A \vert^{2} \vert \nabla A \vert^{2} \gamma^{s-2}d\mu \\
& +
\int_{\Sigma} \Vert f \Vert^{8} 
\vert A \vert^{4} \vert \nabla A \vert^{2} \gamma^{s} d\mu\;
] \\
\leq &
c\int_{[\gamma>0]} \vert A \vert^{2}d\mu\\
&[
\int_{\Sigma} \Vert f \Vert^{8} \vert A \vert^{4} \vert \nabla A \vert^{2} \gamma^{s} d\mu
+
\int_{\Sigma} \Vert f \Vert^{4} \vert \nabla A \vert^{2} \gamma^{s}d\mu 
+
\int_{\Sigma} \Vert f \Vert^{8} \vert \nabla A \vert^{4}\gamma^{s}d\mu
\\
& +
s^{4}\Lambda^{4}\int_{[\gamma>0]}\Vert f \Vert^{8} \vert A \vert^{4} d\mu
+
\int_{\Sigma} \Vert f \Vert^{8} \vert A \vert^{2} \vert \nabla^{2} A \vert^{2} \gamma^{s} d\mu\;
].
\end{split}
\end{gather*}
By integration by parts 
\begin{gather*}
\begin{split}
\int_{\Sigma} \Vert f \Vert^{8} & \vert \nabla A \vert^{4} \gamma^{s}d\mu \\
\leq &
c\int_{\Sigma} \Vert f \Vert^{7} \vert A \vert \vert \nabla A \vert^{3} \gamma^{s}d\mu 
+
c \int_{\Sigma} \Vert f \Vert^{8} \vert A \vert \vert \nabla A \vert^{2} \vert \nabla^{2} A \vert \gamma^{s}d\mu \\
& +
cs\Lambda \int_{\Sigma} \Vert f \Vert^{8} \vert A \vert \vert \nabla A \vert^{3} \gamma^{s-1} d\mu 
\end{split}
\end{gather*}
\begin{gather*}
\begin{split}
\leq &
\varepsilon \int_{\Sigma} \Vert f \Vert^{8} \vert \nabla A \vert^{4} \gamma^{s}d\mu 
+ 
c(\varepsilon) \int_{\Sigma} \Vert f \Vert^{6} \vert A \vert^{2}  \vert \nabla A \vert^{2} \gamma^{s}d\mu \\
& +
c(\varepsilon) \int_{\Sigma} \Vert f \Vert^{8} \vert A \vert^{2}\vert \nabla^{2} A \vert^{2} \gamma^{s} d\mu 
+
c(\varepsilon)s^{4}\Lambda^{4}\int_{[\gamma>0]} \Vert f \Vert^{8} \vert A \vert^{4}  d\mu,
\end{split}
\end{gather*}
i.e. by absorption and Young's inequality 
\begin{gather*}
\hspace{-86pt}
\begin{split}
\int_{\Sigma} \Vert f \Vert^{8} & \vert \nabla A \vert^{4} \gamma^{s}d\mu \\
\leq &
c \int_{\Sigma} \Vert f \Vert^{8} \vert A \vert^{4} \vert \nabla A \vert^{2} \gamma^{s}d\mu
+
c\int_{\Sigma} \Vert f \Vert^{4} \vert \nabla A \vert^{2} \gamma^{s}d\mu 
\end{split}
\end{gather*}
\begin{equation} \label{m=1,1}
\hspace{-16pt}
 +
cs^{4}\Lambda^{4}\int_{[\gamma>0]}\Vert f \Vert^{8} \vert A \vert^{4} d\mu
+
c\int_{\Sigma} \Vert f \Vert^{8} \vert A \vert^{2} \vert \nabla^{2} A \vert^{2} \gamma^{s}d\mu.
\end{equation}
Inserting this inequality yields
\begin{gather*}
\hspace{-22pt}
\begin{split}
\int_{\Sigma}  \Vert f & \Vert^{8}  
\vert A \vert^{4} \vert \nabla A \vert^{2}  \gamma^{s} d\mu \\
\leq & 
c\int_{[\gamma>0]} \vert A \vert^{2}d\mu 
[
\int_{\Sigma} \Vert f \Vert^{8} \vert A \vert^{4} \vert \nabla A \vert^{2} \gamma^{s} d\mu
+
s^{4}\Lambda^{4}\int_{[\gamma>0]}\Vert f \Vert^{8} \vert A \vert^{4} d\mu
\end{split}
\end{gather*}
\begin{equation} \label{m=1,2}
\hspace{112pt}
+
\int_{\Sigma} \Vert f \Vert^{4} \vert \nabla A \vert^{2} \gamma^{s}d\mu 
+
\int_{\Sigma} \Vert f \Vert^{8} \vert A \vert^{2} \vert \nabla^{2} A \vert^{2} \gamma^{s} d\mu\;
].
\end{equation}
Next we have
\begin{gather*}
\begin{split}
\int_{\Sigma}\Vert f \Vert^{8}& \vert A \vert^{2} \vert \nabla^{2} A \vert^{2} \gamma^{s} d\mu 
=
\int_{\Sigma}(\Vert f \Vert^{4} \vert A \vert \vert \nabla^{2} A \vert \gamma^{\frac{s}{2}})^{2}d\mu \\
\leq &
c
[
\int_{\Sigma} \Vert f \Vert^{3} \vert A \vert \vert \nabla^{2} A \vert \gamma^{\frac{s}{2}}d\mu
+
\int_{\Sigma} \Vert f \Vert^{4} \vert \nabla A \vert \vert \nabla^{2}A \vert \gamma^{\frac{s}{2}}d\mu \\
& \hspace{22pt}+
\int_{\Sigma}\Vert f \Vert^{4} \vert A \vert \vert \nabla^{3} A \vert \gamma^{\frac{s}{2}}d\mu
+
s\Lambda\int_{\Sigma} \Vert f \Vert^{4} \vert A \vert \vert \nabla^{2} A \vert \gamma^{\frac{s-2}{2}}d\mu \\
& \hspace{22pt}+
\int_{\Sigma} \Vert f \Vert^{4} \vert A \vert^{2} \vert \nabla^{2} A \vert \gamma^{\frac{s}{2}}d\mu \;
]^{2} \\
\leq &
c\int_{[\gamma>0]}\vert A \vert^{2}d\mu 
[
\int_{\Sigma} \Vert f \Vert^{8} \vert \nabla^{3} A \vert^{2} \gamma^{s}d\mu
+
s^{2}\Lambda^{2}\int_{[\gamma>0]}\Vert f \Vert^{8} \vert \nabla^{2} A \vert^{2} d\mu \\
& \hspace{90pt}+
\int_{\Sigma}\Vert f \Vert^{6} \vert \nabla^{2} A \vert^{2} \gamma^{s}d\mu
+
\int_{\Sigma} \Vert f \Vert^{8} \vert A \vert^{2} \vert \nabla^{2} A \vert^{2} \gamma^{s} d\mu \;
]  
\end{split}
\end{gather*} 
\begin{equation} \label{m=1,3}
\hspace{-100pt}
+
c\int_{[\gamma>0]}\Vert f \Vert^{8} \vert \nabla^{2} A \vert^{2} d\mu \int_{\Sigma} \vert \nabla A \vert^{2} \gamma^{s}d\mu.
\end{equation}
\newpage
\hspace{-22pt}
By integration by parts we have
\begin{gather*}
\hspace{-10pt}
\begin{split}
\int_{\Sigma} 
\Vert f \Vert^{6} \vert \nabla^{2} A \vert^{2} \gamma^{s}d\mu 
\leq &
c\int_{\Sigma} \Vert f \Vert^{5} \vert \nabla A \vert \vert \nabla^{2} A \vert \gamma^{s} d\mu
+
c\int_{\Sigma} \Vert f \Vert^{6} \vert \nabla A \vert \vert \nabla^{3} A \vert \gamma^{s}d\mu \\
& +
cs \Lambda \int_{\Sigma} \Vert f \Vert^{6} \vert \nabla A \vert \vert \nabla^{2} A \vert \gamma^{s-1}d\mu \\
\leq &
\varepsilon \int_{\Sigma} \Vert f \Vert^{6} \vert \nabla^{2} A \vert^{2} \gamma^{s}d\mu 
+
cs^{2}\Lambda^{2}\int_{[\gamma>0]} \Vert f \Vert^{8} \vert \nabla^{2} A \vert^{2}d\mu\\
& +
c\int_{\Sigma} \Vert f \Vert^{8} \vert \nabla^{3} A \vert^{2} \gamma^{s} d\mu
\end{split}
\end{gather*}
\begin{equation} \label{m=1,4}
\hspace{-29pt}
+
c(\varepsilon) \int_{\Sigma} \Vert f \Vert^{4} \vert \nabla A \vert^{2} \gamma^{s} d\mu
.
\end{equation}
Absorbing and inserting (\ref{m=1,4}) in (\ref{m=1,3}) we obtain, since $\int_{[\gamma>0]}\vert A \vert^{2} d\mu \leq \varepsilon_{1},$
\begin{gather*}
\hspace{-30pt}
\begin{split}
\int_{\Sigma} \Vert f \Vert^{8}  \vert A \vert^{2} \vert \nabla^{2} A \vert^{2} \gamma^{s} d\mu 
\leq &
c\int_{[\gamma>0]}\vert A \vert^{2}d\mu\int_{\Sigma} \Vert f \Vert^{8} \vert \nabla^{3} A \vert^{2} \gamma^{s} d\mu \\
& +
c\int_{\Sigma} \Vert f \Vert^{4} \vert \nabla A \vert^{2} \gamma^{s} d\mu \\
& +
c\int_{[\gamma>0]} \Vert f \Vert^{8} \vert \nabla^{2} A \vert^{2} d\mu 
\int_{\Sigma} \vert \nabla A \vert^{2} \gamma^{s}d\mu
\end{split}
\end{gather*}
\begin{equation} \label{m=1,5}
\hspace{56pt}
+
cs^{2}\Lambda^{2}\int_{[\gamma>0]} \Vert f \Vert^{8} \vert \nabla^{2} A \vert^{2} d\mu.
\end{equation}
Inserting finally (\ref{m=1,5}) in (\ref{m=1,2}) we derive
\begin{gather*}
\begin{split}
\int_{\Sigma} \Vert f  \Vert^{8} 
\vert A \vert^{4} \vert \nabla A \vert^{2}&  \gamma^{s} d\mu \\
\leq & 
c\int_{[\gamma>0]} \vert A \vert^{2}d\mu  \\
& [
\int_{\Sigma} \Vert f \Vert^{8} \vert A \vert^{4} \vert \nabla A \vert^{2} \gamma^{s}d\mu 
+
\int_{\Sigma} \Vert f \Vert^{8} \vert \nabla^{3} A \vert^{2} \gamma^{s} d\mu \\
& +
\int_{\Sigma} \Vert f \Vert^{4} \vert \nabla A \vert^{2} \gamma^{s} d\mu
+
\int_{[\gamma>0]} \Vert f \Vert^{8} \vert \nabla^{2} A \vert^{2} d\mu 
\int_{\Sigma} \vert \nabla A \vert^{2} \gamma^{s}d\mu\\
&+
s^{2}\Lambda^{2}\int_{[\gamma>0]} \Vert f \Vert^{8} \vert \nabla^{2} A \vert^{2} d\mu
+
s^{4}\Lambda^{4}\int_{[\gamma>0]} \Vert f \Vert^{8} \vert A \vert^{4} d\mu\,].
\end{split}
\end{gather*}
Since 
$\int_{[\gamma>0]}\Vert f \Vert^{8} \vert A \vert^{4} d\mu
\leq 
c\Vert \,\Vert f \Vert^{4} A \Vert^{4}_{L^{\infty}_{\mu}([\gamma>0])}
+
c(\int_{[\gamma>0]}\vert A \vert^{2}d\mu)^{2}$, 
we conclude
\begin{gather*}
\begin{split}
\int_{\Sigma} \Vert f  \Vert & ^{8} 
\vert A \vert^{4} \vert \nabla A \vert^{2}  \gamma^{s} d\mu \\
\leq & 
c\int_{[\gamma>0]} \vert A \vert^{2}d\mu  
[
\int_{\Sigma} \Vert f \Vert^{8} \vert A \vert^{4} \vert \nabla A \vert^{2} \gamma^{s}d\mu 
+
\int_{\Sigma} \Vert f \Vert^{8} \vert \nabla^{3} A \vert^{2} \gamma^{s} d\mu\;
] \\
& +
c_{1}
(
1
+
\int_{[\gamma>0]}\Vert f \Vert^{8}\vert \nabla^{2} A \vert^{2} d\mu
+
\Vert\,\Vert f \Vert^{4} A \Vert^{4}_{L^{\infty}_{\mu}[\gamma>0]}
)(1+\int_{\Sigma} \vert \nabla A \vert^{2}\gamma^{s} d\mu).
\end{split}
\end{gather*}
Consequently
\begin{gather*}
\begin{split}
\int_{\Sigma} \Vert f  \Vert^{8} 
\vert A \vert^{4} & \vert \nabla A  \vert^{2}  \gamma^{s} d\mu \\
\leq &
\varepsilon\int_{\Sigma} \Vert f \Vert^{8}  \vert \nabla^{3} A \vert^{2} \gamma^{s}d\mu \\
& +
c(\varepsilon,c_{1}) 
(
1
+
\int_{[\gamma>0]}\Vert f \Vert^{8}\vert \nabla^{2} A \vert^{2} d\mu
+
\Vert\,\Vert f \Vert^{4} A \Vert^{4}_{L^{\infty}_{\mu}([\gamma>0])}
)
\end{split}
\end{gather*}
\begin{equation} \label{m=1,6}
\hspace{-22pt}
(1+\int_{\Sigma} \vert \nabla A \vert^{2}\gamma^{s} d\mu),
\end{equation}
provided, that 
$\int_{[\gamma>0]}\vert A \vert^{2}d\mu
\leq \delta(\varepsilon).
$
\begin{itemize}
\item 
$k=5,j=0,i=1$
\end{itemize}
By Young's inequality  we have
\begin{gather*} 
\hspace{-32pt}
\begin{split}
\int_{\Sigma} \Vert f \Vert^{7} \vert A \vert^{5} \vert \nabla A \vert \gamma^{s}d\mu 
\leq &
c\int_{\Sigma} \Vert f \Vert^{8} \vert A \vert^{4} \vert \nabla A \vert^{2} \gamma^{s}d\mu
+
c\int_{\Sigma}\Vert f \Vert^{6} \vert A \vert^{6} \gamma^{s}d\mu.
\end{split}
\end{gather*}
From lemma \ref{A4} we infer 
\begin{gather*}
\begin{split}
\int_{\Sigma} \Vert f \Vert^{6} & \vert A  \vert^{6} \gamma^{s}d\mu
= 
\int_{\Sigma} (\Vert f \Vert^{3}\vert A \vert^{3}\gamma^{\frac{s}{2}})^{2}d\mu \\
\leq &
c
[
\int_{\Sigma} \Vert f \Vert^{2} \vert A \vert^{3} \gamma^{\frac{s}{2}}d\mu
+
\int_{\Sigma} \Vert f \Vert^{3} \vert A \vert^{2} \vert \nabla A \vert \gamma^{\frac{s}{2}}d\mu\\
&\hspace{24pt} +
s\Lambda\int_{\Sigma} \Vert f \Vert^{3} \vert A \vert^{3}\gamma^{\frac{s-2}{2}}d\mu
+
\int_{\Sigma} \Vert f \Vert^{3} \vert A \vert^{4} \gamma^{\frac{s}{2}}d\mu
]^{2}\\
\leq &
c\int_{[\gamma>0]}\vert A \vert^{2}d\mu
[
\int_{\Sigma} \Vert f \Vert^{6} \vert A \vert^{6}\gamma^{s}d\mu
+
\int_{[\gamma>0]}
(
\Vert f \Vert^{2}
+ 
s^{4}\Lambda^{4}\Vert f \Vert^{6}
)
\vert A \vert^{2}d\mu
\\
& \hspace{88pt}+
\int_{\Sigma}\Vert f \Vert^{4} \vert \nabla A \vert^{2}\gamma^{s}d\mu
+
\int_{\Sigma} \Vert f \Vert^{8} \vert A \vert^{4} \vert \nabla A \vert^{2} \gamma^{s}d\mu
].
\end{split}
\end{gather*}
Since 
$\int_{[\gamma>0]}\vert A \vert^{2}d\mu
\leq
\varepsilon_{1}$ we obtain by absorption
\begin{equation} \label{m=1,6.5}
\hspace{-8pt}
\int_{\Sigma} \Vert f \Vert^{6} \vert A \vert^{6} \gamma^{s} d\mu
\leq 
c \int_{\Sigma} \Vert f \Vert^{8} \vert A \vert^{4} \vert \nabla A \vert^{2} \gamma^{s}d\mu
+
c_{1}(1+\int_{\Sigma}\vert \nabla A \vert^{2} \gamma^{s} d\mu).
\end{equation}
Reinserting we conclude by (\ref{m=1,6})
\begin{gather*}
\begin{split}
\int_{\Sigma} \Vert f \Vert^{7} \vert A \vert^{5} \vert \nabla A \vert & \gamma^{s}d\mu 
\leq 
c\int_{\Sigma} \Vert f \Vert^{8} \vert A \vert^{4} \vert \nabla A \vert^{2} \gamma^{s}d\mu
+
c_{1}(1+\int_{\Sigma} \vert \nabla A \vert^{2} \gamma^{s}d\mu) \\
\leq &
\varepsilon\int_{\Sigma} \Vert f \Vert^{8}  \vert \nabla^{3} A \vert^{2} \gamma^{s}d\mu \\
& +
c(\varepsilon,c_{1}) 
(
1
+
\int_{[\gamma>0]}\Vert f \Vert^{8}\vert \nabla^{2} A \vert^{2} d\mu
+
\Vert\,\Vert f \Vert^{4} A \Vert^{4}_{L^{\infty}_{\mu}([\gamma>0])}
) \\
& \hspace{61pt}(
1
+
\int_{\Sigma} \vert \nabla A \vert^{2}\gamma^{s} d\mu
),
\end{split}
\end{gather*}
provided, that
$\int_{[\gamma>0]}\vert A \vert^{2}d\mu
\leq \delta(\varepsilon).$

\begin{itemize}
\item 
$k=3,j=3,i=0$
\end{itemize}
For the first term we have
\begin{gather*} 
\begin{split}
\int_{\Sigma} \Vert f \Vert^{8} \vert A \vert^{2}  \vert \nabla A \vert \vert  \nabla^{3} A \vert & \gamma^{s} d\mu\\
\leq &
\varepsilon \int_{\Sigma} \Vert f \Vert^{8} \vert \nabla^{3} A \vert^{2} \gamma^{s}d\mu
+
c(\varepsilon)\int_{\Sigma} \Vert f \Vert^{8} \vert A \vert^{4} \vert \nabla A \vert^{2} \gamma^{s}d\mu.
\end{split}
\end{gather*}
Apply (\ref{m=1,6}) to obtain a suitable estimate. Next 
\begin{gather*}
\begin{split}
\int_{\Sigma} 
\hspace{-1pt}
\Vert f \Vert^{8} \vert A \vert \vert \nabla A \vert^{2} \vert \nabla^{2} A \vert  \gamma^{s}d\mu 
\leq &
c
\hspace{-3pt}
\int_{\Sigma} \Vert f \Vert^{8} \vert \nabla A \vert^{4} \gamma^{s} d\mu
+
c
\hspace{-3pt}
\int_{\Sigma} \Vert f \Vert^{8} \vert A \vert^{2} \vert \nabla^{2} A \vert^{2} \gamma^{s} d\mu.
\end{split}
\end{gather*}
Using (\ref{m=1,1}) in a first step, then applying (\ref{m=1,5}) we derive 
\begin{gather*}
\hspace{-18pt}
\begin{split}
c\int_{\Sigma} \Vert f \Vert^{8} & \vert  \nabla  A \vert^{4} \gamma^{s} d\mu
+
c
\int_{\Sigma} \Vert f \Vert^{8} \vert A \vert^{2} \vert \nabla^{2} A \vert^{2} \gamma^{s} d\mu
\\
\leq &
c\int_{[\gamma>0]}\vert A \vert^{2}d\mu\int_{\Sigma} \Vert f \Vert^{8} \vert \nabla^{3} A \vert^{2} \gamma^{s} d\mu  
+
c \int_{\Sigma} \Vert f \Vert^{8} \vert A \vert^{4} \vert \nabla A \vert^{2} \gamma^{s}d\mu \\
& +
c\int_{\Sigma} \Vert f \Vert^{4} \vert \nabla A \vert^{2} \gamma^{s}d\mu  
+
c\int_{[\gamma>0]} \Vert f \Vert^{8} \vert \nabla^{2} A \vert^{2} d\mu 
\int_{\Sigma} \vert \nabla A \vert^{2} \gamma^{s}d\mu 
\end{split}
\end{gather*}
\begin{equation} \label{m=1,7}
\hspace{24pt}
+
cs^{4}\Lambda^{4}
\int_{[\gamma>0]}\Vert f \Vert^{8} \vert A \vert^{4}  d\mu 
+
cs^{2}\Lambda^{2}
\int_{[\gamma>0]} \Vert f \Vert^{8} \vert \nabla^{2} A \vert^{2} d\mu,
\end{equation}
\newpage \hspace{-22pt}
hence by (\ref{m=1,6}) and 
$\int_{[\gamma>0]}\Vert f \Vert^{8} \vert A \vert^{4} d\mu
\leq 
c\Vert \,\Vert f \Vert^{4} A \Vert^{4}_{L^{\infty}_{\mu}([\gamma>0])}
+
c(\int_{[\gamma>0]}\vert A \vert^{2}d\mu)^{2}$ 
\begin{gather*}
\begin{split}
\int_{\Sigma}\Vert f \Vert^{8} \vert A \vert  \vert \nabla A \vert^{2} & \vert \nabla^{2} A \vert  \gamma^{s}d\mu \\
\leq &
\varepsilon\int_{\Sigma} \Vert f \Vert^{8}  \vert \nabla^{3} A \vert^{2} \gamma^{s}d\mu \\
& +
c(\varepsilon,c_{1}) 
(
1
+
\int_{[\gamma>0]}\Vert f \Vert^{8}\vert \nabla^{2} A \vert^{2} d\mu
+
\Vert\,\Vert f \Vert^{4} A \Vert^{4}_{L^{\infty}_{\mu}([\gamma>0])}
) \\
& \hspace{50pt}
(
1
+
\int_{\Sigma} \vert \nabla A \vert^{2}\gamma^{s} d\mu
),
\end{split}
\end{gather*}
provided, that
$\int_{[\gamma>0]}\vert A \vert^{2}d\mu
\leq \delta(\varepsilon).$ \\
For the last term of this case, i.e.
$
\int_{\Sigma} \Vert f \Vert^{8} \vert \nabla A \vert^{4} \gamma^{s} d\mu,
$
apply (\ref{m=1,7}) and (\ref{m=1,6}).

\begin{itemize}
\item 
$k=3,j=2,i=1$
\end{itemize}
By Young's inequality we have
\begin{gather*} 
\begin{split}
\int_{\Sigma} \Vert f \Vert^{7} \vert A \vert^{2} \vert \nabla A \vert \vert \nabla^{2} A \vert &\gamma^{s} d\mu
+
\int_{\Sigma} \Vert f \Vert^{7} \vert A \vert \vert \nabla A \vert^{3} \gamma^{s} d\mu\\
\leq &
c \int_{\Sigma}\Vert f \Vert^{6} \vert A \vert^{2} \vert \nabla A \vert^{2}
\gamma^{s}d\mu 
+c\int_{\Sigma} \Vert f \Vert^{8} \vert \nabla A \vert^{4} \gamma^{s} d\mu
 \\
& +
c\int_{\Sigma} \Vert f \Vert^{8} \vert A \vert^{2} \vert \nabla^{2} A
 \vert^{2} \gamma^{s}d\mu
\\
\leq &
c\int_{\Sigma} \Vert f \Vert^{4} \vert \nabla A \vert^{2} \gamma^{s} d\mu
+
c\int_{\Sigma} \Vert f \Vert^{8} \vert A \vert^{4} \vert \nabla A \vert^{2} \gamma^{s}d\mu\\
& +
c\int_{\Sigma} \Vert f \Vert^{8} \vert A \vert^{2} \vert \nabla^{2} A
 \vert^{2} \gamma^{s}d\mu 
+
c\int_{\Sigma} \Vert f \Vert^{8} \vert \nabla A \vert^{4} \gamma^{s} d\mu.
\end{split}
\end{gather*}
Apply (\ref{m=1,7}) and (\ref{m=1,6}).

\begin{itemize}
\item 
$k=3,j=1,i=2$
\end{itemize}
Apply (\ref{m=1,6}) to
\begin{gather*} 
\begin{split}
\int_{\Sigma}
(
\Vert f \Vert^{6}
 +
\Vert f \Vert^{7} \vert A \vert 
) &
\vert A \vert^{2} \vert \nabla A \vert^{2} \gamma^{s} d\mu \\
\leq &
c\int_{\Sigma} \Vert f \Vert^{8} \vert A \vert^{4} \vert \nabla A \vert^{2} \gamma^{s} d\mu
+
c\int_{\Sigma} \Vert f \Vert^{4} \vert \nabla A \vert^{2} \gamma^{s} d\mu.
\end{split}
\end{gather*}

\begin{itemize}
\item 
$k=3,j=0,i=3$
\end{itemize}
By Young's inequality we have
\begin{gather*} 
\begin{split}
\int_{\Sigma}(
\Vert f \Vert^{5}+
\Vert f \Vert^{6} \vert A \vert
& +
\Vert f \Vert^{7} \vert A \vert^{2}
+
\Vert f \Vert^{7} \vert \nabla A \vert
)
\vert A \vert^{3} \vert \nabla A \vert \gamma^{s} d\mu \\
\leq &
c\int_{\Sigma} \Vert f \Vert^{8} \vert A \vert^{4} \vert \nabla A \vert^{2} \gamma^{s} d\mu
+
c\int_{\Sigma} \Vert f \Vert^{4} \vert \nabla A \vert^{2} \gamma^{s}d\mu \\
& +
c\int_{[\gamma>0]} \Vert f \Vert^{2} \vert A \vert^{2} d\mu 
+
c\int_{\Sigma} \Vert f \Vert^{6} \vert A \vert^{6} \gamma^{s} d\mu.
\end{split}
\end{gather*}
Apply (\ref{m=1,6.5})  and (\ref{m=1,6}).
\begin{itemize}
\item 
$k=1,j=4,i=1$
\end{itemize}
We use integration by parts to derive
\begin{gather*} 
\begin{split}
\int_{\Sigma} \nabla \Vert f \Vert^{8} & *\nabla^{4} A *  \nabla A \gamma^{s}d\mu \\
= &
-\int_{\Sigma} \nabla^{2} \Vert f \Vert^{8}*\nabla^{3} A * \nabla A \gamma^{s}d\mu
-
\int_{\Sigma} \nabla \Vert f \Vert^{8}*\nabla^{3} A * \nabla^{2} A \gamma^{s}d\mu\\
& -
s\int_{\Sigma} \nabla \Vert f \Vert^{8}*\nabla^{3} A * \nabla A*\nabla \gamma\, \gamma^{s-1}d\mu \\
\leq &
c(n)\int_{\Sigma} 
(
\Vert f \Vert^{6} 
+
\Vert f \Vert^{7} \vert A \vert
)
\vert \nabla A \vert\vert \nabla^{3} A \vert  \gamma^{s}d\mu \\
& +
c(n)\int_{\Sigma} \Vert f \Vert^{7} \vert \nabla^{2} A \vert \vert \nabla^{3} A \vert \gamma^{s}d\mu\\
& +
c(n)s \Lambda \int_{\Sigma} \Vert f \Vert^{7} \vert \nabla A \vert \vert \nabla^{3} A \vert \gamma^{s-1}d\mu \\
\leq &
\varepsilon \int_{\Sigma} \Vert f \Vert^{8} \vert \nabla^{3} A \vert^{2} \gamma^{s} d\mu
+
c(n,\varepsilon) \int_{\Sigma} \Vert f \Vert^{4} \vert \nabla A \vert^{2} \gamma^{s}d\mu\\
& +
c(n,\varepsilon)\int_{\Sigma} \Vert f \Vert^{6} \vert A \vert^{2} \vert \nabla A \vert^{2} \gamma^{s}d\mu
+
c(n,\varepsilon) \int_{\Sigma} \Vert f \Vert^{6} \vert \nabla^{2} A \vert^{2} \gamma^{s} d\mu\\
& +
c(n,\varepsilon)s^{2} \Lambda^{2} \int_{\Sigma} \Vert f \Vert^{6} \vert \nabla A \vert^{2} \gamma^{s-2}d\mu.
\end{split}
\end{gather*}
By Young's inequality we have
\begin{gather*}
\begin{split}
\int_{\Sigma} \Vert f \Vert^{6} \vert A \vert^{2} \vert \nabla A \vert^{2} \gamma^{s}d\mu
\leq &
c \int_{\Sigma} \Vert f \Vert^{8} \vert A \vert^{4} \vert \nabla A \vert^{2} \gamma^{s}d\mu
+
c\int_{\Sigma} \Vert f \Vert^{4} \vert \nabla A \vert^{2} \gamma^{s} d\mu.
\end{split}
\end{gather*}
Furthermore we use corollary \ref{A7} to derive
\begin{gather*}
\begin{split}
\int_{\Sigma} \Vert f \Vert^{6} \vert \nabla^{2} & A \vert^{2} \gamma^{s} d\mu
 +
s^{2} \Lambda^{2} \int_{\Sigma} \Vert f \Vert^{6} \vert \nabla A \vert^{2} \gamma^{s-2}d\mu \\
\leq &
\varepsilon \int_{\Sigma} \Vert f \Vert^{8} \vert \nabla^{3} A \vert^{2} \gamma^{s} d\mu
+
c(\varepsilon)\int_{\Sigma} \Vert f \Vert^{4} \vert \nabla A \vert^{2} \gamma^{s}d\mu \\
& +
c(\varepsilon,s,\Lambda)\int_{\Sigma} \Vert f \Vert^{6} \vert \nabla A \vert^{2} \gamma^{s-2}d\mu \\
\leq &
\varepsilon \int_{\Sigma} \Vert f \Vert^{8} \vert \nabla^{3} A \vert^{2} \gamma^{s} d\mu
+
c(\varepsilon)\int_{\Sigma} \Vert f \Vert^{4} \vert \nabla A \vert^{2} \gamma^{s}d\mu \\
& +
\delta \int_{\Sigma} \Vert f \Vert^{6} \vert \nabla^{2} A \vert^{2} \gamma^{s} d\mu \\
& +
c(\varepsilon,\delta,s,\Lambda) \int_{[\gamma>0]}(\Vert f \Vert^{6}+\Vert f \Vert^{4})\vert A \vert^{2}d\mu,
\end{split}
\end{gather*}
i.e. by absorption
\begin{gather*}
\begin{split}
\int_{\Sigma} \Vert f \Vert^{6} \vert \nabla^{2} & A \vert^{2} \gamma^{s} d\mu
 +
s^{2} \Lambda^{2} \int_{\Sigma} \Vert f \Vert^{6} \vert \nabla A \vert^{2} \gamma^{s-2}d\mu \\
\leq &
\varepsilon \int_{\Sigma} \Vert f \Vert^{8} \vert \nabla^{3} A \vert^{2} \gamma^{s} d\mu
+
c(\varepsilon)\int_{\Sigma} \Vert f \Vert^{4} \vert \nabla A \vert^{2} \gamma^{s}d\mu 
\end{split}
\end{gather*}
\begin{equation} \label{m=1,8}
\hspace{8pt}
+c(\varepsilon,s,\Lambda) \int_{[\gamma>0]}(\Vert f \Vert^{6}+\Vert f \Vert^{4})\vert A \vert^{2}d\mu.
\end{equation}
Inserting yields
\begin{gather*}
\begin{split}
\int_{\Sigma} \nabla \Vert f \Vert^{8} & *\nabla^{4} A *  \nabla A \gamma^{s}d\mu\\
\leq &
\varepsilon \int_{\Sigma} \Vert f \Vert^{8} \vert \nabla^{3} A \vert^{2} \gamma^{s} d\mu
+
c(n,\varepsilon)\int_{\Sigma} \Vert f \Vert^{8} \vert A \vert^{4} \vert \nabla A \vert^{2} \gamma^{s}d\mu
\\
& +
c(n,\varepsilon) \int_{\Sigma} \Vert f \Vert^{4} \vert \nabla A \vert^{2} \gamma^{s}d\mu \\
& +
c(n,\varepsilon,s,\Lambda)\int_{[\gamma>0]}(\Vert f \Vert^{6}+\Vert f \Vert^{4})\vert A \vert^{2}d\mu.
\end{split}
\end{gather*}
Apply (\ref{m=1,6}).
\newpage
\begin{itemize}
\item 
$k=1,j=3,i=2$
\end{itemize}
By Young's inequality we have
\begin{gather*} 
\begin{split}
\int_{\Sigma}
(
\Vert f \Vert^{6}
+
\Vert f \Vert^{7} & \vert A \vert
)
\vert \nabla A \vert \vert \nabla^{3} A \vert \gamma^{s}d\mu\\
\leq &
\varepsilon \int_{\Sigma} \Vert f \Vert^{8} \vert \nabla^{3} A \vert^{2} \gamma^{s}d\mu
+
c(\varepsilon)\int_{\Sigma}\Vert f \Vert^{4} \vert \nabla A \vert^{2}\gamma^{s} d\mu\\
& +
c(\varepsilon)\int_{\Sigma}\Vert f \Vert^{6} \vert A \vert^{2} \vert \nabla A \vert^{2} \gamma^{s}d\mu \\
\leq &
\varepsilon \int_{\Sigma} \Vert f \Vert^{8} \vert \nabla^{3} A \vert^{2} \gamma^{s}d\mu
+
c(\varepsilon)\int_{\Sigma} \Vert f \Vert^{8} \vert A \vert^{4} \vert \nabla A \vert^{2} \gamma^{s}
d\mu\\
& +
c(\varepsilon)\int_{\Sigma}\Vert f \Vert^{4} \vert \nabla A \vert^{2}\gamma^{s} d\mu.
\end{split}
\end{gather*}
Apply (\ref{m=1,6}).
\begin{itemize}
\item 
$k=1,j=2,i=3$
\end{itemize}
\begin{gather*} 
\begin{split}
\int_{\Sigma}
(
\Vert f \Vert^{5}
+
\Vert f \Vert^{6} \vert A \vert
& +
\Vert f \Vert^{7} \vert A \vert^{2}
+
\Vert f \Vert^{7} \vert \nabla A \vert
)
\vert \nabla A \vert \vert \nabla^{2} A \vert \gamma^{s} d\mu \\
\leq &
c\int_{[\gamma>0]} \Vert f \Vert^{8} \vert \nabla^{2} A \vert^{2} d\mu
+
c\int_{\Sigma} 
(
\Vert f \Vert^{2}
+
\Vert f \Vert^{4}
) 
\vert \nabla A \vert^{2} \gamma^{s}d\mu\\
&+
c\int_{\Sigma} \Vert f \Vert^{8} \vert A \vert^{2} \vert \nabla^{2} A \vert^{2} \gamma^{s} d\mu
+
c\int_{\Sigma} \Vert f \Vert^{8} \vert A \vert^{4} \vert \nabla A \vert^{2}\gamma^{s}d\mu\\
& +
c\int_{\Sigma} \Vert f \Vert^{8} \vert \nabla A \vert^{4} \gamma^{s} d\mu
+
c\int_{\Sigma} \Vert f \Vert^{6}\vert \nabla^{2} A \vert^{2} \gamma^{s}d\mu.
\end{split}
\end{gather*}
Apply (\ref{m=1,7}), (\ref{m=1,6}) and (\ref{m=1,8}).
Collecting terms we conclude 
\begin{gather*}
\begin{split}
\hspace{8pt}
\sum_{(i,j,k) \in I(1),\,j<5}  \int_{\Sigma}  & \nabla^{i}  \Vert f \Vert^{8}   *P^{j}_{k}(A)*\nabla A \,\gamma^{s}d\mu \\ 
\leq & 
\varepsilon \int_{\Sigma} \Vert f \Vert^{8} \vert \nabla^{3} A \vert^{2}d\mu \\
& +
c(\varepsilon,c_{1})
(
1
+
\int_{[\gamma>0]}\Vert f \Vert^{8}\vert \nabla^{2} A \vert^{2} d\mu
+
\Vert\,\Vert f \Vert^{4} A \Vert^{4}_{L^{\infty}_{\mu}([\gamma>0])}
)\\
& \hspace{49pt}
(
1
+
\int_{\Sigma} \vert \nabla A \vert^{2}\gamma^{s} d\mu
)
\end{split}
\end{gather*}
for $\int_{[\gamma>0]}\vert A \vert^{2}d\mu \leq \delta(\varepsilon,n)$ small enough.
Choosing $\varepsilon=\frac{1}{16}$ proves (\ref{m=1,0}).
\end{proof}
\section{Estimates by smallness assumption, m=2}

\begin{proposition} \label{8.1}
For $n,\Lambda,R,d,\tau >0$ and $s\geq 8$ there exist
\begin{gather*}
\varepsilon_{2}=\varepsilon_{2}(n),\;c_{2}=c_{2}(n,s,\Lambda,R,d,\tau )>0,
\end{gather*}
such that, if \quad
$
f: \Sigma \times [0,T) \rightarrow \mathbb{R}^{n}\setminus \lbrace \mathbb{0} \rbrace,\quad 0<T\leq \tau,
$ \\
is an inverse Willmore flow,
$\gamma=\widetilde{\gamma} \circ f$ as in $(\ref{eq13})$ and
\begin{gather*}
\begin{split}
&
\sup_{0\leq t < T} \int_{[\gamma>0]}\vert A \vert^{2}d\mu
\leq \varepsilon_{2}, \\
&
\sup_{0 \leq t <T}\Vert f \Vert_{L^{\infty}_{\mu}([\gamma>0])}\leq R,\\
&
\sup_{0 \leq t <T}\int_{[\gamma>0]}\vert \nabla A \vert^{2}d\mu 
\leq 
d
,\\
&
\int^{T}_{0}\int_{[\gamma>0]}\Vert f \Vert^{8}\vert \nabla^{k} A \vert^{2} d\mu\; dt
\leq
d
\quad \text{for} \quad k=2,3 \quad
\text{and}\\
&
\int^{T}_{0}\Vert\,\Vert f \Vert^{4} \nabla^{k} A \Vert^{4}_{L^{\infty}_{\mu}([\gamma>0])}dt
\leq 
d
\quad \text{ for } k=0,1,
\end{split}
\end{gather*}
we have 
\begin{gather*}
\begin{split}
\sup_{0 \leq t <T}\int_{\Sigma}\vert \nabla^{2} A \vert^{2} \gamma^{s}d\mu
& + 
\hspace{-2pt}
\int^{T}_{0} \hspace{-7pt}
\int_{\Sigma} \hspace{-0.2pt}
\Vert f \Vert^{8}\vert \nabla^{4} A \vert^{2}\gamma^{s}d\mu\,dt
\leq 
c_{2}
(
1
+
\hspace{-3pt}
\int_{\Sigma} \vert \nabla^{2} A \vert^{2} \gamma^{s} d\mu \lfloor_{t=0}
).
\end{split}
\end{gather*} 
\end{proposition}
\begin{proof}
For abbreviative reasons let
$
c_{2}=c_{2}(n,s,\Lambda,R,d).
$
We will show 
\begin{gather*}
\begin{split}
\sum_{(i,j,k) \in I(2),\,j<6} & \int_{\Sigma}  \nabla^{i}\Vert f \Vert^{8}  *P^{j}_{k}(A)*\nabla A \,\gamma^{s}d\mu \\ 
\leq & 
\int_{\Sigma} \frac{\Vert f \Vert^{8}}{16} \vert \nabla^{4} A \vert^{2}d\mu \\
& +
c_{2}
(
1
+
\int_{[\gamma>0]}\Vert f \Vert^{8}\vert \nabla^{2} A \vert^{2} d\mu
+
\int_{[\gamma>0]}\Vert f \Vert^{8}\vert \nabla^{3} A \vert^{2} d\mu\\
& \hspace{35pt}+
\Vert\,\Vert f \Vert^{4} A \Vert^{4}_{L^{\infty}_{\mu}([\gamma>0])}
+
\Vert\,\Vert f \Vert^{4} \nabla A \Vert^{4}_{L^{\infty}_{\mu}([\gamma>0])}
)
\end{split}
\end{gather*}
\begin{equation} \label{m=2,1}
\hspace{-49pt}
(1+\int_{\Sigma} \vert \nabla^{2} A \vert^{2}\gamma^{s} d\mu).
\end{equation}
Applying this inequality and proposition \ref{5.2} with $\varepsilon=\frac{1}{8c}$ to proposition \ref{4.3} prove the claim using Gronwall's inequality, cf. lemma \ref{Gronwall's inequality}.

To show (\ref{m=2,1}) we give adequate estimates for each term of the sum.

\begin{itemize}
\item 
$k=5 ,j=2, i=0 $
\end{itemize}
By lemma \ref{A4} we estimate
\begin{gather*}
\begin{split}
\int_{\Sigma}\Vert f \Vert^{8}  \vert A \vert^{4} & \vert \nabla^{2} A \vert^{2} \gamma^{s} d\mu
=
\int_{\Sigma}(\Vert f \Vert^{4} \vert A \vert^{2} \vert \nabla^{2} A \vert \gamma^{\frac{s}{2}})^{2} d\mu\\
\leq &
c
[
\int_{\Sigma} \Vert f \Vert^{3} \vert A \vert^{2} \vert \nabla^{2} A \vert \gamma^{\frac{s}{2}}d\mu
+
\int_{\Sigma} \Vert f \Vert^{4} \vert A \vert \vert \nabla A \vert \vert \nabla^{2} A \vert \gamma^{\frac{s}{2}}d\mu\\
& \hspace{22pt}+
\int_{\Sigma} \Vert f \Vert^{4} \vert A \vert^{2} \vert \nabla^{3} A \vert \gamma^{\frac{s}{2}}d\mu
+
s\Lambda\int_{\Sigma} \Vert f \Vert^{4} \vert A \vert^{2} \vert \nabla^{2} A \vert \gamma^{\frac{s-2}{2}}d\mu\\
& \hspace{22pt}+
\int_{\Sigma} \Vert f \Vert^{4} \vert A \vert^{3} \vert \nabla^{2} A \vert \gamma^{\frac{s}{2}}d\mu\;
]^{2} \\
\leq &
c\int_{[\gamma>0]} \vert A \vert^{2} d\mu \\
& 
[
\int_{\Sigma} \Vert f \Vert^{6} \vert A \vert^{2} \vert \nabla^{2} A \vert^{2} \gamma^{s} d\mu
+
\int_{\Sigma} \Vert f \Vert^{8} \vert A \vert^{2} \vert \nabla^{3} A \vert^{2} \gamma^{s} d\mu \\
& \quad + 
s^{2} \Lambda^{2}\int_{\Sigma} \Vert f \Vert^{8} \vert A \vert^{2} \vert \nabla^{2} A \vert^{2} \gamma^{s-2}d\mu
+
\int_{\Sigma} \Vert f \Vert^{8} \vert A \vert^{4} \vert \nabla^{2} A \vert^{2} \gamma^{s}d\mu
\;]\\
& +
c\int_{[\gamma>0]} \Vert f \Vert^{8} \vert A \vert^{2} \vert \nabla A \vert^{2} d\mu \int_{\Sigma} \vert \nabla^{2} A \vert^{2} \gamma^{s} d\mu \\
\leq &
c\int_{[\gamma>0]} \vert A \vert^{2} d\mu \\
& 
[
\int_{\Sigma} \Vert f \Vert^{8} \vert A \vert^{4} \vert \nabla^{2} A \vert^{2} \gamma^{s}d\mu
+
\int_{\Sigma} \Vert f \Vert^{8} \vert A \vert^{2} \vert \nabla^{3} A \vert^{2} \gamma^{s} d\mu\\
& \quad +
\int_{\Sigma} \Vert f \Vert^{4} \vert \nabla^{2} A \vert^{2} \gamma^{s} d\mu
+
s^{4}\Lambda^{4}\int_{[\gamma>0]} \Vert f \Vert^{8} \vert \nabla^{2} A \vert^{2}d\mu
\;]\\
& +
c
[\;
\Vert \, \Vert f \Vert^{4} A \Vert^{4}_{L^{\infty}_{\mu}([\gamma>0])}
+
(\int_{[\gamma>0]}\vert \nabla A \vert^{2} d\mu)^{2}
\;]
\int_{\Sigma} \vert \nabla^{2} A \vert^{2} \gamma^{s} d\mu.
\end{split}
\end{gather*}
Since $\int_{[\gamma>0]} \vert A \vert^{2} d\mu \leq \varepsilon_{2}$, we obtain by absorption \newpage
\begin{gather*}
\begin{split}
\int_{\Sigma}\Vert f \Vert^{8}  \vert A \vert^{4}  \vert \nabla^{2} A \vert^{2} \gamma^{s} d\mu 
\leq &
c\int_{[\gamma>0]} \vert A \vert^{2} d\mu
\int_{\Sigma} \Vert f \Vert^{8} \vert A \vert^{2} \vert \nabla^{3} A \vert^{2} \gamma^{s} d\mu\\
& +
c_{2}
(
1
+
\int_{[\gamma>0]}\Vert f \Vert^{8}\vert \nabla^{2} A \vert^{2} d\mu
+
\Vert\,\Vert f \Vert^{4} A \Vert^{4}_{L^{\infty}_{\mu}([\gamma>0])}
)
\end{split}
\end{gather*}
\begin{equation} \label{m=2,2}
\hspace{30pt}
(1+\int_{\Sigma} \vert \nabla^{2} A \vert^{2}\gamma^{s} d\mu).
\end{equation}
Next we have again by lemma \ref{A4}
\begin{gather*}
\begin{split}
\int_{\Sigma} \Vert f & \Vert^{8} \vert A \vert^{2}  \vert \nabla^{3} A \vert^{2} \gamma^{s} d\mu
=
\int_{\Sigma} (\Vert f \Vert^{4} \vert A \vert \vert \nabla^{3} A \vert \gamma^{\frac{s}{2}})^{2} d\mu \\
\leq &
c \int_{[\gamma>0]} \vert A \vert^{2} d\mu 
[
\int_{\Sigma} \Vert f \Vert^{6} \vert \nabla^{3} A \vert^{2} \gamma^{s} d\mu
+
\int_{\Sigma} \Vert f \Vert^{8} \vert \nabla^{4} A \vert^{2} d\mu \\
& \hspace{76pt}
+
s^{2}\Lambda^{2}\int_{[\gamma>0]} \Vert f \Vert^{8} \vert \nabla^{3} A \vert^{2} d\mu
+
\int_{\Sigma} \Vert f \Vert^{8} \vert A \vert^{2} \vert \nabla^{3} A \vert^{2} \gamma^{s} d\mu
\;]
\end{split}
\end{gather*}
\vspace{-6pt}
\begin{equation} \label{m=2,3}
\hspace{-134pt}
+
c\int_{[\gamma>0]} \Vert f \Vert^{8} \vert \nabla^{3} A \vert^{2} d\mu
\int_{[\gamma>0]} \vert \nabla A \vert^{2} d\mu.
\end{equation}
By corollary \ref{A7} 
\begin{gather*}
\begin{split}
\int_{\Sigma} \Vert f \Vert^{6} \vert \nabla^{3} A \vert^{2} \gamma^{s}d\mu
\leq &
\varepsilon \int_{\Sigma} \Vert f \Vert^{8} \vert \nabla^{4} A \vert^{2} \gamma^{s} d\mu
+
c(\varepsilon)\int_{\Sigma} \Vert f \Vert^{4} \vert \nabla^{2} A \vert^{2} \gamma^{s}d\mu \\
& +
c(\varepsilon)s^{2}\Lambda^{2} \int_{\Sigma} \Vert f \Vert^{6} \vert \nabla^{2} A \vert^{2} \gamma^{s-2}d\mu\\
\leq & 
\varepsilon \int_{\Sigma} \Vert f \Vert^{8} \vert \nabla^{4} A \vert^{2} \gamma^{s} d\mu
+
c(\varepsilon) \int_{\Sigma} \Vert f \Vert^{4} \vert \nabla^{2} A \vert^{2} \gamma^{s}d\mu 
\end{split}
\end{gather*}
\begin{equation} \label{m=2,4}
\hspace{29pt}
+
c(\varepsilon)s^{4} \Lambda^{4} \int_{[\gamma>0]} \Vert f \Vert^{8} \vert \nabla^{2} A \vert^{2} d\mu.
\end{equation}
Applying this inequality to (\ref{m=2,3}) we derive by absorption
\begin{gather*}
\hspace{-94pt}
\int_{\Sigma} 
\Vert f \Vert^{8} \vert A \vert^{2} \vert \nabla^{3} A \vert^{2} \gamma^{s} d\mu 
\leq 
c \int_{[\gamma>0]} \vert A \vert^{2} d\mu \int_{\Sigma} \Vert f \Vert^{8} \vert \nabla^{4} A \vert^{2} d\mu 
\end{gather*}
\begin{equation} \label{m=2,5}
+
c_{2}
(
1 
+ 
\int_{[\gamma>0]}
\Vert f \Vert^{8}\vert \nabla^{2} A \vert^{2} d\mu
+ 
\int_{[\gamma>0]}\Vert f \Vert^{8}\vert \nabla^{3} A \vert^{2} d\mu)
(1+\int_{\Sigma} \vert \nabla^{2} A \vert^{2}\gamma^{s} d\mu).
\end{equation}
\newpage \hspace{-22pt}
Inserting (\ref{m=2,5}) in (\ref{m=2,2}) we conclude
\begin{gather*}
\begin{split}
\int_{\Sigma}\Vert f \Vert^{8}  \vert A \vert^{4} & \vert \nabla^{2} A \vert^{2} \gamma^{s} d\mu \\
\leq &
\varepsilon\int_{\Sigma}\Vert f \Vert^{8} \vert \nabla^{4} A \vert^{2}d\mu \\
& +
c(\varepsilon,c_{2})
(
1
+
\int_{[\gamma>0]}\Vert f \Vert^{8}\vert \nabla^{2} A \vert^{2} d\mu
+
\int_{[\gamma>0]}\Vert f \Vert^{8}\vert \nabla^{3} A \vert^{2} d\mu\\
& \hspace{176pt} +
\Vert\,\Vert f \Vert^{4} A \Vert^{4}_{L^{\infty}_{\mu}([\gamma>0])}
)
\end{split}
\end{gather*}
\begin{equation} \label{m=2,6}
\hspace{-16pt}
(1+\int_{\Sigma} \vert \nabla^{2} A \vert^{2}\gamma^{s} d\mu),
\end{equation}
provided, that $\int_{[\gamma>0]} \vert A \vert^{2} d\mu\leq \delta(\varepsilon).$\\
For the next term to be considered we have by Young's inequality
\begin{gather*}
\begin{split}
\int_{\Sigma} \Vert f \Vert^{8} \vert A \vert^{3} \vert \nabla A \vert^{2} & \vert \nabla^{2} A \vert \gamma^{s} d\mu \\
\leq &
c \int_{\Sigma} \Vert f \Vert^{8} \vert A \vert^{4} \vert \nabla^{2} A \vert^{2} \gamma^{s}d\mu
+
c\int_{\Sigma} \Vert f \Vert^{8} \vert A \vert^{2} \vert \nabla A \vert^{4} \gamma^{s}d\mu.
\end{split}
\end{gather*}
We estimate by lemma \ref{A4}
\begin{gather*}
\begin{split}
\int_{\Sigma} \Vert f \Vert^{8} \vert A \vert^{2} & \vert \nabla A \vert^{4} \gamma^{s}d\mu
=
\int_{\Sigma} (\Vert f \Vert^{4} \vert A \vert \vert \nabla A \vert^{2} \gamma^{\frac{s}{2}})^{2}d\mu\\
\leq &
c
[
\int_{\Sigma} \Vert f \Vert^{3} \vert A\vert  \vert \nabla A \vert^{2} \gamma^{\frac{s}{2}}d\mu
+
\int_{\Sigma} \Vert f \Vert^{4}  \vert \nabla A \vert^{3} \gamma^{\frac{s}{2}}d\mu \\
& \hspace{22pt} +
\int_{\Sigma} \Vert f \Vert^{4} \vert A \vert  \vert \nabla A \vert \vert \nabla^{2} A \vert \gamma^{\frac{s}{2}}d\mu
+
s\Lambda\int_{\Sigma} \Vert f \Vert^{4} \vert A \vert  \vert \nabla A \vert^{2} \gamma^{\frac{s-2}{2}}d\mu \\
& \hspace{22pt} +
\int_{\Sigma} \Vert f \Vert^{4} \vert A \vert^{2}  \vert \nabla A \vert^{2} \gamma^{\frac{s}{2}} d\mu
\;]^{2}\\
\leq &
c\int_{[\gamma>0]} \vert \nabla A \vert^{2} d\mu \\
& [
\int_{\Sigma} \Vert f \Vert^{6} \vert A\vert^{2} \vert \nabla A \vert^{2} \gamma^{s} d\mu
+
\int_{\Sigma} \Vert f \Vert^{8}  \vert \nabla A \vert^{4} \gamma^{s} d\mu\\
& \quad +
\int_{\Sigma} \Vert f \Vert^{8} \vert A \vert^{2} \vert \nabla^{2} A \vert^{2} \gamma^{s} d\mu
+
s^{2} \Lambda^{2} \int_{\Sigma} \Vert f \Vert^{8} \vert A \vert^{2} \vert \nabla A \vert^{2}   \gamma^{s-2} d\mu\; ] \\
& +
c\int_{[\gamma>0]} \vert A \vert^{2} d\mu 
\int_{\Sigma} \Vert f \Vert^{8} \vert A \vert^{2} \vert \nabla A \vert^{4} \gamma^{s} d\mu.
\end{split}
\end{gather*}
Hence by Young's inequality 
\begin{gather*}
\begin{split}
\int_{\Sigma} \Vert f \Vert^{8} \vert A  \vert^{2}  \vert \nabla A & \vert^{4}  \gamma^{s}d\mu
\leq 
(
\varepsilon
+
c\int_{[\gamma>0]} \vert A \vert^{2} d\mu
)
\int_{\Sigma} \Vert f \Vert^{8} \vert A \vert^{2} \vert \nabla A \vert^{4} \gamma^{s} d\mu\\
& +
c \int_{\Sigma} \Vert f \Vert^{8} \vert A \vert^{4} \vert \nabla^{2} A \vert^{2} \gamma^{s} d\mu\\
& +
c(\varepsilon,s,\Lambda)(1+(\int_{[\gamma>0]} \vert \nabla A \vert^{2} d\mu)^{2})\\
& \hspace{58pt}
[
\int_{[\gamma>0]} \Vert f \Vert^{4} \vert A \vert^{2} d\mu
+
\int_{[\gamma>0]} \Vert f \Vert^{8} \vert \nabla A \vert^{4} d\mu \\
& \hspace{60pt}+
\int_{[\gamma>0]} \Vert f \Vert^{8} \vert \nabla^{2} A \vert^{2} d\mu
+
\int_{[\gamma>0]} \Vert f \Vert^{8} \vert A \vert^{2} d\mu\;
].
\end{split}
\end{gather*}
Absorbing, 
$\int_{[\gamma>0]} \Vert f \Vert^{8} \vert \nabla A \vert^{4} d\mu
\leq
c\Vert \, \Vert f \Vert^{4} \nabla A \Vert^{4}_{L^{\infty}_{\mu}([\gamma>0])}
+
c(\int_{[\gamma>0]}\vert \nabla A \vert^{2}d\mu)^{2}$ 
and (\ref{m=2,6}) yield  an adequate estimate. 

\begin{itemize}
\item 
$k=5 ,j=1, i=1 $
\end{itemize}
Clearly
\begin{gather*}
\begin{split}
\int_{\Sigma} \Vert f \Vert^{7} \vert A \vert^{4} \vert \nabla A \vert &  \vert \nabla^{2} A \vert \gamma^{s} d\mu\\
\leq &
c\int_{\Sigma} \Vert f \Vert^{8} \vert A \vert^{4} \vert \nabla^{2} A \vert^{2} \gamma^{s} d\mu
+
c\int_{\Sigma} \Vert f \Vert^{6} \vert A \vert^{4} \vert \nabla A \vert^{2} \gamma^{s} d\mu.
\end{split}
\end{gather*}
By lemma \ref{A4}  we derive
\begin{gather*}
\begin{split}
\int_{\Sigma} \Vert f & \Vert^{6} \vert A \vert^{4}  \vert \nabla A \vert^{2} \gamma^{s} d\mu 
= 
\int_{\Sigma}( \Vert f \Vert^{3} \vert A \vert^{2} \vert \nabla A \vert \gamma^{\frac{s}{2}})^{2} d\mu\\
\leq &
c \int_{[\gamma>0]} \vert A \vert^{2}d\mu \\
&
[
\int_{\Sigma} \Vert f \Vert^{6} \vert A \vert^{4} \vert \nabla A \vert^{2} \gamma^{s} d\mu
+
\int_{[\gamma>0]} \Vert f \Vert^{2} \vert \nabla A \vert^{2}  d\mu\\
& +
\int_{\Sigma} \Vert f \Vert^{6} \vert A \vert^{2} \vert \nabla^{2} A \vert^{2} \gamma^{s} d\mu
+
s^{4} \Lambda^{4}\int_{[\gamma>0]} \Vert f \Vert^{6} \vert \nabla A \vert^{2} d\mu
] \\
& +
\varepsilon\int_{\Sigma} \Vert f \Vert^{6} \vert A \vert^{4} \vert \nabla A \vert^{2} \gamma^{s}d\mu \\
& +
c(\varepsilon)(\int_{[\gamma>0]} \vert \nabla A \vert^{2}d\mu)^{2} 
\int_{\Sigma}\Vert f \Vert^{6} \vert \nabla A \vert^{2} \gamma^{s}d\mu.
\end{split}
\end{gather*}
Absorbing and reinserting we conclude by Young's inequality
\begin{gather*}
\begin{split}
\int_{\Sigma} \Vert f \Vert^{7} \vert A \vert^{4} & \vert \nabla A \vert \vert \nabla^{2} A \vert\gamma^{s}d\mu \\
\leq &
c\int_{\Sigma} \Vert f \Vert^{8} \vert A \vert^{4} \vert \nabla^{2} A \vert^{2} \gamma^{s} d\mu
+
c_{2}(1+\int_{\Sigma}  \vert \nabla^{2} A \vert^{2} \gamma^{s} d\mu).
\end{split}
\end{gather*}
Apply (\ref{m=2,6}).

\begin{itemize}
\item 
$k=5 ,j=0, i=2 $
\end{itemize}
First we have
\begin{gather*}
\begin{split}
\int_{\Sigma} 
(\Vert f \Vert^{6} 
& +  
\Vert f \Vert^{7}\vert A \vert
)
\vert A \vert^{5}  \vert \nabla^{2} A \vert \gamma^{s} d\mu \\
\leq &
c\int_{\Sigma} \Vert f \Vert^{8} \vert A \vert^{4} \vert \nabla^{2} A \vert^{2} \gamma^{s} d\mu
+
c\int_{\Sigma} \Vert f \Vert^{4} \vert A\vert^{6} \gamma^{s} d\mu
+
c\int_{\Sigma} \Vert f \Vert^{6} \vert A \vert^{8} \gamma^{s} d\mu.
\end{split}
\end{gather*} 
By lemma \ref{A4} we estimate

\begin{gather*}
\begin{split}
\int_{\Sigma} \Vert f \Vert^{4} \vert A& \vert^{6} \gamma^{s} d\mu 
= 
\int_{\Sigma} ( \Vert f \Vert^{2} \vert A\vert^{3} \gamma^{\frac{s}{2}})^{2} d\mu \\
\leq &
c
[
\int_{\Sigma} \Vert f \Vert \vert A \vert^{3} \gamma^{\frac{s}{2}} d\mu
+
\int_{\Sigma} \Vert f \Vert^{2} \vert A \vert^{2} \vert \nabla A \vert \gamma^{\frac{s}{2}} d\mu\\
& \hspace{22pt} +
s \Lambda\int_{\Sigma} \Vert f \Vert^{2} \vert A \vert^{3} \gamma^{\frac{s-2}{2}} d\mu
+
\int_{\Sigma} \Vert f \Vert^{2} \vert A \vert^{4} \gamma^{\frac{s}{2}} d\mu\;
]^{2} \\
\leq &
c\int_{[\gamma>0]} \vert \nabla A \vert^{2} d\mu \int_{\Sigma} \Vert f \Vert^{4} \vert A \vert^{4}\gamma^{s}d\mu \\
& +
c \int_{[\gamma>0]} \vert A \vert^{2} d\mu \\
& \quad 
[
\int_{\Sigma} \Vert f \Vert^{2} \vert A \vert^{4} \gamma^{s} d\mu
+
s^{2} \Lambda^{2} \int_{\Sigma} \Vert f \Vert^{4} \vert A \vert^{4} \gamma^{s-2} d\mu
+
\int_{\Sigma} \Vert f \Vert^{4} \vert A \vert^{6} \gamma^{s} d\mu\;
] \\
\leq &
(
\varepsilon
+
c \int_{[\gamma>0]} \vert A \vert^{2} d\mu 
)
\int_{\Sigma} \Vert f \Vert^{4} \vert A \vert^{6} \gamma^{s} d\mu
+
c(\int_{[\gamma>0]}  \vert A \vert^{2} d\mu)^{2}\\
& +
c(\varepsilon)
[\;
(\int_{[\gamma>0]}\vert \nabla A \vert^{2} d\mu)^{2}
+
s^{4}\Lambda^{4}
\;]
\int_{[\gamma>0]} \Vert f \Vert^{4} \vert A \vert^{2}  d\mu,
\end{split}
\end{gather*}
i.e. 
$\int_{\Sigma} \Vert f \Vert^{4} \vert A \vert^{6} \gamma^{s} d\mu \leq c_{2}$
by absorption. 

\hspace{-22pt} Similarly
\begin{gather*}
\begin{split}
\int_{\Sigma} \Vert f \Vert^{6} \vert A & \vert^{8} \gamma^{s} d\mu
= 
\int_{\Sigma} (\Vert f \Vert^{3} \vert A \vert^{4} \gamma^{\frac{s}{2}})^{2} d\mu\\
\leq &
c
[
\int_{\Sigma} \Vert f \Vert^{2} \vert A \vert^{4} \gamma^{\frac{s}{2}} d\mu
+
\int_{\Sigma} \Vert f \Vert^{3} \vert A \vert^{3} \vert \nabla A \vert \gamma^{\frac{s}{2}} d\mu\\
& \hspace{22pt} +
s \Lambda \int_{\Sigma} \Vert f \Vert^{3} \vert A \vert^{4} \gamma^{\frac{s-2}{2}} d\mu
+
\int_{\Sigma} \Vert f \Vert^{3} \vert A \vert ^{5}\gamma^{\frac{s}{2}} d\mu\;
]^{2}\\
\leq &
c \int_{[\gamma>0]} \vert \nabla A \vert^{2} d\mu 
\int_{\Sigma} \Vert f \Vert^{6} \vert A \vert^{6} \gamma^{s} d\mu\\
& +
c \int_{[\gamma>0]} \vert A \vert^{2} d\mu \\
&
[
\int_{\Sigma} \Vert f \Vert^{4} \vert A \vert^{6} \gamma^{s} d\mu
+
s^{2} \Lambda^{2}\int_{\Sigma} \Vert f \Vert^{6} \vert A \vert^{6} \gamma^{s-2} d\mu 
+
\int_{\Sigma} \Vert f \Vert^{6} \vert A \vert^{8} \gamma^{s} d\mu\;
]\\
\leq &
c
(
\varepsilon
+
\int_{[\gamma>0]} \vert A \vert^{2} d\mu 
)
\int_{\Sigma} \Vert f \Vert^{6} \vert A \vert^{8} \gamma^{s} d\mu
+c(\int_{[\gamma>0]} \vert A \vert^{2}d\mu)^{2}
\\
& +
c(\varepsilon,s,\Lambda)
[\;
(\int_{[\gamma>0]}\vert \nabla A \vert^{2}d\mu)^{3}
+
s^{6}\Lambda^{6}
\;]
\int_{[\gamma>0]} \Vert f \Vert^{6} \vert A \vert^{2}  d\mu,
\end{split}
\end{gather*}
where Young's inequality with $p=\frac{3}{2},\,q=3$ and $6=\frac{16}{3}+\frac{2}{3}$ was used in the last step. Hence we conclude
\begin{equation} \label{m=2,7}
\int_{\Sigma}\Vert f \Vert^{4} \vert A \vert^{6} \gamma^{s} d\mu
+
\int_{\Sigma}\Vert f \Vert^{6} \vert A \vert^{8}\gamma^{s} d\mu
\leq c_{2}
\end{equation}
for
$
\int_{[\gamma>0]}\vert A \vert^{2}d\mu \leq \varepsilon_{2}
$
sufficiently small. \\Inserting and applying (\ref{m=2,6}) yields the required result for this case.
\begin{itemize}
\item 
$k=3 ,j=4, i=0 $
\end{itemize}
First we have
\begin{gather*}
\begin{split}
\int_{\Sigma} \Vert f \Vert^{8} \vert A \vert^{2} \vert \nabla^{2} A \vert  \vert \nabla^{4} & A \vert \gamma^{s} d\mu \\
\leq &
\varepsilon \int_{\Sigma} \Vert f \Vert^{8} \vert \nabla^{4} A \vert^{2} \gamma^{s} d\mu
+
c(\varepsilon) \int_{\Sigma} \Vert f \Vert^{8} \vert A \vert^{4} \vert \nabla^{2} A \vert^{2} \gamma^{s}d\mu.
\end{split}
\end{gather*}
Apply (\ref{m=2,6}). Next we estimate
\begin{gather*}
\begin{split}
\int_{\Sigma} \Vert f \Vert^{8}  & \vert A \vert  \vert \nabla A \vert  \vert \nabla^{2} A \vert \vert \nabla^{3} A \vert \gamma^{s} d\mu \\
\leq &
c \int_{\Sigma} \Vert f \Vert^{8} \vert A \vert^{2} \vert \nabla^{3} A \vert^{2} \gamma^{s} d\mu
+
c\int_{\Sigma} \Vert f \Vert^{8} \vert \nabla A \vert^{2} \vert \nabla^{2} A \vert^{2} \gamma^{s} d\mu \\
\leq &
c \int_{\Sigma} \Vert f \Vert^{8} \vert A \vert^{2} \vert \nabla^{3} A \vert^{2} \gamma^{s} d\mu
+
c(1+\Vert \, \Vert f \Vert^{4} \nabla A \Vert^{4}_{L^{\infty}_{\mu}([\gamma>0])})
\hspace{-2pt}
\int_{\Sigma}\vert \nabla^{2} A \vert^{2} \gamma^{s} d\mu.
\end{split}
\end{gather*} 
Apply (\ref{m=2,5}).  For the third term in this case we have
\begin{gather*}
\hspace{-250pt} 
\begin{split}
\int_{\Sigma} \Vert f \Vert^{8} \vert A \vert \vert \nabla^{2} A \vert^{3} \gamma^{s} d\mu
\end{split}
\end{gather*}
\begin{equation} \label{m=2,8}
\leq
\varepsilon \int_{\Sigma} \Vert f \Vert^{12} \vert \nabla^{2} A \vert^{4} \gamma^{s} d\mu
+
c(\varepsilon) \int_{\Sigma} \Vert f \Vert^{4} \vert A \vert^{2} \vert \nabla^{2} A \vert^{2} \gamma^{s} d\mu.
\end{equation}
For the first term of the sum above we estimate by integration by parts
\begin{gather*}
\begin{split}
\int_{\Sigma} \Vert f \Vert^{12} \vert  & \nabla^{2} A \vert^{4} \gamma^{s} d\mu\\
\leq &
c\int_{\Sigma} \Vert f \Vert^{11} \vert \nabla A \vert \vert \nabla^{2} A \vert^{3} \gamma^{s}d\mu
+
c\int_{\Sigma} \Vert f \Vert^{12} \vert \nabla A \vert \vert \nabla^{2} A \vert^{2} \vert \nabla^{3} A \vert \gamma^{s}d\mu \\
& +
cs\Lambda \int_{\Sigma} \Vert f \Vert^{12} \vert \nabla A \vert \vert \nabla^{2} A \vert^{3} \gamma^{s-1} d\mu\\
\leq &
\varepsilon \int_{\Sigma} \Vert f \Vert^{12} \vert \nabla^{2} A \vert^{4} \gamma^{s} d\mu
+
c(\varepsilon) \int_{[\gamma>0]} \Vert f \Vert^{8} \vert \nabla A \vert^{4}  d\mu \\
& +
c(\varepsilon) \int_{\Sigma} \Vert f \Vert^{12} \vert \nabla A \vert^{2} \vert \nabla^{3} A \vert^{2} \gamma^{s} d\mu \\
& +
c(\varepsilon)s^{4} \Lambda^{4} \int_{[\gamma>0]} \Vert f \Vert^{12} \vert \nabla A \vert^{4} d\mu,
\end{split}
\end{gather*}
i.e. by absorption 
\begin{gather*}
\hspace{-270pt}
\begin{split}
\int_{\Sigma} \Vert f \Vert^{12} \vert \nabla^{2} &A \vert^{4} \gamma^{s} d\mu
\end{split}
\end{gather*}
\begin{equation} \label{m=2,9}
\leq 
c \int_{\Sigma} \Vert f \Vert^{12} \vert \nabla A \vert^{2} \vert \nabla^{3} A \vert^{2} \gamma^{s} d\mu
+
c_{2}(1+\Vert \, \Vert f \Vert^{4} \nabla A \Vert^{4}_{L^{\infty}_{\mu}([\gamma>0])}).
\end{equation}
We proceed using lemma \ref{A4} to obtain
\begin{gather*}
\begin{split}
\int_{\Sigma} \Vert f \Vert^{12} \vert \nabla A \vert^{2} & \vert \nabla^{3} A \vert^{2} \gamma^{s} d\mu
=  
\int_{\Sigma}( \Vert f \Vert^{6} \vert \nabla A \vert \vert \nabla^{3} A \vert \gamma^{\frac{s}{2}})^{2} d\mu \\
\leq &
c
[
\int_{\Sigma} \Vert f \Vert^{5} \vert \nabla A \vert \vert \nabla^{3} A \vert \gamma^{\frac{s}{2}}d\mu
+
\int_{\Sigma} \Vert f \Vert^{6} \vert \nabla^{2} A \vert \vert \nabla^{3} A \vert \gamma^{\frac{s}{2}}d\mu\\
& \hspace{22pt}+
\int_{\Sigma} \Vert f \Vert^{6} \vert \nabla A \vert \vert \nabla^{4} A \vert \gamma^{\frac{s}{2}}d\mu
+
s\Lambda
\hspace{-2,2pt}
\int_{\Sigma} \Vert f \Vert^{6} \vert \nabla A \vert \vert \nabla^{3} A \vert \gamma^{\frac{s-2}{2}}d\mu\\
& \hspace{22pt} +
\int_{\Sigma} \Vert f \Vert^{6} \vert A \vert \vert \nabla A \vert \vert \nabla^{3} A \vert \gamma^{\frac{s}{2}}d\mu
\,]^{2}\\
\leq &
c \int_{[\gamma>0]} \Vert f \Vert^{2} \vert \nabla A \vert^{2} d\mu 
\int_{[\gamma>0]} \Vert f \Vert^{8} \vert \nabla^{3 }A \vert^{2} d\mu \\
& +
c \int_{[\gamma>0]} \Vert f \Vert^{8} \vert \nabla^{3} A \vert^{2} d\mu 
\int_{\Sigma} \Vert f \Vert^{4} \vert \nabla^{2}A \vert^{2} \gamma^{s} d\mu \\
& +
c \int_{[\gamma>0]} \Vert f \Vert^{4} \vert \nabla A \vert^{2} d\mu 
\int_{\Sigma} \Vert f \Vert^{8} \vert \nabla^{4} A \vert^{2} \gamma^{s} d\mu \\
& +
c s^{2} \Lambda^{2} \int_{[\gamma>0]} \Vert f \Vert^{4} \vert \nabla A \vert^{2} d\mu 
\int_{[\gamma>0]} \Vert f \Vert^{8} \vert \nabla^{3} A \vert^{2} d\mu \\
& +
c \int_{[\gamma>0]}  \vert A \vert^{2} d\mu 
\int_{\Sigma} \Vert f \Vert^{12} \vert \nabla A \vert^{2} \vert \nabla^{3} A \vert^{2} \gamma^{s}d\mu,
\end{split}
\end{gather*}
\vspace{-3pt}
i.e. by absorption, since $ \int_{[\gamma>0]}  \vert A \vert^{2} d\mu \leq \varepsilon_{2},$
\begin{gather*}
\hspace{-186pt}
\begin{split}
\int_{\Sigma} \Vert f \Vert^{12} \vert \nabla A \vert^{2} &  \vert \nabla^{3} A \vert^{2} \gamma^{s}d\mu\\
\leq &
c_{2}  \int_{\Sigma} \Vert f \Vert^{8} \vert \nabla^{4} A \vert^{2} \gamma^{s} d\mu 
\end{split}
\end{gather*}
\vspace{-3pt}
\begin{equation} \label{m=2,10}
\hspace{40pt} 
+
c_{2} \int_{[\gamma>0]} \Vert f \Vert^{8} \vert \nabla^{3} A \vert^{2} d\mu 
\;(
1 + 
\int_{\Sigma}  \vert \nabla^{2}A \vert^{2} \gamma^{s} d\mu 
).
\end{equation}
Now inserting (\ref{m=2,10}) in (\ref{m=2,9}) we derive
\begin{gather*}
\begin{split}
\int_{\Sigma} \Vert f \Vert^{12} \vert \nabla^{2} A \vert^{4} \gamma^{s} d\mu
\leq &
c_{2}  \int_{\Sigma} \Vert f \Vert^{8} \vert \nabla^{4} A \vert^{2} \gamma^{s} d\mu \\
& +
c_{2} 
(
1 + 
\Vert \, \Vert f \Vert^{4} \nabla A \Vert^{4}_{L^{\infty}_{\mu}([\gamma>0])}
+
\int_{[\gamma>0]} \Vert f \Vert^{8} \vert \nabla^{3} A \vert^{2} d\mu 
)
\end{split}
\end{gather*}
\begin{equation} \label{m=2,11}
\hspace{3pt} 
(
1 + 
\int_{\Sigma}  \vert \nabla^{2}A \vert^{2} \gamma^{s} d\mu 
).
\end{equation}
\newpage \hspace{-22pt}
Applying (\ref{m=2,11}) to (\ref{m=2,8}) we finally obtain  using Young's inequality
\begin{gather*}
\begin{split}
\int_{\Sigma} \Vert f \Vert^{8} \vert A \vert \vert \nabla^{2} A \vert^{3} \gamma^{s} d\mu
\leq &
\varepsilon  \int_{\Sigma} \Vert f \Vert^{8} \vert \nabla^{4} A \vert^{2} \gamma^{s} d\mu \\
& +
\hspace{-1.2pt}
c_{2} 
(
1 + 
\hspace{-1pt}
\Vert \, \Vert f \Vert^{4} \nabla A \Vert^{4}_{L^{\infty}_{\mu}([\gamma>0])}
+
\hspace{-2pt}
\int_{[\gamma>0]} \Vert f \Vert^{8} \vert \nabla^{3} A \vert^{2} d\mu 
)\\
& \hspace{24pt}
(
1 + 
\int_{\Sigma}  \vert \nabla^{2}A \vert^{2} \gamma^{s} d\mu 
)\\
& +
c(\varepsilon,c_{2}) \int_{\Sigma} \Vert f \Vert^{8} \vert A \vert^{4} \vert \nabla^{2} A \vert^{2} \gamma^{s} d\mu.
\end{split}
\end{gather*}
Hence (\ref{m=2,6}) yields an estimate realizing  the full structure of (\ref{m=2,1}).\\
It remains to estimate
\begin{gather*}
\begin{split}
\int_{\Sigma} \Vert f \Vert^{8} \vert \nabla A \vert^{2} \vert \nabla^{2} A \vert^{2} \gamma^{s} d\mu
\leq &
c
(
1
+
\Vert \, \Vert f \Vert^{4} \nabla A \Vert^{4}_{L^{\infty}_{\mu}([\gamma>0])}
)
\int_{\Sigma} \vert \nabla^{2} A \vert^{2} \gamma^{s} d\mu.
\end{split}
\end{gather*}

\begin{itemize}
\item 
$k=3 ,j=3, i=1 $
\end{itemize}
We start estimating
\begin{gather*}
\begin{split}
\int_{\Sigma}  \Vert f \Vert^{7} \vert A \vert^{2} \vert \nabla^{2} A \vert & \vert \nabla^{3} A \vert \gamma^{s} d\mu \\
\leq &
c \int_{\Sigma} \Vert f \Vert^{6} \vert \nabla^{3} A \vert^{2} \gamma^{s} d\mu
+
c \int_{\Sigma} \Vert f \Vert^{8} \vert A \vert^{4} \vert \nabla^{2} A \vert^{2} \gamma^{s} d\mu. 
\end{split}
\end{gather*}
Apply  (\ref{m=2,4}) and (\ref{m=2,6}). Next we have
\begin{gather*}
\begin{split}
\int_{\Sigma} \Vert f \Vert^{7} \vert A \vert \vert \nabla A \vert & \vert \nabla^{2} A \vert^{2} \gamma^{s} d\mu \\
\leq &
c
(
1
+
\Vert \, \Vert f \Vert^{4} \nabla A \Vert^{4}_{L^{\infty}_{\mu}([\gamma>0])}
)
\int_{\Sigma} \vert \nabla^{2} A \vert^{2} \gamma^{s} d\mu\\
& +
c\int_{\Sigma} \Vert f \Vert^{8} \vert A \vert^{4} \vert \nabla^{2} A \vert^{2} \gamma^{s} d\mu
+
c\int_{\Sigma} \Vert f \Vert^{4} \vert \nabla^{2} A \vert^{2} \gamma^{s} d\mu.
\end{split}
\end{gather*}
Apply (\ref{m=2,6}). Finally 
\begin{gather*}
\begin{split}
\int_{\Sigma} \Vert f \Vert^{7} \vert \nabla A \vert^{3} \vert \nabla^{2} A \vert \gamma^{s} d\mu
\leq &
c
(
1
+
\Vert \, \Vert f \Vert^{4} \nabla A \Vert^{4}_{L^{\infty}_{\mu}([\gamma>0])}
)
\int_{\Sigma} \vert \nabla^{2} A \vert^{2} \gamma^{s} d\mu\\
& +
c\int_{\Sigma} \Vert f \Vert^{6} \vert \nabla A \vert^{4} \gamma^{s} d\mu,
\end{split}
\end{gather*}
where we have by lemma \ref{A4}
\begin{gather*}
\begin{split}
\int_{\Sigma} \Vert f \Vert^{6} \vert \nabla A \vert^{4} &  \gamma^{s} d\mu 
= 
\int_{\Sigma} (\Vert f \Vert^{3} \vert \nabla A \vert^{2} \gamma^{\frac{s}{2}})^{2} d\mu \\
\leq &
c
[
\int_{\Sigma} \Vert f \Vert^{2} \vert \nabla A \vert^{2} \gamma^{\frac{s}{2}}d\mu
+
\int_{\Sigma} \Vert f \Vert^{3} \vert \nabla A \vert \vert \nabla^{2} A \vert \gamma^{\frac{s}{2}}d\mu\\
& \hspace{22pt} +
s \Lambda 
\int_{\Sigma} \Vert f \Vert^{3} \vert \nabla A \vert^{2} \gamma^{\frac{s-2}{2}}d\mu
+
\int_{\Sigma} \Vert f \Vert^{3} \vert A \vert \vert \nabla A \vert^{2} \gamma^{\frac{s}{2}}d\mu \;
]^{2} \\
\leq &
c_{2}(1+\int_{\Sigma}\vert \nabla^{2} A \vert^{2} \gamma^{s}d\mu)
+
c \int_{[\gamma>0]}\vert A \vert^{2} d\mu
\int_{\Sigma} \Vert f \Vert^{6} \vert \nabla A \vert^{4}  \gamma^{s} d\mu,
\end{split}
\end{gather*}
i.e. by absorption
\begin{equation} \label{m=2,12}
\int_{\Sigma} \Vert f \Vert^{6} \vert \nabla A \vert^{4}  \gamma^{s} d\mu 
\leq
c_{2}(1+\int_{\Sigma}\vert \nabla^{2} A \vert^{2} \gamma^{s}d\mu).
\end{equation}

\begin{itemize}
\item 
$k=3 ,j=2, i=2 $
\end{itemize}
By Young's inequality we derive
\begin{gather*}
\begin{split}
\int_{\Sigma} 
(
\Vert f \Vert^{6}
+ 
\Vert f \Vert^{7} \vert A \vert
)
\vert A \vert^{2} & \vert \nabla^{2} A \vert^{2} \gamma^{s} d\mu \\
\leq &
c\int_{\Sigma} \Vert f \Vert^{8} \vert A \vert^{4} \vert \nabla^{2} A \vert^{2} \gamma^{s}d\mu
+
c\int_{\Sigma} \Vert f \Vert^{4} \vert \nabla^{2} A \vert^{2} \gamma^{s} d\mu.
\end{split}
\end{gather*}
Apply (\ref{m=2,6}). For the second and last term of this case likewise
\begin{gather*}
\begin{split}
\int_{\Sigma} 
(
\Vert f \Vert^{6}
+ 
\Vert f \Vert^{7} \vert A \vert
)
\vert A \vert & \vert \nabla A \vert^{2} \vert \nabla^{2} A \vert \gamma^{s} d\mu \\
\leq &
c\int_{\Sigma} \Vert f \Vert^{8} \vert A \vert^{4}  \vert \nabla^{2} A \vert^{2} \gamma^{s} d\mu
+
c\int_{\Sigma} \Vert f \Vert^{6} \vert \nabla A \vert^{4} \gamma^{s} d\mu \\
& +
c\int_{\Sigma} \Vert f \Vert^{4} \vert \nabla^{2} A \vert^{2} \gamma^{s} d\mu.
\end{split}
\end{gather*}
Apply (\ref{m=2,6}) and (\ref{m=2,12}).
\newpage
\begin{itemize}
\item 
$k=3 ,j=1, i=3 $
\end{itemize}
Estimating by Young's inequality we have
\begin{gather*}
\begin{split}
\int_{\Sigma} 
(
\Vert f \Vert^{5} 
& +
\Vert f \Vert^{6} \vert A \vert
+
\Vert f \Vert^{7} \vert A \vert^{2}
+
\Vert f \Vert^{7} \vert \nabla A \vert
)
\vert A \vert^{2} \vert \nabla A \vert \vert \nabla^{2} A \vert \gamma^{s} d\mu \\
\leq &
c\int_{\Sigma} \Vert f \Vert^{8} \vert A \vert^{4}  \vert \nabla^{2} A \vert^{2} \gamma^{s} d\mu
+
c \int_{[\gamma>0]} \Vert f \Vert^{2} \vert \nabla A \vert^{2} d\mu \\
& +
c\int_{\Sigma} \Vert f \Vert^{4} \vert A \vert^{2} \vert \nabla A \vert^{2} \gamma^{s} d\mu
+
c\int_{\Sigma} \Vert f \Vert^{6} \vert A \vert^{4} \vert \nabla A \vert^{2} \gamma^{s}d\mu \\
& +
c\int_{\Sigma} \Vert f \Vert^{6} \vert \nabla A \vert^{4} \gamma^{s} d\mu \\
\leq &
c\int_{\Sigma} \Vert f \Vert^{8} \vert A \vert^{4}  \vert \nabla^{2} A \vert^{2} \gamma^{s} d\mu
+
c \int_{[\gamma>0]} \Vert f \Vert^{2} \vert \nabla A \vert^{2} d\mu \\
& +
c\int_{[\gamma>0]} \vert A \vert^{2} d\mu
+
c\int_{\Sigma} \Vert f \Vert^{4}\vert A \vert^{6}\gamma^{s} d\mu
+
c\int_{\Sigma} \Vert f \Vert^{6} \vert A \vert^{8} \gamma^{s}d\mu
\\
& +
c\int_{\Sigma} \Vert f \Vert^{6} \vert \nabla A \vert^{4} \gamma^{s} d\mu.
\end{split}
\end{gather*}
Apply (\ref{m=2,6}),\,(\ref{m=2,7}) and (\ref{m=2,12}).
\begin{itemize}
\item 
$k=3 ,j=0, i=4 $
\end{itemize}
\vspace{-10pt}
\begin{gather*}
\begin{split}
\int_{\Sigma} 
(
\Vert f \Vert^{4} 
& +
\Vert f \Vert^{5} \vert A \vert
+
\Vert f \Vert^{6} \vert A \vert^{2}
+
\Vert f \Vert^{6} \vert \nabla A \vert \\
& +
\Vert f \Vert^{7} \vert A \vert^{3}
+
\Vert f \Vert^{7} \vert A \vert \vert \nabla A \vert
+
\Vert f \Vert^{7} \vert \nabla^{2} A \vert
)
\vert A \vert^{3} \vert \nabla^{2} A \vert \gamma^{s} d\mu\\
\leq &
c\int_{\Sigma} \Vert f \Vert^{8} \vert A \vert^{4}  \vert \nabla^{2} A \vert^{2} \gamma^{s} d\mu
+
c\int_{[\gamma>0]} \vert A \vert^{2} d\mu 
+
c\int_{\Sigma} \Vert f \Vert^{2} \vert A \vert^{4} \gamma^{s} d\mu \\
& +
c\int_{\Sigma} \Vert f \Vert^{4} \vert A \vert^{6} \gamma^{s} d\mu
+
c\int_{\Sigma} \Vert f \Vert^{6} \vert A \vert^{8} \gamma^{s} d\mu
+
c\int_{\Sigma} \Vert f \Vert^{4} \vert A \vert^{2} \vert \nabla A \vert^{2} \gamma^{s} d\mu \\
& +
c\int_{\Sigma} \Vert f \Vert^{6} \vert A \vert^{4} \vert \nabla A \vert^{2} \gamma^{s} d\mu
+
c\int_{\Sigma} \Vert f \Vert^{6} \vert A \vert^{2} \vert \nabla^{2} A \vert^{2} \gamma^{s} d\mu\\
\leq &
c\int_{\Sigma} \Vert f \Vert^{8} \vert A \vert^{4}  \vert \nabla^{2} A \vert^{2} \gamma^{s} d\mu
+
c\int_{[\gamma>0]} \vert A \vert^{2} d\mu \\
& +
c\int_{\Sigma} \Vert f \Vert^{4} \vert A \vert^{6} \gamma^{s} d\mu
+
c\int_{\Sigma} \Vert f \Vert^{6} \vert A \vert^{8} \gamma^{s} d\mu
+
c\int_{\Sigma} \Vert f \Vert^{6}  \vert \nabla A \vert^{4} \gamma^{s} d\mu\\
& +
c\int_{\Sigma} \Vert f \Vert^{4} \vert \nabla^{2} A \vert^{2} \gamma^{s} d\mu.
\end{split}
\end{gather*}
Apply (\ref{m=2,6}), (\ref{m=2,7}) and (\ref{m=2,12}).
\begin{itemize}
\item 
$k=1 ,j=5, i=1 $
\end{itemize}
By integration by parts and Young's inequality we estimate
\begin{gather*}
\begin{split}
\int_{\Sigma} \nabla \Vert f \Vert^{8}* & \nabla^{5} A*\nabla^{2} A \gamma^{s} d\mu \\
\leq &
c(n)
\hspace{-1.5pt}
\int_{\Sigma} 
(
\Vert f \Vert^{6}
+
\Vert f \Vert^{7}\vert A \vert
)
 \vert \nabla^{2} A \vert \vert \nabla^{4} A \vert \gamma^{s}d\mu \\
& +
c(n)
\hspace{-1.5pt}
\int_{\Sigma} \Vert f \Vert^{7} \vert \nabla^{3} A \vert \vert \nabla^{4} A \vert \gamma^{s}d\mu\\
& +
c(n)s\Lambda \int_{\Sigma} \Vert f \Vert^{7} \vert \nabla^{2} A \vert \vert \nabla^{4} A \vert \gamma^{s-1}d\mu\\
\leq &
\varepsilon \int_{\Sigma} \Vert f \Vert^{8} \vert \nabla^{4} A \vert^{2} \gamma^{s} d\mu
+
c(n,\varepsilon)\int_{\Sigma} \Vert f \Vert^{4} \vert \nabla^{2} A \vert^{2} \gamma^{s}d\mu\\
& +
c(n,\varepsilon)\int_{\Sigma} \Vert f \Vert^{6} \vert A \vert^{2} \vert \nabla^{2} A \vert^{2} \gamma^{s}d\mu
+
c(n,\varepsilon)\int_{\Sigma} \Vert f \Vert^{6} \vert \nabla^{3} A \vert^{2} \gamma^{s}d\mu\\
& +
c(n,\varepsilon)s^{2}\Lambda^{2}\int_{\Sigma} \Vert f \Vert^{6} \vert \nabla^{2} A \vert^{2} \gamma^{s-2}d\mu\\
\leq &
\varepsilon \int_{\Sigma} \Vert f \Vert^{8} \vert \nabla^{4} A \vert^{2} \gamma^{s} d\mu
+
c(n,\varepsilon)\int_{\Sigma} \Vert f \Vert^{8} \vert A \vert^{4} \vert \nabla^{2} A\vert^{2} \gamma^{s} d\mu\\
& +
c(n,\varepsilon)\int_{\Sigma} \Vert f \Vert^{4} \vert \nabla^{2} A \vert^{2} \gamma^{s}d\mu
+
c(n,\varepsilon,s,\Lambda)\int_{[\gamma>0]} \Vert f \Vert^{8} \vert \nabla^{2} A \vert^{2} d\mu\\
& +
c(n,\varepsilon)\int_{\Sigma} \Vert f \Vert^{6} \vert \nabla^{3} A \vert^{2} \gamma^{s}d\mu.
\end{split}
\end{gather*}
Apply (\ref{m=2,6}) and (\ref{m=2,4}).
\begin{itemize}
\item 
$k=1 ,j=4, i= 2$
\end{itemize}
Clearly
\begin{gather*}
\begin{split}
\int_{\Sigma} 
(
\Vert f \Vert^{6}
+
\Vert f \Vert^{7}\vert A \vert
) &
\vert \nabla^{2} A \vert \vert \nabla^{4} A \vert \gamma^{s} d\mu\\
\leq &
\varepsilon \int_{\Sigma} \Vert f \Vert^{8} \vert \nabla^{4} A \vert^{2} \gamma^{s} d\mu 
+
c(\varepsilon)\int_{\Sigma} \Vert f \Vert^{4} \vert \nabla^{2} A \vert^{2} \gamma^{s} d\mu \\
& +
c(\varepsilon)\int_{\Sigma} \Vert f \Vert^{8} \vert A \vert^{4} \vert \nabla^{2} A \vert^{2} \gamma^{s} d\mu.
\end{split}
\end{gather*}
Apply (\ref{m=2,6}).
\newpage
\begin{itemize}
\item 
$k=1 ,j=3, i=3 $
\end{itemize}
By Young's inequality we have
\begin{gather*}
\begin{split}
\int_{\Sigma} 
(
\Vert f \Vert^{5}
+
\Vert f \Vert^{6}\vert A \vert
& +
\Vert f \Vert^{7} \vert A \vert^{2}
+
\Vert f \Vert^{7} \vert \nabla A \vert
)
\vert \nabla^{2} A \vert \vert \nabla^{3} A \vert \gamma^{s} d\mu\\
\leq &
c\int_{\Sigma} \Vert f \Vert^{6} \vert \nabla^{3} A \vert^{2} \gamma^{s} d\mu
+
c\int_{\Sigma} \Vert f \Vert^{4} \vert \nabla^{2} A \vert^{2} \gamma^{s} d\mu\\
& +
c\int_{\Sigma} \Vert f \Vert^{8} \vert A \vert^{4} \vert \nabla^{2} A \vert^{2} \gamma^{s} d\mu \\
& +
c
(
1
+
\Vert \, \Vert f \Vert^{4} \nabla A \Vert^{4}_{L^{\infty}_{\mu}([\gamma>0])}
)
\int_{\Sigma}  \vert \nabla^{2} A \vert^{2} \gamma^{s} d\mu.
\end{split}
\end{gather*}
Apply (\ref{m=2,4}) and (\ref{m=2,6}).

\begin{itemize}
\item 
$k=1 ,j=2, i=4 $
\end{itemize}
Finally 
\begin{gather*}
\begin{split}
\int_{\Sigma}
(
\Vert f \Vert^{4} 
& +
\Vert f \Vert^{5} \vert  A \vert
+
\Vert f \Vert^{6} \vert A \vert^{2}
+
\Vert f \Vert^{6} \vert \nabla A \vert\\
& +
\Vert f \Vert^{7}\vert A \vert^{3}
+
\Vert f \Vert^{7} \vert A \vert \vert \nabla A \vert
+
\Vert f \Vert^{7} \vert \nabla^{2} A \vert
)
\vert \nabla^{2} A \vert^{2} \gamma^{s} d\mu \\
\leq &
c \int_{\Sigma} \Vert f \Vert^{8} \vert A \vert^{4} \vert \nabla^{2} A\vert^{2} \gamma^{s} d\mu
+
c \int_{\Sigma} \Vert f \Vert^{4} \vert \nabla^{2} A \vert^{2} \gamma^{s} d\mu\\
& + 
c\Vert \, \Vert f \Vert^{4} \nabla A \Vert_{L^{\infty}_{\mu}([\gamma>0])}\;
\int_{\Sigma} \Vert f \Vert^{2} \vert \nabla^{2} A \vert^{2} \gamma^{s}  
+
\Vert f \Vert^{3} \vert A \vert \vert \nabla^{2} A \vert^{2} \gamma^{s} d\mu
\\
& +
\varepsilon \int_{\Sigma} \Vert f \Vert^{12} \vert \nabla^{2} A \vert^{4} \gamma^{s} d\mu
+
c(\varepsilon) \int_{\Sigma} \Vert f \Vert^{2} \vert \nabla^{2} A \vert^{2} \gamma^{s} d\mu\\
\leq &
c \int_{\Sigma} \Vert f \Vert^{8} \vert A \vert^{4} \vert \nabla^{2} A\vert^{2} \gamma^{s} d\mu
+
c \int_{\Sigma} \Vert f \Vert^{4} \vert \nabla^{2} A \vert^{2} \gamma^{s} d\mu\\
& + 
c_{2}
(
1
+
\Vert \, \Vert f \Vert^{4} \nabla A \Vert^{4}_{L^{\infty}_{\mu}([\gamma>0])}
)
\int_{\Sigma} \vert \nabla^{2} A \vert^{2} \gamma^{s} d\mu\\
& +
\varepsilon \int_{\Sigma} \Vert f \Vert^{12} \vert \nabla^{2} A \vert^{4} \gamma^{s} d\mu
+
c(\varepsilon) \int_{\Sigma} \Vert f \Vert^{2} \vert \nabla^{2} A \vert^{2} \gamma^{s} d\mu.
\end{split}
\end{gather*}
Apply (\ref{m=2,6}) and (\ref{m=2,11}).
\end{proof}

\section{Further estimates}
\begin{proposition} \label{9.1}
For $n,\Lambda,R,d,\tau >0$,\;$m \geq 3$ and $ s\geq 2m+4$ there exists 
\begin{gather*}
c_{3}=c_{3}(n,m,s,\Lambda,R,d,\tau )>0,
\end{gather*}
such that, if \quad
$
f: \Sigma \times [0,T) \rightarrow \mathbb{R}^{n}\setminus \lbrace \mathbb{0} \rbrace,\quad 0<T\leq \tau,
$ \\
is an inverse Willmore flow,
$\gamma=\widetilde{\gamma} \circ f$ as in $(\ref{eq13})$ and
\begin{gather*}
\begin{split}
& 
\sup_{0 \leq t <T}\Vert f \Vert_{L^{\infty}_{\mu}([\gamma>0])} \leq R,
 \\
&
\sup_{0 \leq t < T}\Vert A \Vert_{W^{m-3,\infty}_{\mu}([\gamma>0])} 
\leq
d,
\\
& 
\int^{T}_{0}\Vert\,\Vert f \Vert^{4} \nabla^{k} A \Vert^{4}_{L^{\infty}_{\mu}([\gamma>0])}dt
\leq 
d 
\quad \text{ for } k=m-2,m-1,\\
&
\sup_{0 \leq t <T} \Vert A \Vert_{W^{m-1,2}_{\mu}([\gamma>0])}
\leq
d
\quad \text{ and }\\
&
\int^{T}_{0}\int_{[\gamma>0]}\Vert f \Vert^{8}\vert \nabla^{k} A \vert^{2} d\mu\; dt
\leq
d 
\quad \text{ for } k=m,m+1,
\end{split}
\end{gather*}
we have 
\begin{gather*}
\begin{split}
\sup_{0 \leq t <T}\int_{\Sigma}\vert \nabla^{m} A \vert^{2} \gamma^{s}d\mu
 + 
\int^{T}_{0} 
\int_{\Sigma}
\Vert f \Vert^{8}\vert \nabla^{m+2} A & \vert^{2}\gamma^{s}d\mu \;dt \\
\leq &
c_{3}
(
1
+
\int_{\Sigma} \vert \nabla^{m} A \vert^{2} \gamma^{s} d\mu \lfloor_{t=0}
).
\end{split}
\end{gather*} 
\end{proposition}
\begin{proof}
For abbreviative reasons let
$
c_{m}=c_{m}
(
n,m,s,\Lambda, R,d
).
$
We will show 
\begin{gather*}
\begin{split}
\sum_{(i,j,k) \in I(m),\,j<m+4} & \int_{\Sigma}  \nabla^{i}\Vert f \Vert^{8}  *P^{j}_{k}(A)*\nabla^{m} A \,\gamma^{s}d\mu \\ 
\leq & 
\int_{\Sigma} \frac{\Vert f \Vert^{8}}{16} \vert \nabla^{m+2} A \vert^{2}\gamma^{s} d\mu \\
& +
c_{m}
(
1
+
\int_{[\gamma>0]}\Vert f \Vert^{8}\vert \nabla^{m} A \vert^{2} d\mu
+
\int_{[\gamma>0]}\Vert f \Vert^{8}\vert \nabla^{m+1} A \vert^{2} d\mu\\
& \hspace{38pt}+
\Vert\,\Vert f \Vert^{4} \nabla^{m-2} A \Vert^{4}_{L^{\infty}_{\mu}([\gamma>0])}
+
\Vert\,\Vert f \Vert^{4} \nabla^{m-1} A \Vert^{4}_{L^{\infty}_{\mu}([\gamma>0])}
)
\end{split}
\end{gather*}
\begin{equation} \label{m>=3,1}
\hspace{-58pt}
(1+\int_{\Sigma} \vert \nabla^{m} A \vert^{2}\gamma^{s} d\mu).
\end{equation}
Applying this inequality and proposition \ref{5.2} with $\varepsilon=\frac{1}{8c}$ to proposition \ref{4.3} prove the claim using Gronwall's inequality, cf. lemma \ref{Gronwall's inequality}.\\
To show (\ref{m>=3,1}) we give adequate estimates for each term of the sum, but due to the $W^{m-3,\infty}$-bounds we only have to examine the cases 
\begin{gather*}
m-2 \leq j < m+4 \text{ or }m \leq i \leq m+2.
\end{gather*}

\begin{itemize}
\item 
$k=5,j=m,i=0$
\end{itemize}
We only have to estimate the following four terms, since
\begin{gather*}
\sup_{0 \leq t < T}\Vert A \Vert_{W^{m-3,\infty}_{\mu}([\gamma>0])} 
\leq
d
\end{gather*}
by assumption. First
\begin{gather*}
\begin{split}
\int_{\Sigma} \Vert f \Vert^{8} \vert A \vert^{4} \vert \nabla^{m} A \vert^{2} \gamma^{s} d\mu 
\leq &
\Vert f \Vert^{8}_{L^{\infty}_{\mu}([\gamma>0])}\Vert A \Vert^{4}_{L^{\infty}_{\mu}([\gamma>0])}\int_{\Sigma} \vert \nabla^{m} A \vert^{2} \gamma^{s} d\mu \\
\leq & 
c_{m} \int_{\Sigma} \vert \nabla^{m} A \vert^{2} \gamma^{s} d\mu.
\end{split}
\end{gather*}
Next
\begin{gather*}
\begin{split}
\int_{\Sigma} \Vert f \Vert^{8}\vert A \vert^{3}&\vert \nabla A \vert \vert \nabla^{m-1}
A \vert \vert \nabla^{m}A \vert \gamma^{s}d\mu \\
\leq &
c \int_{\Sigma} \vert \nabla^{m} A \vert^{2} \gamma^{s} d\mu \\
& +
c\Vert f \Vert^{8}_{L^{\infty}_{\mu}([\gamma>0])}\Vert A \Vert^{6}_{L^{\infty}_{\mu}([\gamma>0])}
(
1
+
\Vert \; \Vert f \Vert^{4} \nabla A \Vert^{4}_{L^{\infty}_{\mu}([\gamma>0])}
)\\
& \hspace{80pt}
\int_{[\gamma>0]}\vert \nabla^{m-1} A \vert^{2} d\mu\\
\leq &
c_{m}
+
c_{m}
\Vert \; \Vert f \Vert^{4} \nabla A \Vert^{4}_{L^{\infty}_{\mu}([\gamma>0])}
+
c\int_{\Sigma}\vert \nabla^{m} A \vert^{2} \gamma^{s}d\mu 
\end{split}
\end{gather*}
and analogously
\begin{gather*}
\begin{split}
\int_{\Sigma} \Vert f \Vert^{8} \vert A \vert^{3}  \vert \nabla^{2} A \vert & \vert \nabla^{m-2} A \vert \vert \nabla^{m} A \vert \gamma^{s} d\mu \\ 
\leq &
c_{m}
+
c_{m}\Vert \, \Vert f \Vert^{4} \nabla^{2} A \Vert^{4}_{L^{\infty}_{\mu}([\gamma>0])}
+
c\int_{\Sigma} \vert \nabla^{m} A \vert^{2} \gamma^{s} d\mu.
\end{split}
\end{gather*}
Finally
\begin{gather*}
\begin{split}
\int_{\Sigma} \Vert f \Vert^{8} \vert A \vert^{2}& \vert \nabla A \vert^{2} \vert \nabla^{m-2} A \vert \vert \nabla^{m} A \vert \gamma^{s} d\mu \\
\leq & 
c \int_{\Sigma} \vert \nabla^{m} A \vert^{2} \gamma^{s} d\mu \\
& +
c \Vert A \Vert^{4}_{L^{\infty}_{\mu}([\gamma>0])}
\Vert \; \Vert f \Vert^{4} \nabla A \Vert^{4}_{L^{\infty}_{\mu}([\gamma>0])}
\int_{[\gamma>0]} \vert \nabla^{m-2} A \vert^{2} d\mu\\
\leq &
c_{m}\Vert \; \Vert f \Vert^{4} \nabla A \Vert^{4}_{L^{\infty}_{\mu}([\gamma>0])}
+
c \int_{\Sigma} \vert \nabla^{m} A \vert^{2} \gamma^{s} d\mu.
\end{split}
\end{gather*}

\begin{itemize}
\item 
$k=5,j=m-1,i=1$
\end{itemize}
Clearly we have
\begin{gather*}
\begin{split}
\int_{\Sigma} \Vert f \Vert^{7} \vert A \vert^{4} & \vert \nabla^{m-1} A \vert \vert \nabla^{m} A \vert \gamma^{s} d\mu \\
\leq &
c \int_{\Sigma} \vert \nabla^{m} A \vert^{2} \gamma^{s}d\mu
+
c\Vert f \Vert^{14}_{L^{\infty}_{\mu}([\gamma>0])} \Vert A \Vert^{8}_{L^{\infty}_{\mu}([\gamma>0])}\int_{[\gamma>0]}\vert \nabla^{m-1} A \vert^{2} d\mu\\
\leq &
c_{m}
+
c \int_{\Sigma} \vert \nabla^{m} A \vert \gamma^{s}d\mu.
\end{split}
\end{gather*}
The second term to be considered is estimated via
\begin{gather*}
\begin{split}
\int_{\Sigma} \Vert f \Vert^{7} \vert A \vert^{3} & \vert \nabla A \vert \vert \nabla^{m-2} A \vert \vert \nabla^{m} A \vert \gamma^{s} d\mu \\
\leq &
c \int_{\Sigma} \vert \nabla^{m} A \vert^{2} \gamma^{s}d\mu\\
& +
c\Vert f \Vert^{6}_{L^{\infty}_{\mu}([\gamma>0])}\Vert A \Vert^{6}_{L^{\infty}_{\mu}([\gamma>0])}
(
1
+
\Vert \; \Vert f \Vert^{4} \nabla A \Vert^{4}_{L^{\infty}_{\mu}([\gamma>0])}
)\\
& \hspace{80pt}
\int_{[\gamma>0]} \vert \nabla^{m-2} A \vert^{2} d\mu \\
\leq &
c_{m}
+
c_{m}\Vert \; \Vert f \Vert^{4} \nabla A \Vert^{4}_{L^{\infty}_{\mu}([\gamma>0])}
+
c \int_{\Sigma} \vert \nabla^{m} A \vert \gamma^{s}d\mu.
\end{split}
\end{gather*}

\newpage
\begin{itemize}
\item 
$k=5,j=m-2,i=2$
\end{itemize}
Clearly
\begin{gather*}
\begin{split}
\int_{\Sigma} 
(
\Vert f \Vert^{6}
+ &
\Vert f \Vert^{7}\vert A \vert
)
\vert A \vert^{4} \vert \nabla^{m-2} A \vert \vert \nabla^{m} A \vert \gamma^{s}
d\mu\\
\leq &
c \int_{\Sigma} \vert \nabla^{m} A \vert^{2} \gamma^{s} d\mu \\
& +
c
(
1
+
\Vert f\Vert^{14}_{L^{\infty}_{\mu}([\gamma>0])}
)
(
1
+
\Vert A \Vert^{10}_{L^{\infty}_{\mu}([\gamma>0])}
)
\int_{[\gamma>0]}\vert \nabla^{m-2} A \vert^{2} d\mu \\
\leq &
c_{m}
+
c \int_{\Sigma} \vert \nabla^{m} A \vert^{2} \gamma^{s} d\mu.
\end{split}
\end{gather*}

\begin{itemize}
\item 
$k=5,j=0,i=m$
\end{itemize}
Note, that it suffices to estimate, cf. corollary \ref{tabel},
\begin{gather*}
\begin{split}
\int_{\Sigma} \Vert f \Vert^{7} & \vert \nabla^{m-2} A \vert \vert A \vert^{5} \vert \nabla^{m} A \vert \gamma^{s} d\mu \\
\leq &
c \int_{\Sigma} \vert \nabla^{m} A \vert^{2} \gamma^{s} d\mu
+
c\Vert f \Vert^{14}_{L^{\infty}_{\mu}([\gamma>0])}
\Vert A \Vert^{10}_{L^{\infty}_{\mu}([\gamma>0])}
\int_{[\gamma>0]}\vert \nabla^{m-2} A \vert^{2} d\mu\\
\leq &
c \int_{\Sigma} \vert \nabla^{m} A \vert^{2} \gamma^{s} d\mu
+
c_{m}.
\end{split}
\end{gather*}

\begin{itemize}
\item 
$k=3,j=m+2,i=0$
\end{itemize}
In this case the following five terms have to be considered.
\begin{gather*}
\begin{split}
\int_{\Sigma} \Vert f  \Vert ^{8}  \vert A \vert^{2} \vert  \nabla^{m+2} A \vert \vert \nabla^{m} A & \vert \gamma^{s} d\mu \\
\leq &
\varepsilon \int_{\Sigma} \Vert f \Vert^{8} \vert \nabla^{m+2} A \vert^{2} \gamma^{s} d\mu \\
& +
c(\varepsilon)\Vert f \Vert^{8}_{L^{\infty}_{\mu}([\gamma>0])}
\Vert A \Vert^{4}_{L^{\infty}_{\mu}([\gamma>0])} \hspace{-1pt}
\int_{\Sigma} \vert \nabla^{m} A \vert^{2} \gamma^{s} d\mu.
\end{split}
\end{gather*}

\begin{gather*}
\hspace{-55pt}
\begin{split}
\int_{\Sigma} \Vert f \Vert^{8} & \vert A \vert \vert \nabla A \vert \vert \nabla^{m+1} A \vert \vert \nabla^{m} A \vert \gamma^{s}d\mu\\
\leq &
c \int_{[\gamma>0]}\Vert f \Vert^{8} \vert \nabla^{m+1} A \vert^{2} d\mu\\
& +
c\Vert A \Vert^{2}_{L^{\infty}_{\mu}([\gamma>0])}
(1
+
\Vert \;\Vert f \Vert^{4} \nabla A \Vert^{4}_{L^{\infty}_{\mu}([\gamma>0])}
)
\int_{\Sigma} \vert \nabla^{m} A \vert^{2} \gamma^{s}d\mu.
\end{split}
\end{gather*}

\begin{gather*}
\hspace{-12pt}
\begin{split}
\int_{\Sigma} \Vert f \Vert^{8} 
& (
\vert A \vert \vert \nabla^{2} A \vert
+
\vert \nabla A \vert^{2}
)
\vert \nabla^{m} A \vert^{2} \gamma^{s} d\mu\\
\leq &
c
(
1
+
\Vert f \Vert^{4}_{L^{\infty}_{\mu}([\gamma>0])}
)
(
1
+
\Vert A \Vert_{L^{\infty}_{\mu}([\gamma>0])}
) \\
& \hspace{6pt}
(
1
+
\Vert\;\Vert f \Vert^{4} \nabla A \Vert^{4}_{L^{\infty}_{\mu}([\gamma>0])}
+
\Vert\; \Vert f \Vert^{4} \nabla^{2} A \Vert^{4}_{L^{\infty}_{\mu}([\gamma>0])}
)
\int_{\Sigma} \vert \nabla^{m} A \vert^{2} \gamma^{s} d\mu.
\end{split}
\end{gather*}

\begin{gather*}
\hspace{-24pt}
\begin{split}
\int_{\Sigma} \Vert f \Vert^{8} 
(
\vert A \vert \vert \nabla^{3} A \vert
& +
\vert \nabla A \vert \vert \nabla^{2} A \vert
)
\vert \nabla^{m-1} A \vert \vert \nabla^{m} A \vert \gamma^{s} d\mu \\
\leq &
c\Vert f \Vert^{4}_{L^{\infty}_{\mu}([\gamma>0])}\Vert A \Vert_{L^{\infty}_{\mu}([\gamma>0])}
(
\Vert\;\Vert f \Vert^{4} \nabla^{m-1} A \Vert^{4}_{L^{\infty}_{\mu}([\gamma>0])}
) \\
& \hspace{6pt}
(
\int_{\Sigma} \vert \nabla^{3} A \vert^{2} \gamma^{s} d\mu
+
\int_{\Sigma} \vert \nabla^{m} A \vert^{2} \gamma^{s} d\mu
) \\
& +
c
(
1
+
\Vert \;\Vert f \Vert^{4} \nabla A \Vert^{4}_{L^{\infty}_{\mu}([\gamma>0])}
+
\Vert \; \Vert f \Vert^{4}\nabla^{2} A \Vert^{4}_{L^{\infty}_{\mu}([\gamma>0])}
)\\
& \hspace{20pt}
(
\int_{[\gamma>0]} \vert \nabla^{m-1} A \vert^{2}  d\mu
+
\int_{\Sigma} \vert \nabla^{m} A \vert^{2} \gamma^{s} d\mu
).
\end{split}
\end{gather*}

\begin{gather*}
\hspace{-66pt}
\begin{split}
\int_{\Sigma} \Vert f \Vert^{8}
(
\vert A \vert \vert \nabla^{4} A \vert
& +
\vert \nabla A \vert \vert \nabla^{3} A \vert
+
\vert \nabla^{2} A \vert^{2}
)
\vert \nabla^{m-2} A \vert \vert \nabla^{m} A \vert \gamma^{s} d\mu \\
\leq &
c (1+\Vert \;\Vert f \Vert^{4} \nabla^{m-2} A\Vert^{4}_{L^{\infty}_{\mu}([\gamma>0])}) \int_{\Sigma} \vert \nabla^{m} A \vert^{2} \gamma^{s} d\mu \\
& +
c \Vert A \Vert^{2}_{L^{\infty}_{\mu}([\gamma>0])} \int_{[\gamma>0]} \Vert f \Vert^{8} \vert \nabla^{4} A \vert^{2} d\mu\\
& +
c 
(
1
+
\Vert \;\Vert f \Vert^{4} \nabla A \Vert^{4}_{L^{\infty}_{\mu}([\gamma>0])}
)
\int_{\Sigma} \vert \nabla^{3} A \vert^{2} \gamma^{s}d\mu \\
& +
c
(
1
+
\Vert \;\Vert f \Vert^{4} \nabla^{2} A \Vert^{4}_{L^{\infty}_{\mu}([\gamma>0])}
)
\int_{[\gamma>0]} \vert \nabla^{2} A \vert^{2} d\mu.
\end{split}
\end{gather*}
These estimates suffice to establish (\ref{m>=3,1}), since $m \geq 3.$

\begin{itemize}
\item 
$k=3,j=m+1,i=1$
\end{itemize}
For the first three terms we easily derive
\begin{gather*}
\begin{split}
\int_{\Sigma} \Vert f \Vert^{7} & \vert A \vert^{2} \vert \nabla^{m+1 }A \vert \vert \nabla^{m} A \vert \gamma^{s}d\mu \\
\leq &
c \int_{\Sigma} \vert \nabla^{m} A \vert^{2} \gamma^{s} d\mu
+
c \Vert f \Vert^{6}_{L^{\infty}_{\mu}([\gamma>0])} \Vert A \Vert^{4}_{L^{\infty}_{\mu}([\gamma>0])} \int_{[\gamma>0]} \Vert f \Vert^{8} \vert \nabla^{m+1} A \vert^{2} d\mu.
\end{split}
\end{gather*}

\begin{gather*}
\hspace{-10pt}
\begin{split}
\int_{\Sigma} \Vert f \Vert^{7} & \vert A \vert \vert \nabla A \vert \vert \nabla^{m} A \vert^{2} \gamma^{s}d\mu \\
\leq &
c \Vert f \Vert^{3}_{L^{\infty}_{\mu}([\gamma>0])} \Vert A \Vert_{L^{\infty}_{\mu}([\gamma>0])}
(1+\Vert\; \Vert f \Vert^{4} \nabla A  \Vert^{4}_{L^{\infty}_{\mu}([\gamma>0])}) \int_{\Sigma} \vert \nabla^{m} A \vert^{2} \gamma^{s} d\mu. 
\end{split}
\end{gather*}

\begin{gather*}
\hspace{-116pt}
\begin{split}
\int_{\Sigma} \Vert f \Vert^{7} & \vert A \vert \vert \nabla^{2} A \vert \vert \nabla^{m-1} A \vert \vert \nabla^{m} A \vert\gamma^{s}d\mu \\
\leq &
c
(
1
+
\Vert \; \Vert f \Vert^{4} \nabla^{m-1} A\Vert^{4}_{L^{\infty}_{\mu}([\gamma>0])}
)
\int_{\Sigma} \vert \nabla^{m} A \vert^{2} \gamma^{s} d\mu\\
& +
c \Vert f \Vert^{6}_{L^{\infty}_{\mu}([\gamma>0])}
\Vert A \Vert^{2}_{L^{\infty}_{\mu}([\gamma>0])} 
\int_{[\gamma>0]} \vert \nabla^{2} A \vert^{2} d\mu.
\end{split}
\end{gather*}
For the fourth term we  use Young's inequality to obtain
\begin{gather*}
\begin{split}
\int_{\Sigma} \Vert f \Vert^{7} &  \vert \nabla A \vert^{2} \vert \nabla^{m-1}A \vert \vert \nabla^{m} A \vert \gamma^{s}d\mu \\
\leq &
c 
\hspace{-1pt}
\int_{\Sigma} \vert \nabla^{m} A \vert^{2} \gamma^{s}d \mu
+
c\Vert\; \Vert f \Vert^{4} \nabla A  \Vert^{2}_{L^{\infty}_{\mu}([\gamma>0])}
\int_{\Sigma} \Vert f \Vert^{6} \vert \nabla A \vert^{2} \vert \nabla^{m-1} A \vert^{2} \gamma^{s} d\mu.
\end{split}
\end{gather*}
To estimate the second summand above we use integration by parts to derive
\begin{gather*}
\begin{split}
\int_{\Sigma} \Vert f \Vert^{6} & \vert \nabla A \vert^{2} \vert \nabla^{m-1} A \vert^{2} \gamma^{s} d\mu \\
\leq &
c \int_{\Sigma} \Vert f \Vert^{5} \vert A \vert \vert \nabla  A \vert \vert \nabla^{m-1} A \vert^{2} \gamma^{s} d\mu
+
c\int_{\Sigma} \Vert f \Vert^{6} \vert A \vert \vert \nabla^{2} A \vert \vert \nabla^{m-1} A \vert^{2} \gamma^{s} d\mu\\
& +
c \int_{\Sigma} \Vert f \Vert^{6} \vert A \vert \vert \nabla A \vert \vert \nabla^{m-1} A \vert \vert \nabla^{m} A \vert \gamma^{s} d\mu \\
& +
cs\Lambda \int_{\Sigma} \Vert f \Vert^{6} \vert A \vert \vert \nabla A \vert \vert \nabla^{m-1} A \vert^{2} \gamma^{s-1} d\mu\\
\leq &
\varepsilon \int_{\Sigma} \Vert f \Vert^{6} \vert \nabla A \vert^{2} \vert \nabla^{m-1} A \vert^{2} \gamma^{s} d\mu 
+
c(\varepsilon) \int_{[\gamma>0]} \Vert f \Vert^{4} \vert A \vert^{2} \vert \nabla^{m-1} A \vert^{2} d\mu \\
& +
c \int_{[\gamma>0]} \Vert f \Vert^{8} \vert \nabla^{2} A \vert^{2} \vert \nabla^{m-1} A \vert^{2} d\mu
+
c(\varepsilon) \int_{\Sigma} \Vert f \Vert^{6} \vert A \vert^{2} \vert \nabla^{m} A \vert^{2} \gamma^{s} d\mu\\
& +
c(\varepsilon)s^{2} \Lambda^{2} \int_{[\gamma>0]} \Vert f \Vert^{6} \vert A \vert^{2} \vert \nabla^{m-1}A \vert^{2} d\mu,
\end{split}
\end{gather*}
i.e. by absorption and $m \geq 3$
\begin{gather*}
\begin{split}
\int_{\Sigma} \Vert f \Vert^{6} & \vert \nabla A \vert^{2} \vert \nabla^{m-1} A \vert^{2} \gamma^{s} d\mu \\
\leq &
c_{m}
+
c_{m} \Vert \;\Vert f \Vert^{4} \nabla^{m-1} A \Vert^{2}_{L^{\infty}_{\mu}([\gamma>0])}
+
c_{m} \int_{\Sigma} \vert \nabla^{m} A \vert^{2} \gamma^{s} d\mu.
\end{split}
\end{gather*}
Inserting proves a correct estimate. \\
For the fifth term to be considered we clearly have
\begin{gather*}
\begin{split}
\int_{\Sigma} \Vert f \Vert^{7} & \vert  A \vert \vert \nabla^{3} A \vert \vert \nabla^{m-2}A \vert \vert \nabla^{m} A \vert \gamma^{s}d\mu \\
\leq &
c_{m}
(
1
+
\Vert \;\Vert f \Vert^{4} \nabla^{m-2} A \Vert^{4}_{L^{\infty}_{\mu}([\gamma>0])}
)
\int_{\Sigma} \vert \nabla^{m}A \vert^{2} \gamma^{s} d\mu
+
c\int_{\Sigma} \vert \nabla^{3} A \vert^{2} \gamma^{s} d\mu.
\end{split}
\end{gather*}
For the sixth and last one we estimate
\begin{gather*}
\begin{split}
\int_{\Sigma} \Vert f \Vert& ^{7}  \vert \nabla A \vert \vert \nabla^{2} A \vert \vert \nabla^{m-2}A \vert \vert \nabla^{m} A \vert \gamma^{s}d\mu \\
\leq &
c 
\hspace{-2pt}
\int_{\Sigma} \vert \nabla^{m}A \vert^{2} \gamma^{s} d\mu 
+
c\Vert \; \Vert f \Vert^{4} \nabla^{2} A \Vert^{2}_{L^{\infty}_{\mu}([\gamma>0])} \hspace{-2pt}
\int_{\Sigma} \Vert f \Vert^{6} \vert \nabla A \vert^{2} \vert \nabla^{m-2} A \vert^{2} \gamma^{s} d\mu.
\end{split}
\end{gather*}
As above with $m-2$ instead of $m-1$ we obtain by integration by parts
\begin{gather*}
\begin{split}
\int_{\Sigma} \Vert f \Vert^{6} & \vert \nabla A \vert^{2} \vert \nabla^{m-2} A \vert^{2} \gamma^{s} d\mu \\
 \leq &
c_{m}
+
c_{m} \Vert \;\Vert f \Vert^{4} \nabla^{m-2} A \Vert^{2}_{L^{\infty}_{\mu}([\gamma>0])}
+
c_{m} \int_{[\gamma>0]} \vert \nabla^{m-1} A \vert^{2}  d\mu.
\end{split}
\end{gather*}
Inserting yields the required result.

\begin{itemize}
\item 
$k=3,j=m,i=2$
\end{itemize}
Clearly
\begin{gather*}
\begin{split}
\int_{\Sigma} 
(
\Vert f \Vert^{6}
+
\Vert f \Vert^{7}\vert A \vert 
)
\vert A \vert^{2}
\vert \nabla^{m} A \vert^{2} \gamma^{s}d\mu
\leq c_{m}\int_{\Sigma} \vert \nabla^{m} A \vert^{2} \gamma^{s} d\mu.
\end{split}
\end{gather*}
Next
\begin{gather*}
\begin{split}
\int_{\Sigma} 
(
\Vert f \Vert^{6}
& +
\Vert f \Vert^{7}\vert A \vert 
)
\vert A \vert \vert \nabla A \vert\vert \nabla^{m-1} A \vert \vert \nabla^{m} A \vert \gamma^{s}
d\mu \\
\leq &
c \int_{\Sigma} \vert \nabla^{m} A \vert^{2} \gamma^{s} d\mu
+
c_{m} 
(
1
+
\Vert \; \Vert f \Vert^{4} \nabla^{m-1} A \Vert^{4}_{L^{\infty}_{\mu}([\gamma>0])}
)
\int_{[\gamma>0]} \vert \nabla A \vert^{2} d\mu.
\end{split}
\end{gather*}
Likewise
\begin{gather*}
\begin{split}
\int_{\Sigma} 
(
\Vert f \Vert^{6}
& +
\Vert f \Vert^{7}\vert A \vert 
)
\vert A \vert  \vert \nabla^{2} A \vert \vert \nabla^{m-2} A \vert \vert \nabla^{m} A \vert \gamma^{s} d\mu \\
\leq &
c \int_{\Sigma} \vert \nabla^{m} A \vert^{2} \gamma^{s} d\mu
+
c_{m} 
(
1
+
\Vert \; \Vert f \Vert^{4} \nabla^{m-2} A \Vert^{4}_{L^{\infty}_{\mu}([\gamma>0])}
)
\int_{[\gamma>0]} \vert \nabla^{2} A \vert^{2} d\mu.
\end{split}
\end{gather*}
Finally we have
\begin{gather*}
\begin{split}
\int_{\Sigma} 
(
\Vert f \Vert^{6}
& +
\Vert f \Vert^{7}\vert A \vert 
)
\vert \nabla A \vert^{2} \vert \nabla^{m-2} A \vert \vert \nabla^{m} A \vert \gamma^{s} 
d\mu \\
\leq &
c 
\hspace{-2.1pt}
\int_{\Sigma} \vert \nabla^{m} A \vert^{2} \gamma^{s} d\mu
+
c_{m} 
(
1
+
\Vert \Vert f \Vert^{4}\nabla^{m-2} A \Vert^{4}_{L^{\infty}_{\mu}([\gamma>0])}) \hspace{-2pt}
\int_{\Sigma}  \vert \nabla A \vert^{4} \gamma^{s} d\mu
\end{split}
\end{gather*}
and by integration by parts
\begin{gather*}
\begin{split}
\int_{\Sigma} \vert \nabla A \vert^{4} \gamma^{s}d\mu
\leq &
c\int_{\Sigma} \vert A \vert \vert \nabla A \vert^{2} \vert \nabla^{2} A \vert \gamma^{s}d\mu
+
cs\Lambda\int_{\Sigma} \vert A \vert \vert \nabla A \vert^{3} \gamma^{s-1} d\mu\\
\leq &
\varepsilon \int_{\Sigma} \vert \nabla A \vert^{4} \gamma^{s}d\mu
+c(\varepsilon,c_{m}),
 \end{split}
\end{gather*}
i.e. by absorption
\begin{equation} \label{m>=3,2} 
\sup_{0 \leq t <T}\int_{\Sigma}  \vert \nabla A \vert^{4} \gamma^{s} d\mu \leq c_{m}.
\end{equation}
\begin{itemize}
\item 
$k=3,j=m-1,i=3$
\end{itemize}
\vspace{-4pt}
\begin{gather*}
\begin{split}
\int_{\Sigma} 
(
\Vert f \Vert^{5}
& +
\Vert f \Vert^{6} \vert A \vert
+
\Vert f \Vert^{7} \vert A \vert^{2}
+
\Vert f \Vert^{7} \vert \nabla A \vert
)
\vert A \vert^{2} \vert \nabla^{m-1} A \vert \vert \nabla^{m} A \vert \gamma^{s} 
d\mu \\
\leq &
c \int_{\Sigma} \vert \nabla^{m} A \vert^{2} \gamma^{s} d\mu
+
c_{m}
(
1
+
\Vert \;\Vert f \Vert^{4} \nabla A \Vert^{4}_{L^{\infty}_{\mu}([\gamma>0])}
)
\int_{[\gamma>0]} \vert \nabla^{m-1} A \vert^{2} d\mu,
\end{split}
\end{gather*}
and by (\ref{m>=3,2})
\begin{gather*}
\begin{split}
\int_{\Sigma}
(
\Vert f \Vert^{5}
& +
\Vert f \Vert^{6} \vert A \vert
+
\Vert f \Vert^{7} \vert A \vert^{2}
+
\Vert f \Vert^{7} \vert \nabla A \vert
)
\vert A \vert \vert \nabla A \vert \vert \nabla^{m-2} A \vert \nabla^{m} A \vert \gamma^{s} d\mu \\
\leq &
c \int_{\Sigma} \vert \nabla^{m} A \vert^{2} \gamma^{s} d\mu \\
& +
c_{m}
(
1
+
\Vert \; \Vert f \Vert^{4} \nabla^{m-2} A \Vert^{4}_{L^{\infty}_{\mu}([\gamma>0])}
)
(
\int_{[\gamma>0]} \vert \nabla A \vert^{2} d\mu
+
\int_{\Sigma} \vert \nabla A \vert^{4} \gamma^{s}d\mu
)\\
\leq &
c \int_{\Sigma} \vert \nabla^{m} A \vert^{2} \gamma^{s} d\mu \\
& +
c_{m}
(
1
+
\Vert \; \Vert f \Vert^{4} \nabla^{m-2} A \Vert^{4}_{L^{\infty}_{\mu}([\gamma>0])}
),
\end{split}
\end{gather*}

\newpage

\begin{itemize}
\item 
$k=3,j=m-2,i=4$
\end{itemize}
Clearly
\begin{gather*}
\begin{split}
\int_{\Sigma} 
(
\Vert f \Vert^{4} 
& +
\Vert f \Vert^{5} \vert A \vert 
+
\Vert f \Vert^{6} \vert  A \vert^{2}
+
\Vert f \Vert^{6} \vert \nabla A \vert \\
&+
\Vert f \Vert^{7} \vert A \vert^{3}
+
\Vert f \Vert^{7} \vert A \vert \vert \nabla A \vert
+
\Vert f \Vert^{7} \vert \nabla^{2} A \vert
)
\vert A \vert^{2} \vert \nabla^{m-2} A \vert \vert \nabla^{m} A \vert \gamma^{s}
d\mu  \\
\leq &
c \int_{\Sigma} \vert \nabla^{m} A \vert^{2} \gamma^{s}d\mu
+
c_{m}
(
1
+
\Vert \; \Vert f \Vert^{4} \nabla^{m-2} A\Vert^{4}_{L^{\infty}_{\mu}([\gamma>0])}
).
\end{split}
\end{gather*}

\begin{itemize}
\item 
$k=3,j=0,i=m+2$
\end{itemize}
It suffices to estimate
\begin{gather*}
\begin{split}
\int_{\Sigma}
(\,
\Vert f \Vert^{7}\vert \nabla^{m} A \vert
& +
\Vert f \Vert^{7} \vert A \vert \vert \nabla^{m-1} A \vert
+
\Vert f \Vert^{7} 
[\,
\vert \nabla A \vert
+
\vert A \vert^{2}
\,] 
\vert \nabla^{m-2} A \vert \\
& +
\Vert f \Vert^{6} \vert \nabla^{m-1} A \vert 
+
\Vert f \Vert^{6} \vert A \vert \vert \nabla^{m-2} A \vert \\
& \hspace{83pt}+
\Vert f \Vert^{5} \vert \nabla^{m-2} A \vert
\,)\;
\vert A \vert^{3} \vert \nabla^{m} A \vert \gamma^{s}
 d\mu \\
\leq &
c_{m}
+
c_{m} 
(
1
+
\Vert \, \Vert f \Vert^{4} \nabla A \Vert^{4}_{L^{\infty}_{\mu}([\gamma>0])}
)
\int_{\Sigma} \vert \nabla^{m} A \vert^{2} \gamma^{s} d\mu.
\end{split}
\end{gather*}

\begin{itemize}
\item 
$k=3,j=1,i=m+1$
\end{itemize}
Similarly
\begin{gather*}
\begin{split}
\int_{\Sigma} 
(
\Vert f \Vert^{7} & \vert \nabla^{m-1} A \vert
+
\Vert f \Vert^{7} \vert A \vert \vert \nabla^{m-2} A \vert
+
\Vert f \Vert^{6} \vert \nabla^{m-2} A \vert
)
\vert A \vert^{2} \vert \nabla A \vert \vert \nabla^{m} A \vert \gamma^{s}
d\mu \\
\leq &
c \int_{\Sigma} \vert \nabla^{m} A \vert^{2} \gamma^{s}d\mu \\
& +
c_{m}
(
1
+
\Vert \;\Vert f \Vert^{4} \nabla^{m-2} A\Vert^{4}_{L^{\infty}_{\mu}([\gamma>0])}
+
\Vert \; \Vert f \Vert^{4} \nabla^{m-1} A \Vert^{4}_{L^{\infty}_{\mu}([\gamma>0])}
).
\end{split}
\end{gather*}

\begin{itemize}
\item 
$k=3,j=2,i=m$
\end{itemize}
By Young's inequality and (\ref{m>=3,2}) we derive
\begin{gather*}
\begin{split}
\int_{\Sigma} \Vert f \Vert^{7} & \vert \nabla^{m-2} A \vert
(
\vert A \vert^{2} \vert \nabla^{2} A \vert
+
\vert A \vert \vert \nabla A \vert^{2}
)
\vert \nabla^{m} A \vert \gamma^{s}
d\mu \\
\leq &
c \int_{\Sigma} \vert \nabla^{m} A \vert^{2} \gamma^{s} d\mu \\
& +
c_{m}
(
1
+
\Vert\;\Vert f \Vert^{4} \nabla^{m-2} A \Vert^{4}_{L^{\infty}_{\mu}([\gamma>0])}
)
(
\int_{[\gamma>0]}\vert \nabla^{2} A \vert^{2} d\mu
+
\int_{\Sigma} \vert \nabla A \vert^{4} \gamma^{s}d\mu. 
) \\
\leq &
c \int_{\Sigma} \vert \nabla^{m} A \vert^{2} \gamma^{s} d\mu 
+
c_{m}
(
1
+
\Vert\;\Vert f \Vert^{4} \nabla^{m-2} A \Vert^{4}_{L^{\infty}_{\mu}([\gamma>0])}
).
\end{split}
\end{gather*}

\begin{itemize}
\item 
$k=1,j=m+3,i=1$
\end{itemize}
First we have by integration by parts
\begin{gather*}
\begin{split}
\int_{\Sigma} \nabla \Vert & f  \Vert^{8}  * \nabla^{m+3} A *\nabla^{m} A \gamma^{s} d\mu \\
\leq &
c(n)\int_{\Sigma} 
(
\Vert f \Vert^{6}
+
\Vert f \Vert^{7} \vert A \vert
)
\vert \nabla^{m+2} A \vert \vert \nabla^{m} A \vert \gamma^{s} d\mu \\
& +
c(n) \int_{\Sigma} \Vert f \Vert^{7} \vert \nabla^{m+2} A \vert \vert \nabla^{m+1} A \vert \gamma^{s}d\mu \\
& + 
c(n)s\Lambda \int_{\Sigma} \Vert f \Vert^{7} \vert \nabla^{m+2} A \vert \vert \nabla^{m} A \vert \gamma^{s-1} d\mu \\
\leq &
\varepsilon \int_{\Sigma} \Vert f \Vert^{8} \vert \nabla^{m+2} A \vert^{2} \gamma^{s} d\mu
+
c(\varepsilon,c_{m})\int_{\Sigma} \vert \nabla^{m} A \vert^{2} \gamma^{s} d\mu\\
& +
c(n,\varepsilon) \int_{\Sigma} \Vert f \Vert^{6} \vert \nabla^{m+1} A \vert^{2} \gamma^{s}d\mu
+
c(\varepsilon,c_{m}) \int_{[\gamma>0]} \Vert f \Vert^{8} \vert \nabla^{m} A \vert^{2} d\mu.
\end{split}
\end{gather*}
Corollary \ref{A7} shows
\begin{gather*}
\begin{split}
\int_{\Sigma} \Vert f \Vert^{6} \vert \nabla^{m+1} A \vert^{2} \gamma^{s} & d\mu 
\leq 
\varepsilon \int_{\Sigma} \Vert f \Vert^{8} \vert \nabla^{m+2} A \vert^{2} \gamma^{s} d\mu
+
c(\varepsilon) \int_{\Sigma} \Vert f \Vert^{4} \vert \nabla^{m} A \vert^{2} \gamma^{s} d\mu\\
& \hspace{26pt}+
c(\varepsilon)s^{2}\Lambda^{2} \int_{\Sigma} \Vert f \Vert^{6} \vert \nabla^{m} A \vert^{2} \gamma^{s-2}d\mu\\
\leq &
\varepsilon \int_{\Sigma} \Vert f \Vert^{8} \vert \nabla^{m+2} A \vert^{2} \gamma^{s} d\mu
+
c(\varepsilon,c_{m}) \int_{\Sigma} \vert \nabla^{m} A \vert^{2} \gamma^{s}d\mu
\end{split} 
\end{gather*}
\begin{equation} \label{m>=3,3}
\hspace{-32pt} +
c(\varepsilon,c_{m})\int_{[\gamma>0]} \Vert f \Vert^{8} \vert \nabla^{m} A \vert^{2} d\mu.
\end{equation}
\newpage \hspace{-22pt}
Inserting we conclude
\begin{gather*}
\begin{split}
\int_{\Sigma} \nabla \Vert f \Vert^{8} & * \nabla^{m+3} A *\nabla^{m} A \gamma^{s} d\mu \\
\leq &
\varepsilon \int_{\Sigma} \Vert f \Vert^{8} \vert \nabla^{m+2} A \vert \gamma^{s} d\mu
+
c(\varepsilon,c_{m}) \int_{\Sigma} \vert \nabla^{m} A \vert^{2} \gamma^{s}d\mu\\
& +
c(\varepsilon,c_{m})\int_{[\gamma>0]} \Vert f \Vert^{8} \vert \nabla^{m} A \vert^{2} d\mu,
\end{split}
\end{gather*}
which is the required result.

\begin{itemize}
\item 
$k=1,j=m+2,i=2$
\end{itemize}
We simply have
\begin{gather*}
\begin{split}
\int_{\Sigma} 
(
\Vert f \Vert^{6}
& +
\Vert f \Vert^{7}\vert A \vert
)
\vert \nabla^{m+2} A \vert \vert \nabla^{m} A \vert \gamma^{s} d\mu\\
\leq &
\varepsilon \int_{\Sigma} \Vert f \Vert^{8} \vert \nabla^{m+2} A \vert \gamma^{s} d\mu 
+
c(\varepsilon,c_{m})\int_{\Sigma} \vert \nabla^{m} A \vert^{2} \gamma^{s} d\mu.
\end{split}
\end{gather*}

\begin{itemize}
\item 
$k=1,j=m+1,i=3$
\end{itemize}
Apply (\ref{m>=3,3}) to
\begin{gather*}
\begin{split}
\int_{\Sigma} 
(
\Vert f \Vert^{5}
& +
\Vert f \Vert^{6} \vert A \vert
+
\Vert f \Vert^{7}\vert A \vert^{2}
+
\Vert f \Vert^{7} \vert \nabla A \vert
)
\vert \nabla^{m+1} A \vert \vert \nabla^{m} A \vert \gamma^{s} d\mu
\\
\leq &
c \int_{\Sigma} \Vert f \Vert^{6} \vert \nabla^{m+1} A \vert^{2} \gamma^{s} d\mu \\
& +
c_{m}(1+\Vert \;\Vert f \Vert^{4} \nabla A \Vert^{4}_{L^{\infty}_{\mu}([\gamma>0])})
\int_{\Sigma} \vert \nabla^{m} A \vert^{2} \gamma^{s} d\mu.
\end{split}
\end{gather*}

\begin{itemize}
\item 
$k=1,j=m,i=4$
\end{itemize}
Clearly
\begin{gather*}
\begin{split}
\int_{\Sigma} 
(
\Vert f \Vert^{4} 
& +
\Vert f \Vert^{5} \vert A \vert
+
\Vert f \Vert^{6} \vert A \vert^{2}
+
\Vert f \Vert^{6} \vert \nabla A \vert \\
& +
\Vert f \Vert^{7} \vert A \vert^{3}
+
\Vert f \Vert^{7} \vert A \vert \vert \nabla A \vert
+
\Vert f \Vert^{7} \vert \nabla^{2} A \vert
)
\vert \nabla^{m} A \vert^{2} \gamma^{s}
d\mu \\
\leq &
c_{m}
(
1
+
\Vert \;\Vert f \Vert^{4} \nabla A \Vert^{4}_{L^{\infty}_{\mu}([\gamma>0])}
+
\Vert \;\Vert f \Vert^{4} \nabla^{2} A\Vert^{4}_{L^{\infty}_{\mu}([\gamma>0])}
) 
\int_{\Sigma} \vert \nabla^{m} A \vert^{2} \gamma^{s} d\mu.
\end{split}
\end{gather*}

\begin{itemize}
\item 
$k=1,j=m-1,i=5$
\end{itemize}
Since $m\geq 3$ and by (\ref{m>=3,2}) we have
\begin{gather*}
\begin{split}
\int_{\Sigma} 
(
\Vert f \Vert^{3} 
& +
\Vert f \Vert^{4} \vert A \vert
+
\Vert f \Vert^{5} \vert A \vert^{2}
+
\Vert f \Vert^{5} \vert \nabla A \vert \\
& +
\Vert f \Vert^{6} \vert A \vert^{3}
+
\Vert f \Vert^{6} \vert A \vert \vert \nabla A \vert
+
\Vert f \Vert^{6} \vert \nabla^{2} A \vert \\
& +
\Vert f \Vert^{7} \vert A \vert^{4}
+
\Vert f \Vert^{7} \vert A \vert^{2}\vert \nabla A \vert
+
\Vert f \Vert^{7} \vert \nabla A \vert^{2}\\
& +
\Vert f \Vert^{7} \vert A \vert \vert \nabla^{2} A \vert
+
\Vert f \Vert^{7} \vert \nabla^{3} A \vert
)
\vert \nabla^{m-1} A \vert \vert \nabla^{m}A \vert \gamma^{s}
d\mu \\
\leq &
c \int_{\Sigma} \vert \nabla^{m} A \vert^{2} \gamma^{s} d\mu
+
c_{m} \int_{[\gamma>0]} \vert \nabla^{m-1} A \vert^{2} d\mu\\
& +
c_{m}
(1
+
\Vert \;\Vert f \Vert^{4} \nabla^{m-1} A \Vert^{4}_{L^{\infty}_{\mu}([\gamma>0])}
) \\
& \hspace{28pt}
(
\int_{[\gamma>0]} \vert \nabla A \vert^{2} d\mu
+
\int_{[\gamma>0]} \vert \nabla^{2} A \vert^{2} d\mu\\
& \hspace{106pt}+
\int_{\Sigma} \vert \nabla A \vert^{4} \gamma^{s} d\mu
+
\int_{\Sigma} \vert \nabla^{3} A \vert^{2} \gamma^{s}d\mu
) \\
\leq &
c_{m}
(
1
+
\Vert \;\Vert f \Vert^{4} \nabla^{m-1} A \Vert^{4}_{L^{\infty}_{\mu}([\gamma>0])}
)
(
1
+
\int_{\Sigma} \vert \nabla^{m} A \vert^{2} \gamma^{s} d\mu
).
\end{split}
\end{gather*}

\begin{itemize}
\item 
$k=1,j=m-2,i=6$
\end{itemize}
Recalling $i \leq m+2$ from (\ref{I(m)=}) we may assume $m \geq 4$ for this case.  Hence
\begin{gather*}
\begin{split}
\int_{\Sigma}
(&
\Vert f \Vert^{2} 
+
\Vert f \Vert^{3} \vert A \vert
+
\Vert f \Vert^{4} \vert A \vert^{2}
+
\Vert f \Vert^{4} \vert \nabla A \vert\\
& +
\Vert f \Vert^{5} \vert A \vert^{3} 
+
\Vert f \Vert^{5} \vert A \vert \vert \nabla A \vert
+
\Vert f \Vert^{5} \vert \nabla^{2} A \vert \\
& +
\Vert f \Vert^{6} \vert A \vert^{4}
+
\Vert f \Vert^{6}\vert A \vert^{2} \vert \nabla A \vert
+
\Vert f \Vert^{6} \vert \nabla A \vert^{2} \\
& +
\Vert f \Vert^{6} \vert A \vert \vert \nabla^{2} A \vert
+
\Vert f \Vert^{6} \vert \nabla^{3} A \vert \\
& +
\Vert f \Vert^{7} \vert A \vert^{5}
+
\Vert f \Vert^{7} \vert A \vert^{3} \vert \nabla A \vert
+
\Vert f \Vert^{7} \vert A \vert \vert \nabla A \vert^{2}
+
\Vert f \Vert^{7}\vert A \vert^{2} \vert \nabla^{2} A \vert \\
& +
\Vert f \Vert^{7} \vert \nabla A \vert \vert \nabla^{2} A \vert
+
\Vert f \Vert^{7} \vert A \vert \vert \nabla^{3} A \vert
+
\Vert f \Vert^{7} \vert \nabla^{4} A \vert
\;)\;
\vert \nabla^{m-2} A \vert \vert\nabla^{m} A \vert \gamma^{s}
d\mu \\
\leq &
c \int_{\Sigma} \vert \nabla^{m} A \vert^{2} \gamma^{s} d\mu
+
c_{m} \int_{[\gamma>0]} \vert \nabla^{m-2} A \vert^{2} \gamma^{s} d\mu\\
& +
c_{m}
(
1
+
\Vert \;\Vert f \Vert^{4} \nabla^{m-2} A \Vert^{4}_{L^{\infty}_{\mu}([\gamma>0])}
) \\
& \hspace{28pt}
(
\int_{[\gamma>0]} \vert \nabla^{2} A \vert^{2} d\mu
+
\int_{[\gamma>0]} \vert \nabla^{3} A \vert^{2} d\mu
+
\int_{\Sigma} \vert \nabla^{4} A \vert^{2} \gamma^{s}d\mu
) \\
\leq &
c_{m}
(
1
+
\Vert \;\Vert f \Vert^{4} \nabla^{m-2} A \Vert^{4}_{L^{\infty}_{\mu}([\gamma>0])}
)
(
1
+
\int_{\Sigma} \vert \nabla^{m} A \vert^{2} \gamma^{s} d\mu
),
\end{split}
\end{gather*}
where 
$
\sup_{0 \leq t <T}\Vert \nabla A \Vert_{L^{\infty}_{\mu}([\gamma>0])}\, ,\,
\sup_{0 \leq t <T}\int_{[\gamma>0]} \vert \nabla^{3} A \vert^{2}d\mu\leq d
$
was used.
\begin{itemize}
\item 
$k=1,j=2,i=m+2$
\end{itemize}
Here we simply estimate by Young's inequality
\begin{gather*}
\begin{split}
\int_{\Sigma} 
(&
\Vert f \Vert^{7} \vert \nabla^{m} A \vert
+
\Vert f \Vert^{7} \vert A \vert \vert \nabla^{m-1} A \vert
+
\Vert f \Vert^{7} 
(
\vert A \vert^{2}
+
\vert \nabla A \vert
)
\vert \nabla^{m-2} A \vert \\
& +
\Vert f \Vert^{6} \vert \nabla^{m-1} A \vert
+
\Vert f \Vert^{6} \vert A \vert \vert \nabla^{m-2} A \vert  
+
\Vert f \Vert^{5} \vert \nabla^{m-2} A \vert
\;)
\vert \nabla^{2} A \vert \vert \nabla^{m} A \vert \gamma^{s} d\mu \\
\leq &
c\int_{\Sigma} \Vert f \Vert^{6} \vert \nabla A \vert^{2} \vert \nabla^{m-2} A \vert^{2} \gamma^{s} d\mu \\
& +
c_{m}
+
c_{m} 
(
1
+
\Vert \; \Vert f \Vert^{4} \nabla^{2} A \Vert^{4}_{L^{\infty}_{\mu}([\gamma>0])}
)\int_{\Sigma} \vert \nabla^{m} A \vert^{2} \gamma^{s} d\mu.
\end{split}
\end{gather*}
If $m=3,$ apply (\ref{m>=3,2}). If $m \geq 4,$ then $\int_{\Sigma} \Vert f \Vert^{6} \vert \nabla A \vert^{2} \vert \nabla^{m-2} A \vert^{2} \gamma^{s} d\mu\leq c_{m},$ since $\sup_{0 \leq t<T} \Vert \nabla A \Vert_{L^{\infty}_{\mu}([\gamma>0])}
 \leq d$ by assumption.

\begin{itemize}
\item 
$k=1,j=3,i=m+1$
\end{itemize}
Similarly 
\begin{gather*}
\begin{split}
\int_{\Sigma} 
(
\Vert f \Vert^{7}  & \vert \nabla^{m-1} A \vert
+
\Vert f \Vert^{7} \vert A \vert \vert \nabla^{m-2}A \vert
+
\Vert f \Vert^{6} \vert \nabla^{m-2} A \vert
)
\vert \nabla^{3} A \vert \vert \nabla^{m} A \vert \gamma^{s} d\mu
\\
\leq &
c \int_{\Sigma} \vert \nabla^{m} A \vert^{2} \gamma^{s}d\mu \\
& +
c_{m}
(
1
+
\Vert \;\Vert f \Vert^{4} \nabla^{m-2} A\Vert^{4}_{L^{\infty}_{\mu}([\gamma>0])}
+
\Vert \; \Vert f \Vert^{4} \nabla^{m-1} A \Vert^{4}_{L^{\infty}_{\mu}([\gamma>0])}
) \\
& \hspace{28pt}
\int_{\Sigma} \vert \nabla^{3} A \vert^{2} \gamma^{s} d\mu.
\end{split}
\end{gather*}
Recall $m\geq 3.$
\begin{itemize}
\item 
$k=1,j=4,i=m$
\end{itemize}
Finally, and this time really finally,
\begin{gather*}
\begin{split}
\int_{\Sigma} 
\Vert f \Vert&^{7}  \vert \nabla^{m-2} A \vert 
\vert \nabla^{4} A \vert
\vert \nabla^{m} A \vert \gamma^{s}d\mu \\
\leq &
c \int_{\Sigma} \Vert f \Vert^{6} \vert \nabla^{4} A \vert^{2} \gamma^{s}d\mu
+
c
(
1
+
\Vert \; \Vert f \Vert^{4} \nabla^{m-2} A \Vert^{4}_{L^{\infty}_{\mu}([\gamma>0])}
)
\hspace{-1pt}
\int_{\Sigma} \vert \nabla^{m} A \vert^{2} \gamma^{s}d\mu.
\end{split}
\end{gather*}
If $m \geq 4,$ we have
$\int_{\Sigma} \Vert f \Vert^{6} \vert \nabla^{4} A \vert^{2} \gamma^{s}d\mu \leq c_{m}+c_{m}\int_{\Sigma} \vert \nabla^{m}A \vert^{2} \gamma^{s} d\mu.$\\
If $m=3, $ then $4=m+1$ and (\ref{m>=3,3}) yields a correct result.
\end{proof}
\newpage
\section{Interpolation inequalities}
We come to prove some interpolation inequalities, which will enable us to join the previous four sections to an induction. 
\begin{lemma} \label{lemma5.1}
For $k \in \lbrace 0,\mathbb{R}_{\geq 2} \rbrace,\;s \geq 4$ and $\Lambda>0$ there exists 
\begin{gather*}
c=c(s,k,\Lambda)>0, 
\end{gather*}
such that, if 
\quad $f:\Sigma \rightarrow \mathbb{R}^{n}$\quad 
is  a closed, immersed surface,\\
$\Phi$ a normal valued $l$-linear form along $f$ and $\gamma $ as in (\ref{eq13}), we have
\begin{gather*}
\begin{split}
\Vert \;\Vert f \Vert^{k} \Phi &\Vert^{4}_{L^{\infty}_{\mu}([\gamma=1])}\\
\leq &
c\int_{[\gamma>0]} \Vert f \Vert^{2k} \vert \Phi \vert^{2}d\mu \,
(\;
\int_{\Sigma} \Vert f \Vert^{2k}\vert \nabla^{2} \Phi \vert^{2} \gamma^{s}d\mu
+
\int_{\Sigma} \Vert f \Vert^{2k} \vert A \vert^{4}
\vert \Phi \vert^{2} \gamma^{s} d\mu \\
\end{split}
\end{gather*}
\begin{equation} \label{eq23}
 \hspace{144pt}+
\int_{[\gamma>0]}
[\;
\Vert f \Vert^{2k}
+
k\,\Vert f \Vert^{2k-4}
\;]
\vert \Phi \vert^{2}d\mu
\;).
\end{equation} 
\end{lemma}
\begin{proof}
We define 
$\Psi:=\Vert f \Vert^{k}\vert\Phi\vert \gamma^{\frac{s}{2}}.$ \\
Applying lemma \ref{A5} with
$m=2,p=4,\alpha=\frac{2}{3}$ we obtain 
\begin{gather*}
\begin{split}
\Vert \Psi \Vert_{L^{\infty}_{\mu}(\Sigma)}
\leq &
c \Vert \Psi \Vert^{\frac{1}{3}}_{L^{2}_{\mu}(\Sigma)}
(\Vert \nabla \Psi \Vert_{L^{4}_{\mu}(\Sigma)}
+
\Vert H \Psi \Vert_{L^{4}_{\mu}(\Sigma)})^{\frac{2}{3}} \\
\leq &
c \Vert \Psi \Vert^{\frac{1}{3}}_{L^{2}_{\mu}(\Sigma)}
(
\Vert \Psi \Vert^{\frac{1}{3}}_{L^{\infty}_{\mu}(\Sigma)}
\Vert \nabla^{2} \Psi \Vert^{\frac{1}{3}}_{L^{2}_{\mu}(\Sigma)}
+
\Vert H \Psi \Vert^{\frac{2}{3}}_{L^{4}_{\mu}(\Sigma)}
),
\end{split}
\end{gather*}
where proposition \ref{A6} for 
\begin{gather*}
\gamma=1 \in C^{\infty}_{0}(\Sigma),\; k=u_{1}=u_{2}=v_{1}=v_{2}=\Lambda=0
\text{ and } r=2,\,p \rightarrow\infty
\end{gather*}
was used in the last step.\\
Since 
$\Vert H \Psi  \Vert^{\frac{2}{3}}_{L^{4}_{\mu}(\Sigma)}
\leq 
\Vert \Psi \Vert^{\frac{1}{3}}_{L^{\infty}_{\mu}(\Sigma)}
\Vert\, \vert \Psi \vert^{\frac{1}{2}} 
\vert H \vert \,\Vert^{\frac{2}{3}}_{L^{4}_{\mu}(\Sigma)}$ we conclude
\begin{gather*}
\Vert \Psi \Vert^{4}_{L^{\infty}_{\mu}(\Sigma)}
\leq 
c \Vert \Psi \Vert^{2}_{L^{2}_{\mu}(\Sigma)}(\,\Vert \nabla^{2}\Psi \Vert^{2}_{L^{2}_{\mu}(\Sigma)}+\Vert \,\vert \Psi \vert^{2} \vert H \vert^{4}\Vert_{L^{1}_{\mu}(\Sigma)}\,).
\end{gather*}
Recalling (\ref{eq13}) we have for $k \geq 2$ and $c=c(s,k,\Lambda)>0$
\begin{gather*}
\begin{split}
\Vert \nabla^{2}\Psi \Vert^{2}_{L^{2}_{\mu}(\Sigma)}
= &
\int_{\Sigma} \vert \nabla^{2}(\Vert f \Vert^{k} \Phi \gamma^{\frac{s}{2}}) \vert^{2}d\mu\\
\leq &
c 
\hspace{-2pt}
\int_{[\gamma>0]} 
\hspace{-2pt} \vert\,
(\Vert f \Vert^{k-2}+\Vert f \Vert^{k-1}\vert A \vert) \vert \Phi \vert \gamma^{\frac{s}{2}}
+\Vert f \Vert^{k-1}\vert \nabla \Phi \vert \gamma^{\frac{s}{2}} \\
& \quad \quad \;
+\Vert f \Vert^{k-1} \vert \Phi \vert \gamma^{\frac{s-2}{2}} 
+
\Vert f \Vert^{k} \vert \nabla^{2} \Phi \vert \gamma^{\frac{s}{2}} \\
& \quad \quad \;
+\Vert f \Vert^{k} \vert \nabla \Phi \vert \gamma^{\frac{s-2}{2}}
+\Vert f \Vert^{k} \vert \Phi \vert(1+\vert A \vert \gamma)\gamma^{\frac{s-4}{2}}
\;\vert^{2}d\mu \\
\leq &
c \int_{\Sigma} \Vert f \Vert^{2k}\vert \nabla^{2} \Phi \vert^{2} \gamma^{s}d\mu
+
c\int_{\Sigma} (\Vert f \Vert^{2k-2}\gamma^{s} +\Vert f \Vert^{2k} \gamma^{s-2})\vert \nabla \Phi \vert^{2}d\mu \\
& +
c\int_{[\gamma>0]}[\;
\Vert f \Vert^{2k-4}\gamma^{s}
+
\Vert f \Vert^{2k-2}\gamma^{s}\vert A \vert^{2} \\
& \hspace{48pt}
+
\Vert f \Vert^{2k-2}\gamma^{s-2}
+
\Vert f \Vert^{2k}\gamma^{s-4}
+\Vert f \Vert^{2k}\gamma^{s-2}\vert A \vert^{2}\;]\vert \Phi \vert^{2}d\mu.
\end{split}
\end{gather*}
By integration by parts 
\begin{gather*}
\begin{split}
\int_{\Sigma} 
[\;
\Vert f \Vert  ^{2k-2} &
\gamma^{s} +\Vert f \Vert^{2k} \gamma^{s-2}
\;]
\vert \nabla \Phi \vert^{2}d\mu \\
\leq &
c(s,k,\Lambda)\int_{\Sigma}
[\;
\Vert f \Vert^{2k-3}\gamma^{s}
+
\Vert f \Vert^{2k-2}\gamma^{s-1}\\
& \hspace{100pt}
+
\Vert f\Vert^{2k-1}\gamma^{s-2}
+
\Vert f \Vert^{2k}\gamma^{s-3}
\;]
\vert \Phi \vert \vert \nabla \Phi \vert d\mu \\
& +
c\int_{\Sigma}
[\;
\Vert f \Vert^{2k-2} 
\gamma^{s} 
+
\Vert f \Vert^{2k} \gamma^{s-2}
\;]
\vert \Phi \vert \vert \nabla^{2} \Phi \vert d\mu \\
\leq &
\varepsilon \int_{\Sigma} 
[\;
\Vert f \Vert  ^{2k-2} 
\gamma^{s} +\Vert f \Vert^{2k} \gamma^{s-2}
\;]
\vert \nabla \Phi \vert^{2}d\mu
+
\int_{\Sigma} \Vert f \Vert^{2k}\vert \nabla^{2} \Phi \vert^{2} \gamma^{s}d\mu\\
& +
c(\varepsilon,k,s,\Lambda)
\int_{[\gamma>0]}
[\;
\Vert f \Vert^{2k-4}\gamma^{s}
+
\Vert f \Vert^{2k}\gamma^{s-4}
\;]\vert \Phi \vert^{2}d\mu.
\end{split}
\end{gather*}
Absorbing and plugging in we derive for 
$c=c(s,k,\Lambda)>0$
\begin{gather*}
\begin{split}
\Vert \nabla^{2}\Psi \Vert^{2}_{L^{2}_{\mu}(\Sigma)}
\leq &
c \int_{\Sigma} \Vert f \Vert^{2k}\vert \nabla^{2} \Phi \vert^{2} \gamma^{s}d\mu
+
c \int_{[\gamma>0]}
[\,
\Vert f \Vert^{2k}
+
\Vert f \Vert^{2k-4}\gamma^{4}
]
\vert \Phi \vert^{2} \gamma^{s-4}d\mu\\
& +
c\int_{\Sigma}
[\;
\Vert f \Vert^{2k}\gamma^{s-2}
+
\Vert f \Vert^{2k-2}\gamma^{s}
\;]
\vert A \vert^{2}\vert \Phi \vert^{2}d\mu.
\end{split}
\end{gather*}
Collecting terms we conclude
\begin{gather*}
\begin{split}
\Vert \;\Vert f \Vert^{k} \Phi \Vert^{4}_{L^{\infty}_{\mu}([\gamma=1])}
\leq &
\Vert \Psi \Vert^{4}_{L^{\infty}_{\mu}(\Sigma)} \\
\leq &
c\Vert \; \Vert f \Vert^{k} \Phi \gamma^{\frac{s}{2}} \Vert^{2}_{L^{2}_{\mu}(\Sigma)}\\
& \hspace{50pt}
(\;
\int_{\Sigma} \Vert f \Vert^{2k}\vert \nabla^{2} \Phi \vert^{2} \gamma^{s}d\mu\\
& \hspace{54pt}+
\int_{[\gamma>0]}
[\;
\Vert f \Vert^{2k}
+
\Vert f \Vert^{2k-4}\gamma^{4}
\;]
\vert \Phi \vert^{2}\gamma^{s-4}d\mu\\
& \hspace{54pt}+
\int_{\Sigma}
[\;
\Vert f \Vert^{2k}\gamma^{s-2}
+
\Vert f \Vert^{2k-2}\gamma^{s}
\;]
\vert A \vert^{2}\vert \Phi \vert^{2}d\mu \\
&  \hspace{54pt} +
\int_{\Sigma} \Vert f \Vert^{2k} \gamma^{s} \vert H \vert^{4}\vert \Phi \vert^{2}d\mu
\;) \\
\leq &
c\int_{[\gamma>0]} \Vert f \Vert^{2k} \vert \Phi \vert^{2}d\mu\,
(\;
\int_{\Sigma} \Vert f \Vert^{2k}\vert \nabla^{2} \Phi \vert^{2} \gamma^{s}d\mu\\
& \hspace{104pt}+
\int_{[\gamma>0]}
[\;
\Vert f \Vert^{2k}
+
\Vert f \Vert^{2k-4}
\;]
\vert \Phi \vert^{2}d\mu\\
&  \hspace{104pt} +
\int_{\Sigma} \Vert f \Vert^{2k}  \vert A \vert^{4}\vert \Phi \vert^{2}\gamma^{s}d\mu
\;).
\end{split}
\end{gather*}
The case $k=0$ follows analogously.
\end{proof}

\begin{proposition}  \label{prop5.2}
For $n,m,\Lambda,R,d>0$ and $s \geq 4$ there exist 
\begin{gather*}
\varepsilon_{3},\;c_{3}=c_{3}(m,s,\Lambda,R,d)>0
\end{gather*}
such that, if \;
$f:\Sigma  \rightarrow \mathbb{R}^{n}$ \; is a  closed immersed surface, 
$\gamma $ as in (\ref{eq13}) and 
\begin{gather*}
\begin{split}
&
\int_{[\gamma>0]} \vert A \vert^{2}d\mu \leq \varepsilon_{3},
\\
&
\Vert f \Vert_{L^{\infty}_{\mu}([\gamma>0])}\leq R,
\\
&
\Vert A \Vert_{W^{m,2}_{\mu}([\gamma>0])} \leq d,
\end{split} 
\end{gather*}
we have for $k=0,4$ 
\begin{gather*}
\begin{split}
\Vert \,\Vert & f  \Vert^{k}  \nabla^{m} A   \Vert^{4}_{L^{\infty}_{\mu}([\gamma=1])} 
\leq 
c_{3}(1+\int_{\Sigma} \Vert f \Vert^{2k} \vert \nabla^{m+2}A \vert^{2} \gamma^{s}d\mu).
\end{split}
\end{gather*}
\end{proposition}

\begin{proof}
Considering lemma \ref{lemma5.1} with $\Phi=\nabla^{m} A$
it suffices to estimate
\begin{gather*}
\begin{split}
\int_{\Sigma} \Vert f \Vert & ^{2k}  \vert A  \vert^{4}  \vert \nabla^{m} A \vert^{2} \gamma^{s}d\mu 
\leq 
c_{3}
(
1
+ 
\int_{\Sigma} \Vert f \Vert^{2k} \vert \nabla^{m+2}A \vert^{2} \gamma^{s}d\mu).
\end{split}
\end{gather*}
For $m=k=0$ we have by lemma \ref{A4}
\begin{gather*}
\begin{split}
\int_{\Sigma} \vert A  \vert^{6}  & \gamma^{s}  d\mu 
=  
\int_{\Sigma} (\vert A \vert^{3} \gamma^{\frac{s}{2}})^{2}d\mu \\
\leq & 
c
[
\int_{\Sigma}
\vert A \vert^{2} \vert \nabla A \vert \gamma^{\frac{s}{2}}d\mu
+
s \Lambda
\int_{\Sigma}
\vert A \vert^{3} \gamma^{\frac{s-2}{2}} d\mu
 +
\int 
\vert A \vert^{4} \gamma^{\frac{s}{2}}d\mu
]^{2} \\
\leq &
c\int_{[\gamma>0]}\vert A \vert^{2} d\mu
[
\int_{\Sigma} \vert A \vert^{2}\vert \nabla A \vert^{2}\gamma^{s} d\mu 
+
s^{2}\Lambda^{2}\int_{\Sigma} \vert A \vert^{4}\gamma^{s-2}d\mu 
+
\int_{\Sigma} \vert A \vert^{6} \gamma^{s} d\mu 
] \\
\leq &
c
\int_{\Sigma}
\vert A \vert^{2}\vert \nabla A \vert^{2}\gamma^{s} d\mu +c(\varepsilon,c_{3})
+
c(\int_{[\gamma>0]} \vert A \vert^{2} d\mu +\varepsilon) 
\int_{\Sigma} \vert A \vert^{6} \gamma^{s}d\mu.
\end{split}
\end{gather*}
Absorbing yields
\begin{gather*}
\begin{split}
\int_{\Sigma} \vert A \vert^{6} \gamma^{s}d\mu 
\leq 
c \int_{\Sigma} \vert A \vert^{2} \vert \nabla A \vert^{2} \gamma^{s} d\mu
+c_{3}
\end{split}
\end{gather*}
Since $s \geq 4,$ i.e. $\frac{s}{2}\leq s-2,$ we have
\begin{gather*}
\begin{split}
\int_{\Sigma} \vert A \vert^{2} \vert \nabla A \vert^{2} \gamma^{s} d\mu
= &
\int_{\Sigma} (\vert A \vert \vert \nabla A \vert \gamma^{\frac{s}{2}})^{2}d\mu \\
\leq &
c
[
\int_{\Sigma} \vert \nabla A \vert^{2} \gamma^{\frac{s}{2}}d\mu
+
\int_{\Sigma} \vert A \vert \vert \nabla^{2} A \vert \gamma^{\frac{s}{2}}d\mu \\
& \quad  +
s \Lambda\int_{\Sigma} \vert A \vert \vert \nabla A \vert \gamma^{\frac{s-2}{2}}d\mu
+
\int_{\Sigma} \vert A \vert^{2}\vert \nabla A \vert \gamma^{\frac{s}{2}}d\mu
]^{2} \\
\leq &
c
\int_{\Sigma} \vert \nabla^{2} A \vert^{2} \gamma^{s} d\mu  
+
c_{3}
+
c(\int_{\Sigma}\vert \nabla A \vert^{2}\gamma^{\frac{s}{2}}d\mu)^{2}\\
& +
c\int_{[\gamma>0]}\vert A \vert^{2}d\mu
\int_{\Sigma} \vert A \vert^{2} \vert \nabla A \vert^{2} \gamma^{s} d\mu.
\end{split}
\end{gather*}
Therefore by absorption 
\begin{gather*}
\begin{split}
\int_{\Sigma} \vert A \vert^{2} \vert \nabla A \vert^{2} \gamma^{s} d\mu
\leq &
c
\int_{\Sigma} \vert \nabla^{2} A \vert^{2} \gamma^{s} d\mu
+c_{3}
+
c(\int_{\Sigma}\vert \nabla A \vert^{2}\gamma^{\frac{s}{2}}d\mu)^{2}\\
\leq &
c \int_{\Sigma} \vert \nabla^{2} A \vert^{2} \gamma^{s} d\mu
+c_{3},
\end{split}
\end{gather*}
where we used
\begin{gather*}
\begin{split}
(\int_{\Sigma} \vert \nabla A \vert^{2} \gamma^{\frac{s}{2}}d\mu)^{2}
\leq &
(
c\int_{\Sigma} \vert A \vert \vert \nabla^{2} A \vert \gamma^{\frac{s}{2}}d\mu
+
cs\Lambda\int_{\Sigma} \vert A \vert \vert \nabla A \vert \gamma^{\frac{s-2}{2}}d\mu 
)^{2} \\
\leq &
c\int_{\Sigma} \vert \nabla^{2} A \vert^{2} \gamma^{s} d\mu
+
c_{3} \int_{\Sigma} \vert \nabla A \vert^{2} \gamma^{s-2}d\mu\\
\leq &
c\int_{\Sigma} \vert \nabla^{2} A \vert^{2} \gamma^{s} d\mu
+
\varepsilon (\int_{\Sigma} \vert \nabla A \vert^{2} \gamma^{\frac{s}{2}}d\mu)^{2}
+
c(\varepsilon,c_{3})
\end{split}
\end{gather*}
by integration by parts, H\"older's inequality and $\frac{s}{2}\leq s-2.$\\
We conclude for $\int_{[\gamma>0]} \vert A \vert^{2} d\mu \leq \varepsilon_{3}$ sufficiently small
\begin{gather*}
\begin{split}
\int_{\Sigma} \vert A \vert^{6}\gamma^{s}
\leq &
c_{3}(1+\int_{\Sigma} \vert \nabla^{2} A \vert^{2} \gamma^{s}d\mu).
\end{split}
\end{gather*}
For $m=0,k=4$ we refer to (\ref{eq20}).\\
For arbitrary $m\geq 1$ and $k=4$ we estimate
\begin{gather*}
\begin{split}
\int_{\Sigma} \Vert f \Vert^{8}  \vert A \vert^{4} & \vert  \nabla^{m}A \vert^{2}\gamma^{s}d\mu
= 
\int_{\Sigma} (\Vert f \Vert^{4} \vert A \vert^{2} \vert \nabla^{m}A \vert\gamma^{\frac{s}{2}})^{2}d\mu \\
\leq &
c
[
\int_{\Sigma} \Vert f \Vert^{3} \vert A \vert^{2} \vert \nabla^{m}A \vert \gamma^{\frac{s}{2}}d\mu
+
\int_{\Sigma} \Vert f \Vert^{4} \vert A \vert \vert\nabla A \vert \vert \nabla^{m}A \vert \gamma^{\frac{s}{2}}d\mu \\
& \quad \quad 
+
\int_{\Sigma} 
\Vert f \Vert^{4} \vert A \vert^{2} \vert \nabla^{m+1}A \vert \gamma^{\frac{s}{2}}d\mu
+
s\Lambda 
\int_{\Sigma}
\Vert f \Vert^{4} \vert A \vert^{2} \vert \nabla^{m} A \vert \gamma^{\frac{s-2}{2}}d\mu \\
& \quad \quad
+
\int_{\Sigma} \Vert f \Vert^{4} \vert A \vert^{3} \vert \nabla^{m} A \vert \gamma^{\frac{s}{2}}d\mu
]^{2} \\
\leq &
c \int_{[\gamma>0]}\vert A \vert^{2}d\mu 
\int_{\Sigma} \Vert f \Vert^{8} \vert A \vert^{2} \vert \nabla^{m+1} A \vert^{2} \gamma^{s}d\mu
\\
& +
c(\int_{[\gamma>0]}\vert A \vert^{2}d\mu+\varepsilon)
\int_{\Sigma} \Vert f \Vert^{8}  \vert A \vert^{4} \vert  \nabla^{m}A \vert^{2}\gamma^{s}d\mu\\
& +
c(\varepsilon,c_{3}) 
(
(\int_{[\gamma>0]}\vert A \vert^{2}d\mu)^{2}
+
(\int_{[\gamma>0]}\vert \nabla A \vert^{2}d\mu)^{2}
) \\
& 
\hspace{100pt}
\int_{[\gamma>0]} [\Vert f \Vert^{8}+\Vert f \Vert^{4}\gamma^{4}] \vert \nabla^{m} A \vert^{2} \gamma^{s-4} d\mu,
\end{split}
\end{gather*}
i.e. by absorption
\begin{gather*}
\begin{split}
\int_{\Sigma} \Vert f \Vert^{8}  \vert A \vert^{4} \vert &  \nabla^{m}A \vert^{2}\gamma^{s}d\mu 
\leq 
c \int_{\Sigma} \Vert f \Vert^{8} \vert A \vert^{2} \vert \nabla^{m+1} A \vert^{2} \gamma^{s}d\mu
+
c_{3}.
\end{split}
\end{gather*}
Since we have
\begin{gather*}
\begin{split}
\int_{\Sigma} \Vert f \Vert^{8} \vert & A \vert^{2}  \vert  \nabla^{m+1} A \vert^{2} \gamma^{s}d\mu
= 
\int_{\Sigma} (\Vert f \Vert^{4} \vert A \vert \vert \nabla^{m+1} A \vert \gamma^{\frac{s}{2}})^{2}d\mu \\
\leq &
c
[
\int_{\Sigma} \Vert f \Vert^{3} \vert A \vert \vert \nabla^{m+1} A \vert \gamma^{\frac{s}{2}}d\mu 
+
\int_{\Sigma} \Vert f \Vert^{4} \vert \nabla A \vert \vert \nabla^{m+1} A \vert \gamma^{\frac{s}{2}} d\mu \\
& \quad \quad +
\int_{\Sigma} \Vert f \Vert^{4} \vert A \vert \vert \nabla^{m+2} A \vert \gamma^{\frac{s}{2}}d\mu 
+
s \Lambda
\hspace{-2.9pt}
\int_{\Sigma} \Vert f \Vert^{4} \vert A \vert \vert \nabla^{m+1} A \vert \gamma^{\frac{s-2}{2}} d\mu\\
& \quad \quad +
\int_{\Sigma} \Vert f \Vert^{4} \vert A \vert^{2} \vert \nabla^{m+1} A \vert \gamma^{\frac{s}{2}}d\mu
]^{2} \\
\leq &
c\int_{[\gamma>0]} \vert A \vert^{2}d\mu 
\int_{\Sigma} \Vert f \Vert^{8} \vert \nabla^{m+2} A \vert^{2} \gamma^{s} d\mu \\ 
& +
c
\int_{[\gamma>0]} \vert A \vert^{2}d\mu
\int_{\Sigma} \Vert f \Vert^{8} \vert A \vert^{2} \vert \nabla^{m+1} A \vert^{2} \gamma^{s} d\mu \\
& +
c_{3}
(
\int_{[\gamma>0]} \vert A \vert^{2} d\mu 
+
\int_{[\gamma>0]} \vert \nabla A \vert^{2} d\mu
) \\
& \hspace{60pt}
(\int_{\Sigma} \Vert f \Vert^{8} \vert \nabla^{m+1} A \vert^{2} \gamma^{s-2}d\mu 
+ \hspace{-0.5pt}
\int_{\Sigma} \Vert f \Vert^{6} \vert \nabla^{m+1} A \vert^{2} \gamma^{s}d\mu 
)
\end{split}
\end{gather*}
and by corollary \ref{A7} for $p=2,u=0,v=1$ and $p=2,u=1,v=0$ 
\begin{gather*}
\begin{split}
\int_{\Sigma} \Vert f \Vert^{8} \vert \nabla^{m+1}   
A  \vert^{2}   \gamma^{s-2}d\mu  
& +
\int_{\Sigma} \Vert f \Vert^{6} \vert \nabla^{m+1} A \vert^{2} \gamma^{s}d\mu \\
\leq & 
\varepsilon\int_{\Sigma} \Vert f \Vert^{8} \vert \nabla^{m+2} A \vert^{2} \gamma^{s}d\mu \\
& \hspace{-1pt} +
c(\varepsilon,s,\Lambda) \int_{[\gamma>0]}
[
\Vert f \Vert^{8}\gamma^{s-4}
+
\Vert f \Vert^{4} \gamma^{s}
]
\vert \nabla^{m} A \vert^{2} d\mu,
\end{split}
\end{gather*}
we conclude by absorption
\begin{gather*}
\begin{split}
\int_{\Sigma}
\Vert f \Vert^{8} \vert A \vert^{2} \vert & \nabla^{m+1}  A \vert^{2}  \gamma^{s}d\mu 
\leq 
c
\int_{\Sigma} 
\Vert f \Vert^{8} \vert \nabla^{m+2} A \vert^{2} \gamma^{s}d\mu 
+
c_{3}.
\end{split} 
\end{gather*}
Consequently for $\int_{[\gamma>0]} \vert A \vert^{2} d\mu \leq \varepsilon_{3}$ sufficiently small
\begin{gather*}
\begin{split}
\int_{\Sigma}
\Vert f \Vert^{8}  \vert A \vert^{4} \vert   \nabla^{m}A \vert^{2} & \gamma^{s}d\mu 
\leq   
c_{3}
(
1
+
\int_{\Sigma}
\Vert f \Vert^{8} \vert \nabla^{m+2} A \vert \gamma^{s}d\mu 
).
\end{split}
\end{gather*}
For arbitrary $m\geq 1 $ and $k=0$ we have
\begin{gather*}
\begin{split}
\int_{\Sigma} \vert A \vert^{4} \vert \nabla^{m} A &\vert^{2}  \gamma^{s} d\mu 
=
\int_{\Sigma} (\vert A \vert^{2} \vert \nabla^{m} A \vert \gamma^{\frac{s}{2}})^{2} d\mu \\
\leq &
c
[
\int_{\Sigma} \vert A \vert \vert \nabla A \vert \vert \nabla^{m} A \vert \gamma^{\frac{s}{2}} d\mu
+
\int_{\Sigma} \vert A \vert^{2} \vert \nabla^{m+1} A \vert \gamma^{\frac{s}{2}}d\mu \\
& \quad \quad+
s \Lambda \int_{\Sigma} \vert A \vert^{2} \vert \nabla^{m} A \vert \gamma^{\frac{s-2}{2}}d\mu
+
\int_{\Sigma} \vert A \vert^{3} \vert \nabla^{m} A \vert \gamma^{\frac{s}{2}}d\mu
]^{2} \\
\leq &
c \int_{[\gamma>0]} \vert A \vert^{2} d\mu
\int_{\Sigma} \vert A \vert^{2} \vert \nabla^{m+1}
A \vert^{2} \gamma^{s} d\mu\\
& +
c
(
\int_{[\gamma>0]} \vert A \vert^{2} d\mu 
+
\varepsilon
)
\int_{\Sigma} \vert A \vert^{4} \vert \nabla^{m} A \vert^{2}  \gamma^{s} d\mu \\
& +
c(\varepsilon,c_{3})
(
(\int_{[\gamma>0]} \vert A \vert^{2} d\mu)^{2}
+
(\int_{[\gamma>0]} \vert \nabla A \vert^{2} d\mu)^{2}
) 
\int_{[\gamma>0]} \vert \nabla^{m} A \vert^{2} d\mu,
\end{split}
\end{gather*}
i.e. by absorption
$\int_{\Sigma} \vert A \vert^{4} \vert \nabla^{m} A \vert^{2}  \gamma^{s} d\mu 
\leq 
c\int_{\Sigma} \vert A \vert^{2} \vert \nabla^{m+1} A \vert^{2} \gamma^{s} d\mu 
+
c_{3}.$
Since 
\begin{gather*}
\begin{split}
\int_{\Sigma} \vert A \vert^{2} & \vert  \nabla^{m+1} A \vert^{2}  \gamma^{s} d\mu 
= 
\int_{\Sigma} (\vert A \vert \vert \nabla^{m+1} A \vert \gamma^{\frac{s}{2}})^{2} d\mu \\
\leq &
c \hspace{-1pt}\int_{[\gamma>0]}\vert A \vert^{2}d\mu 
\int_{\Sigma} \vert \nabla^{m+2} A \vert^{2} \gamma^{s}d\mu
 +
c  \hspace{-1pt}\int_{[\gamma>0]}\vert A \vert^{2}d\mu
\int_{\Sigma} \vert A \vert^{2} \vert \nabla^{m+1} A \vert^{2}  \gamma^{s} d\mu \\
& +
c_{3}
(
\int_{[\gamma>0]} \vert A \vert^{2} d\mu 
+
\int_{[\gamma>0]} \vert \nabla A \vert^{2} d\mu
) 
\int_{\Sigma} \vert \nabla^{m+1} A \vert^{2} \gamma^{s-2}d\mu
\end{split}
\end{gather*}
and by corollary \ref{A7}
\begin{gather*}
\begin{split}
\int_{\Sigma} \vert \nabla^{m+1} A \vert^{2} \gamma^{s-2}d\mu
\leq &
\varepsilon\int_{\Sigma} \vert \nabla^{m+2} A \vert^{2} \gamma^{s} d\mu 
+
c(\varepsilon,c_{3})\int_{[\gamma>0]} \vert \nabla^{m}A \vert^{2}d\mu,
\end{split}
\end{gather*}
we obtain by absorption
\begin{gather*}
\begin{split}
\int_{\Sigma} \vert A \vert^{2} \vert \nabla^{m+1} A \vert^{2}  \gamma^{s} d\mu 
\leq &
c\int_{\Sigma} \vert \nabla^{m+2} A \vert^{2} \gamma^{s}d\mu 
+
c_{3}
\end{split}
\end{gather*}
and consequently for $\int_{[\gamma>0]} \vert A \vert^{2} d\mu \leq \varepsilon_{3}$
\begin{gather*}
\int_{\Sigma} \vert A \vert^{4} \vert \nabla^{m} A \vert^{2}  \gamma^{s} d\mu 
\leq 
c_{3}(1+\int_{\Sigma} \vert \nabla^{m+2} A \vert^{2} \gamma^{s}d\mu).
\end{gather*}
\end{proof}
\newpage
\section{The main results}
\begin{thm} \label{theoremI}
For $n,m,R,\tau>0$ and $\alpha \in \mathbb{R}_{> 0}^{m+2}$ 
there exist 
\begin{gather*}
\varepsilon=\varepsilon(n),\;
c=c(n,m,R,\alpha,\tau)>0 
\end{gather*}
such that, if \quad
$f:\Sigma \times [0,T)\longrightarrow \mathbb{R}^{n} \setminus 
\lbrace \mathbb{0} \rbrace, \quad 0<T\leq \tau,$ \\
is an inverse Willmore flow satisfying
\begin{gather*}
\begin{split}
&
\sup_{0 \leq t < T} \int_{B_{2\rho}(x_{0})}\vert A \vert^{2}d\mu \leq
\varepsilon, \\
&
\sup_{0\leq t <T} \rho^{-1}\Vert f \Vert_{L^{\infty}_{\mu}(B_{2\rho}(x_{0}))}\leq R,\\
&
\rho^{2i}\int_{B_{2\rho}(x_{0})} \vert \nabla^{i}A \vert^{2}d\mu \lfloor _{t=0} \;\leq
\alpha_{i},\quad i=1,\ldots,m+2
\end{split}
\end{gather*}
for some $x_{0} \in \mathbb{R}^{n}$ and $\rho>0$ , we have
\begin{gather*}
\sup_{0\leq t <T}\Vert \nabla^{m} A \Vert_{L^{\infty}_{\mu}(B_{\rho}(x_{0}))}
\leq \rho^{-(m+1)}c.
\end{gather*}
\end{thm}
\begin{proof}
Rescaling, cf. Lemmata \ref{A2} and \ref{A3}, we may assume $\rho=1.$ \\
Define 
\begin{gather*}
\varepsilon(n)
:= 
\min \lbrace \varepsilon_{0}(n),\varepsilon_{1}(n),\varepsilon_{2}(n),\varepsilon_{3} \rbrace, 
\end{gather*}
cf.
propositions \ref{6.1}, \ref{7.1}, \ref{8.1} and \ref{prop5.2} and choose a sequence $(\tilde{\gamma}_{i}) \subset C^{2}_{0}(\mathbb{R}^{n}),$ which satisfies
\begin{enumerate}
\item
$1\geq \tilde{\gamma}_{0}\geq \tilde{\gamma}_{1}\geq \ldots \geq \tilde{\gamma}_{m}\geq \ldots \geq 0,$
\item
$B_{2}(x_{0})\supset supp(\tilde{\gamma}_{0}) \supset [\tilde{\gamma}_{0}=1] \supset supp(\tilde{\gamma}_{1})  \supset [\tilde{\gamma}_{1}=1] \supset \ldots \supset B_{1}(x_{0}),$
\item
$\Vert \tilde{\gamma}_{i} \Vert_{C^{2}(\mathbb{R}^{n})}\leq c_{i}$ for all $i \geq 0.$
\end{enumerate}
Let $\gamma_{i}:=\tilde{\gamma}_{i}\circ f.$
From proposition \ref{6.1} for $s=4, \gamma= \gamma_{0}$ we infer
\begin{gather*}
\int^{T}_{0}\int_{\Sigma}\Vert f \Vert^{8} \vert \nabla^{2} A \vert^{2} \gamma^{4}_{0}\,d\mu\,dt 
\leq  
c(n,R,\tau)
\end{gather*}
and hence by proposition \ref{prop5.2} for $m=0,\,k=s=4$ and $d=\varepsilon$
\begin{gather*}
\begin{split}
\int^{T}_{0}\Vert \; \Vert f \Vert^{4} A \Vert^{4}_{L^{\infty}_{\mu}([\gamma_{0}=1])}\,dt 
\leq &
c(n,R)(T+c(n,R,\tau))\\
\leq &
c(n,R,\tau).
\end{split}
\end{gather*}
Since  $[\gamma_{0}=1] \supset supp(\gamma_{1}),$ we conclude
\begin{gather*}
\begin{split}
&\int^{T}_{0}\int_{[\gamma_{1}>0]}\Vert f \Vert^{8} \vert \nabla^{2} A \vert^{2} \,d\mu\,dt
\leq 
c(n,R,\tau),
\\
&\int^{T}_{0}\Vert \; \Vert f \Vert^{4} A \Vert^{4}_{L^{\infty}_{\mu}([\gamma_{1}>0])}\,dt 
\leq 
c(n,R,\tau).
\end{split}
\end{gather*}
Next we have by proposition \ref{7.1} for $s=6,\,\gamma=\gamma_{1}$ and $d=c(n,R,\tau)$
\begin{gather*}
\begin{split}
\sup_{0 \leq t < T}\int_{\Sigma}\vert \nabla A \vert^{2} \gamma^{6}_{1}d\mu
& + 
\int^{T}_{0} 
\int_{\Sigma}
\Vert f \Vert^{8}\vert \nabla^{3} A \vert^{2}\gamma^{6}_{1}d\mu \;dt \\
\leq &
c(n,R,d,\tau)(1+\int_{\Sigma}\vert \nabla A \vert^{2}\gamma_{1}^{6} d\mu \lfloor_{t=0}) \\
\leq &
c(n,R,\alpha_{1},\tau).
\end{split}
\end{gather*}
By proposition \ref{prop5.2} for $m=1,\,k=4,\,s=6$ and
$d=c(n,R,\alpha_{1},\tau)$ and
\begin{gather*}
[\gamma_{1}=1] \supset supp(\gamma_{2}) \supset [\gamma_{2}=1] \supset supp(\gamma_{3})
\end{gather*} 
we derive
\begin{gather*}
\begin{split}
&
\sup_{0 \leq t < T}\int_{[\gamma_{3}>0]}\vert \nabla A \vert^{2}d\mu
\leq
c(n,R,\alpha_{1},\tau),\\
&
\int^{T}_{0} \int_{[\gamma_{3}>0]} \Vert f \Vert^{8} \vert \nabla^{k+2} A \vert^{2}\,d\mu\,dt
\leq 
c(n,R,\alpha_{1},\tau),\\
&
\int^{T}_{0} \Vert \; \Vert f \Vert^{4} \nabla^{k} A \Vert^{4}_{L^{\infty}_{\mu}([\gamma_{3}>0])}\,dt
\leq 
c(n,R,\alpha_{1},\tau)
\end{split}
\end{gather*}
for $k=0,1.$ \\
Hence by proposition \ref{8.1} for $s = 8,\,\gamma=\gamma_{3}$ and $d=c(n,R,\alpha_{1},\tau)$
\begin{gather*}
\begin{split}
\sup_{0 \leq t <T}\int_{\Sigma}\vert \nabla^{2} A \vert^{2} \gamma^{8}_{3}d\mu
& + 
\int^{T}_{0} 
\int_{\Sigma}
\Vert f \Vert^{8}\vert \nabla^{4} A \vert^{2}\gamma^{8}_{3}d\mu \;dt \\
\leq &
c(n,R,d,\tau)(1+\int_{\Sigma}\vert \nabla^{2} A \vert^{2} \gamma_{3}^{8}d\mu \lfloor_{t=0})\\
\leq &
c(n,R,\alpha_{1},\alpha_{2},\tau) .
\end{split}
\end{gather*}
By proposition \ref{prop5.2} for $m=0,2, \;k=0,4,\; s=8
,\;d=c(n,R,\alpha_{1},\alpha_{2},\tau)$ and,\newpage \hspace{-22pt}
since
\begin{center}
$[\gamma_{3}=1] \supset supp(\gamma_{4}),$ 
\end{center}
we obtain both
\begin{gather*}
\begin{split}
& \hspace{-50pt}
\sup_{0 \leq t <T} \Vert A \Vert_{L^{\infty}_{\mu}[(\gamma_{4}=1])} 
\leq 
c(n,R,\alpha_{1},\alpha_{2},\tau),\\
\end{split}
\end{gather*}
\begin{equation} \label{gg}
\int^{T}_{0} \Vert \;\Vert f \Vert^{4} \nabla^{2} A \Vert^{4}_{L^{\infty}_{\mu}([\gamma_{4}=1])} \,dt
\leq 
c(n,R,\alpha_{1},\alpha_{2},\tau)
.
\end{equation}
Therefore and  by $[\gamma_{4}=1] \supset supp(\gamma_{5})$ we may assume for 
\begin{gather*}
m \geq 3,\,\gamma=\gamma_{2m-1}
\end{gather*}
inductively
\begin{gather*}
\begin{split}
& 
\sup_{0 \leq t < T}\Vert A \Vert_{W^{m-3,\infty}_{\mu}([\gamma_{2m-1}>0])} 
\leq
d, \\
& 
\int^{T}_{0}\Vert\,\Vert f \Vert^{4} \nabla^{k} A \Vert^{4}_{L^{\infty}_{\mu}([\gamma_{2m-1}>0])}dt \leq
d
\quad \text{for} \quad k=m-2 , m-1,\\&
\sup_{0 \leq t <T} \Vert A \Vert_{W^{m-1,2}_{\mu}([\gamma_{2m-1}>0])}
\leq
d,\\
&
\int^{T}_{0}\int_{[\gamma_{2m-1}>0]}\Vert f \Vert^{8}\vert \nabla^{k} A \vert^{2} d\mu\; dt  \leq 
d
\quad\text{for} \quad k=m,m+1.
\end{split}
\end{gather*}
with $d=c(n,m,R,\alpha_{1},\ldots,\alpha_{m-1},\tau).$ \\
By proposition \ref{9.1} for $s=2m+4$ we conclude 
\begin{gather*}
\begin{split}
\sup_{0 \leq t <T}\int_{\Sigma}\vert \nabla^{m} A \vert^{2} \gamma^{2m+4}_{2m-1}d\mu
& + 
\int^{T}_{0} 
\int_{\Sigma}
\Vert f \Vert^{8}\vert \nabla^{m+2} A \vert^{2}\gamma^{2m+4}_{2m-1}d\mu \;dt \\
\leq & 
c(n,m,R,d,\tau)
(
1
+
\int_{\Sigma}\vert \nabla^{m} A \vert^{2} \gamma^{2m+4}_{2m-1} d\mu \lfloor_{t=0}
)\\
\leq &
c(n,m,R,\alpha_{1},\ldots,\alpha_{m},\tau).
\end{split}
\end{gather*}
Now the claim follows by induction over $m \geq 3$ from proposition \ref{prop5.2} and 
\begin{center}
$[\gamma_{2m-1}=1] \supset supp(\gamma_{2m})\supset [\gamma_{2m}=1]\supset
supp(\gamma_{2m+1})= supp(\gamma_{2(m+1)-1}).$ 
\end{center}
\end{proof}
Next we state a version of theorem \ref{thm1.1}.
\begin{thm} \label{thm2}
For \;$0<R,\alpha_{1},\alpha_{2}<\infty$ and $n \in \mathbb{N}$ there exist 
\begin{gather*}
\delta=\delta(n),\; k=k(n),
\;c=c(n,R,\alpha_{1},\alpha_{2})>0
\end{gather*}
such that, if
\quad
$f_{0}:\Sigma \longrightarrow \mathbb{R}^{n} \setminus \lbrace \mathbb{0} \rbrace$ 
\quad
is a closed immersed surface\\
satisfying
\begin{gather*}
\begin{split}
&
\sup_{x \in \mathbb{R}^{n}} \int_{B_{2\rho}(x)} \vert A \vert^{2}d\mu \leq\delta, \\
&
\sup_{x \in \mathbb{R}^{n}}\rho^{2i}\int_{B_{2\rho}(x)} \vert \nabla^{i} A \vert^{2}d\mu 
\leq 
\alpha_{i},\quad i=1,2\;, \\ 
& 
\rho^{-1}\Vert f_{0} \Vert_{L^{\infty}_{\mu}(\Sigma)} \leq R 
\end{split}
\end{gather*}
for some $\rho>0$, there exists an inverse Willmore flow
\begin{gather*}
f:\Sigma \times [0,T)\longrightarrow \mathbb{R}^{n} \setminus \lbrace \mathbb{0} \rbrace,\quad f(\cdot,0)=f_{0} 
\end{gather*}
with $T >\rho^{-4}c$. Moreover we have the estimates
\begin{gather*}
\begin{split}
&
\sup_{0 \leq t \leq\rho^{-4}c} \rho^{-1}\Vert f \Vert_{L^{\infty}_{\mu}(\Sigma)}
\leq 2R
\quad \text{and} \quad
\sup_{0 \leq t \leq \rho^{-4}c,\, x \in \mathbb{R}^{n}} \int_{B_{2\rho}(x)} \vert A \vert^{2} d\mu\leq k  \delta.
\end{split}
\end{gather*}
\end{thm} 

\begin{proof}
Rescaling we may assume $\rho=1$.
Let $T>0$ be maximal with respect to the existence of an inverse Willmore flow
\begin{gather*}
f:\Sigma \times [0,T)\longrightarrow \mathbb{R}^{n} \setminus \lbrace \mathbb{0} \rbrace,\quad f(\cdot,0)=f_{0} 
\end{gather*}
We clearly have for some $\Gamma=\Gamma(n)>1$
\begin{gather*}
\varepsilon(t)
:=
\sup_{x \in \mathbb{R}^{n}} \int_{B_{2}(x)}\vert A \vert^{2}d\mu
\leq \Gamma\sup_{x \in \mathbb{R}^{n}} \int_{B_{1}(x)}\vert A \vert^{2}d\mu.
\end{gather*}
Let 
$\delta := \frac{\varepsilon}{3\Gamma},$ cf. theorem \ref{7.1}, and
\begin{gather*}
t_{b}:=\sup \lbrace 0 \leq t<T \mid \sup_{0 \leq \tau <t}\Vert f \Vert_{L^{\infty}_{\mu}(\Sigma)} \leq 2R \rbrace.
\end{gather*} 
For arbitrary $\lambda>0$ we have by continuity
\begin{gather*}
t_{0}:=\sup \lbrace 0 \leq t < \min (T,\lambda,t_{b}) \mid
\varepsilon(\tau)\leq 3\Gamma \delta 
\text{ for all } 0 \leq \tau <t \rbrace >0.
\end{gather*}
If $t_{0}<\min(T,\lambda,t_{b})$, then we have by definition of $\delta$ and $t_{b}$
\begin{gather*}
\begin{split}
&
\varepsilon(t_{0})=3\Gamma\delta=\varepsilon,
\\
&
\sup_{0 \leq t <t_{0},\;x \in \mathbb{R}^{n}}\int_{B_{2}(x)} \vert A \vert^{2}d\mu \leq \varepsilon,
\\
&
\sup_{0 \leq t <t_{0}} \Vert f \Vert_{L^{\infty}_{\mu}(\Sigma)} \leq 2R.
\end{split}
\end{gather*}
Hence we obtain for suitable 
\begin{gather*}
\widetilde{\gamma} \in C^{2}_{0}(B_{2}(0)) \quad\text{with}\quad \mathbb{1}_{B_{1}(0)} \leq \widetilde{\gamma} \leq \mathbb{1}_{B_{2}(0)} \quad \text{and} \quad \Vert \widetilde{\gamma} \Vert_{C^{2}(\mathbb{R}^{n})}\leq \Lambda=\Lambda(n),
\end{gather*}
by proposition \ref{6.1} choosing $s=4$ and
$\widetilde{\gamma}_{x}:=\widetilde{\gamma}(\cdot-x) \circ f$ for $x \in \mathbb{R}^{n}$
\begin{gather*}
\begin{split}
\sup_{x \in \mathbb{R}^{n}} 
\int_{B_{1}(x)}\vert A \vert^{2}d\mu \lfloor_{t=t_{0}}
\leq & 
\sup_{x \in \mathbb{R}^{n}} \int_{B_{2}(x)}\vert A \vert^{2}d\mu \lfloor_{t=0}
+
c(n)((2R)^{8}+(2R)^{4})\,t_{0} \\
\leq &
\delta
+
c(n)(R^{8}+R^{4})\,\lambda
\leq
2 \delta,
\end{split}
\end{gather*}
defining
$\lambda:=\frac{\varepsilon}{3\Gamma c(n)(R^{8}+R^{4})}.$
Hence $\varepsilon(t_{0})\leq 2\Gamma\delta<\varepsilon,$ a contradiction,
and
\begin{gather*}
t_{0}=\min(T,
\frac{c(n)}{R^{8}+R^{4}},t_{b}).
\end{gather*}
We will show, that $t_{0}=\frac{c(n)}{R^{8}+R^{4}}$ or $t_{0}=t_{b}$ imply $T>c(n,R,\alpha_{1},\alpha_{2})$, whereas
$t_{0}=T$ leads to a contradiction, what proves the claim. \\
If $t_{0} = \frac{c(n)}{R^{8}+R^{4}}<T$, then 
$T > c(n,R)$.\\
If $t_{0}=t_{b}<T$, then we have
\begin{gather*}
\begin{split}
& 
\sup_{0 \leq t < t_{0},\; x \in \mathbb{R}^{n}} \int_{B_{2}(x)} \vert A \vert^{2}d\mu \leq \varepsilon,\\
&
\sup_{0 \leq t < t_{0}} \Vert f \Vert_{L^{\infty}_{\mu}(\Sigma)}\leq 2 R, \\
&
\sup_{x \in \mathbb{R}^{n}}\int_{B_{2}(x)} \vert \nabla^{i} A \vert^{2} d\mu \lfloor _{t=0}
\leq 
\alpha_{i},\;i=1,2\;,\\
&
\Vert f(\sigma,0) \Vert \leq R,\quad \Vert f(\sigma,t_{0}) \Vert=2  R
\end{split}
\end{gather*}
for some $\sigma \in \Sigma.$ 
For $t_{0}>1$, we have $T > 1.$ Otherwise we have 
$t_{0}\leq 1$ and infer from 
Theorem \ref{7.1} for $m=0,\,\rho=\tau=1$, taking (\ref{gg}) into account
\begin{gather*}
\begin{split}
\sup_{0 \leq t <t_{0}} \Vert A \Vert_{L^{\infty}_{\mu}(\Sigma)},\;
\int^{t_{0}}_{0} \Vert \;\Vert f \Vert^{4} \nabla^{2} A \Vert^{4}_{L^{\infty}_{\mu}(\Sigma)}\,dt
\leq 
c(n,R,\alpha_{1},\alpha_{2})
\end{split}
\end{gather*}
by covering $\overline{B_{2R}(0)}\subset \cup_{i \in N=N(n,R)}B_{1}(x_{i})$. It follows 
\begin{gather*}
\begin{split}
2R
= &
\Vert f(\sigma,t_{0}) \Vert
=
\Vert f(\sigma,0) +\int^{t_{0}}_{0}\partial_{t}f(\sigma,s)ds \Vert \\
\leq &
R
+
c\,\int^{t_{0}}_{0}
[\,\Vert f \Vert^{8}(\vert \nabla^{2} A \vert + \vert A \vert^{3})](\sigma,s)ds \\
\leq &
R
+
cR^{4}\int^{t_{0}}_{0} \Vert  \;\Vert f \Vert^{4} \nabla^{2} A \Vert_{L^{\infty}_{\mu}(\Sigma)}dt
+
cR^{8}\int^{t_{0}}_{0}\Vert A \Vert^{3}_{L^{\infty}_{\mu}(\Sigma)}dt \\
\leq &
R
+
cR^{4}t_{0}^{\frac{3}{4}}[\int^{t_{0}}_{0} \Vert  \;\Vert f \Vert^{4} \nabla^{2} A \Vert^{4}_{L^{\infty}_{\mu}(\Sigma)}dt]^{\frac{1}{4}}
+
c(n,R,\alpha_{1},\alpha_{2})t_{0}\\
\leq &
\frac{3}{2}R
+
c(n,R,\alpha_{1},\alpha_{2})t_{0},
\end{split}
\end{gather*}
what shows
\begin{gather*}
c(n,R,\alpha_{1},\alpha_{2})t_{0}
\geq R.
\end{gather*}
Hence 
\begin{gather*}
T
>
\min(1,\frac{R}{c(n,R,\alpha_{1},\alpha_{2})}) =:
c(n,R,\alpha_{1},\alpha_{2}).
\end{gather*}
If $t_{0}=T$, then we have
\begin{gather*}
\begin{split}
&
t_{0}\leq \frac{c(n)}{R^{8}+R^{4}}, \\
&
\sup_{0 \leq t <t_{0},\, x\in \mathbb{R}^{n}}\int_{B_{2}(x)}\vert A \vert^{2}d\mu \leq \varepsilon,
\\
&
\sup_{0 \leq t < t_{0}} \Vert f \Vert_{L^{\infty}_{\mu}(\Sigma)}\leq 2 R 
\end{split}
\end{gather*}
and hence by theorem \ref{theoremI} for 
$
\rho=1,\,\tau=\frac{c(n)}{R^{8}+R^{4}},\;
\alpha_{i}
:=
\int_{\Sigma} \vert \nabla^{i} A \vert^{2} d\mu \lfloor _{t=0}
$
\begin{gather}\label{timecorollary1}
\sup_{0 \leq t < t_{0}}\Vert \nabla^{m} A \Vert_{L^{\infty}_{\mu}(\Sigma)}
\leq
c(n,m,R,\alpha_{0},\ldots,\alpha_{m+2}) 
<\infty .
\end{gather}
From this we will draw several conclusions. \newpage
\begin{enumerate}
\item
For the metric we derive by (\ref{formula7}) and (\ref{eq1})
\begin{gather*}
\partial_{t}g 
=
-2\langle A,\partial_{t}f \rangle
=
\Vert f \Vert^{8} (\nabla^{2}A*A+A*A*A*A)
\end{gather*}
and hence
\begin{gather*}
\begin{split}
\sup_{0 \leq t < t_{0}} \Vert \partial_{t} g \Vert_{L^{\infty}_{\mu}(\Sigma)} 
< \infty.
\end{split}
\end{gather*}
Particularly the metrics $g(t), t\in [0,t_{0}),$ are equivalent and converge uniformly to a metric $g(t) \longrightarrow :g(t_{0})$ for $t \longrightarrow t_{0}$, cf. \cite{ref3}, Lemma 14.2.
\item
Similarly we obtain for the Christoffel symbols, cf. (\ref{formula9})
\begin{gather*}
\begin{split}
\partial_{t} \Gamma^{k}_{i,j}
= &
-g^{l,k}
(
\langle \nabla_{l}A_{i,j},\partial_{t}f \rangle
-
\langle A_{i,j},\nabla_{l} \partial_{t}f\rangle \\
& \hspace{36pt}+
\langle A_{i,l},\nabla_{j} \partial_{t}f \rangle
+
\langle A_{j,l},\nabla_{i} \partial_{t} f\rangle
),
\end{split}
\end{gather*}
what shows 
\begin{gather*}
\begin{split}
\partial_{t} \Gamma=\Vert f \Vert^{8}(P^{3}_{2}(A)+P^{1}_{4}(A))
+
\nabla \Vert f \Vert^{8} *(P^{2}_{1}(A)+P^{0}_{4}(A))
\end{split}
\end{gather*}
and therefore by virtue of corollary \ref{tabel}
\begin{equation} \label{timecorollary3}
\sup_{0 \leq t <t_{0}}\Vert \nabla^{m} \partial_{t} \Gamma \Vert_{L^{\infty}_{\mu}(\Sigma)} 
< \infty.
\end{equation}
\item
Turning to the coordinate derivatives we note for $T \in T(p,q)$ 
\begin{gather*}
\nabla^{m}T=\partial^{m}T+\sum^{m}_{l=1} \sum_{k_{1}+\ldots+k_{l}+k\leq m-1}\partial^{k_{1}}\Gamma*\ldots *\partial^{k_{l}}\Gamma *\partial^{k}T,
\end{gather*}
which is easily checked by induction. This shows
\begin{gather*}
\vert \partial^{m} T \vert 
\leq 
c(n,m)(1+\vert \Gamma \vert + \ldots + \vert \partial^{m-1}\Gamma \vert)^{m} 
(\vert \nabla^{m} T \vert + \vert \partial^{m-1} T \vert + \ldots + \vert T \vert)
\end{gather*}
and inductively
\begin{gather*}
\vert \partial^{m} T \vert 
\leq
c(n,m)
(1+\vert \Gamma \vert + \ldots + \vert \partial^{m-1}\Gamma \vert)^{\sum^{m}_{i=1}i}
(\vert \nabla^{m} T \vert + \ldots + \vert T \vert).
\end{gather*}
For $T=\partial_{t}\Gamma$ we conclude by 
$\partial^{m}\partial_{t}\Gamma=\partial_{t}\partial^{m}\Gamma$ and
(\ref{timecorollary3}) inductively 
\begin{equation} \label{timecorollary4}
\sup_{0 \leq t <t_{0}} \Vert \partial^{m}\partial_{t} \Gamma \Vert_{L^{\infty}_{\mu}(\Sigma)},\;
\sup_{0 \leq t <t_{0}} \Vert \partial^{m} \Gamma \Vert_{L^{\infty}_{\mu}(\Sigma)}
< \infty.
\end{equation} 
\item
For $k+l=m \geq 0$ we have
\begin{equation} \label{timecorollary5}
\sup_{0 \leq t <t_{0}}\Vert \partial^{k} \nabla^{l} A \Vert_{L^{\infty}_{\mu}(\Sigma)},\,
\sup_{0 \leq t <t_{0}} \Vert \partial^{m+1} f \Vert_{L^{\infty}_{\mu}(\Sigma)} \,
< \infty. 
\end{equation}
This holds obviously true for $m=0$ by (\ref{timecorollary1}).
Now for $k+l=m+1$ 
\begin{gather*}
\begin{split}
\partial^{k}\nabla^{l}A-\nabla^{m+1} A
= &
\partial^{k}\nabla^{l}A-\nabla^{k}\nabla^{l} A\\
= &
\sum^{k}_{i=1}\partial^{i-1}(\partial-\nabla)\nabla^{m+1-i}A.
\end{split}
\end{gather*}
Since for any normal valued $l$-linear form $\Phi$ along $f$
\begin{gather*}
\begin{split}
0
= &
\partial_{i}\langle \partial_{j} f , \Phi_{i_{1},\ldots,i_{l}} \rangle 
= 
\langle A_{i,j},\Phi_{i_{1},\ldots,i_{l}} \rangle
+
\langle \partial_{j}f,\partial_{i}\Phi_{i_{1},\ldots,i_{l}} \rangle
\end{split}
\end{gather*}
and hence 
\begin{gather*}
\begin{split}
\nabla_{i}\Phi_{i_{1},\ldots,i_{l}}
= &
\partial^{\perp}_{i}\Phi_{i_{1},\ldots,i_{l}}
-
\sum^{l}_{k=1}\Phi_{i_{1},\ldots,k,\ldots,i_{l}}\Gamma^{k}_{i,i_{k}} \\
= &
\partial_{i}\Phi_{i_{1},\ldots,i_{l}}
+
g^{k,j}\langle A_{i,j},\Phi_{i_{1},\ldots,i_{l}} \rangle
\partial_{k}f
-
\sum^{l}_{k=1}\Phi_{i_{1},\ldots,k,\ldots,i_{l}}\Gamma^{k}_{i,i_{k}},
\end{split}
\end{gather*}
we derive
\begin{gather*}
\partial^{k}\nabla^{l}A-\nabla^{m+1} A
= 
\sum^{k}_{i=1}\partial^{i-1}(\nabla^{m+1-i}A*(A*\partial f+\Gamma)).
\end{gather*}
By induction hypothesis and (\ref{timecorollary4})
\begin{gather*}
\partial^{k}\nabla^{l}A\text{ for }k+l\leq m ,\;\partial^{i}f\text{ for }1 \leq i \leq m+1 \text{ and }
\partial^{m}\Gamma \text{ for all } m\in \mathbb{N}
\end{gather*}
are bounded. This proves
\begin{gather} \label{nnn}
\sup_{0 \leq t <t_{0}}\Vert \partial^{k} \nabla^{l} A \Vert_{L^{\infty}_{\mu}(\Sigma)}
< \infty \quad\text{for} \quad k+l=m+1.
\end{gather}
By $\partial_{i,j}f=A_{i,j}+\Gamma^{k}_{i,j}\partial_{k}f$ we then have
\begin{gather*}
\begin{split}
\partial^{m+2}f
=
\partial^{m}A+\sum_{k+l=m}\partial^{k}\Gamma\cdot\partial^{l+1}f.
\end{split}
\end{gather*}
The induction hypothesis, (\ref{timecorollary4}) and (\ref{nnn}) yield
\begin{gather*}
\sup_{0 \leq t <t_{0}}\Vert \partial^{m+2} f \Vert_{L^{\infty}_{\mu}(\Sigma)}\,
< \infty.
\end{gather*}
This proves (\ref{timecorollary5}).
\end{enumerate}
From the evolution equation (\ref{eq1}) of the inverse Willmore flow
\begin{gather*}
\partial_{t} f=-\frac{\Vert f \Vert^{8}}{2}(\Delta H + Q(A^{0})H) 
\end{gather*}
and (\ref{timecorollary5}) we infer, since $\sup_{0 \leq t < t_{0}} \Vert f \Vert_{L^{\infty}_{\mu}(\Sigma)} \leq 2R$, that for all $m \geq 0$
\begin{gather*}
\sup_{0 \leq t < t_{0}} \Vert \partial^{m}f \Vert_{L^{\infty}_{\mu}(\Sigma)},\; 
\sup_{0 \leq t < t_{0}} \Vert \partial^{m} \partial_{t} f \Vert_{L^{\infty}_{\mu}(\Sigma)}
< \infty.
\end{gather*}
Therefore $f(t)\longrightarrow :f(t_{0})$ in $C^{m}(\Sigma,\mathbb{R}^{n})$ as $t \nearrow t_{0}$ for all $m \geq 0.$\\
We conclude by what was already proven above, see the first point,
\begin{gather*}
g(t_{0})\longleftarrow g(t)=g_{f(t)}\longrightarrow g_{f(t_{0})} \quad \text{uniformly}.
\end{gather*}
Thus $g_{f(t_{0})}$ is a metric and hence
$f(t_{0}):\Sigma \longrightarrow \mathbb{R}^{n}$  a smooth immersion. 
If $f(t_{0}):\Sigma \longrightarrow \mathbb{R}^{n}\setminus \lbrace \mathbb{0} \rbrace,$ we may extend the inverse Willmore flow by 
short time existence, what shows, that 
$t_{0}=T$ could not have been the 
maximal time of lifespan of $f.$ 
This contradicts the maximality of $T.$ \\
Consequently it remains to show, that
$f(t_{0}):\Sigma \longrightarrow \mathbb{R}^{n}\setminus \lbrace \mathbb{0} \rbrace.$ \\
Otherwise there exists by continuity $\sigma \in \Sigma,$ such that
\begin{gather*}
f(\sigma,t) \longrightarrow \mathbb{0} \text{ as }t \nearrow t_{0}.
\end{gather*}
We choose a sequence $\tau_{n}\nearrow t_{0}$ satisfying
\begin{gather*}
0 \longleftarrow \Vert f(\sigma,\tau_{n}) \Vert \geq \Vert f(\sigma,t) \Vert
\quad  \text{for all} \quad \tau_{n}\leq t <t_{0}.
\end{gather*}
By the evolution equation (\ref{eq1}) it follows for $\tau_{n}\leq t <t_{0}$
\begin{gather*}
\begin{split}
\Vert f(\sigma,t)-f(\sigma,\tau_{n}) \Vert 
\leq & 
\int^{t}_{\tau_{n}} [\,\frac{\Vert f\Vert^{8}}{2}
(
\vert \Delta H \vert
+
\vert Q(A^{0})H\vert)\,](\sigma,s)
ds \\
\leq &
c\,\Vert f(\sigma,\tau_{n})\Vert^{8}
\int^{t_{0}}_{\tau_{n}} 
(
\vert \nabla^{2}A \vert
+
\vert A \vert^{3}
)
(\sigma,s)ds.
\end{split}
\end{gather*}
Letting $t \nearrow t_{0}$ we conclude
\begin{gather*}
\infty \longleftarrow \Vert f(\sigma,\tau_{n}) \Vert^{-7}
\leq 
\int^{t_{0}}_{\tau_{n}} 
(
\vert \nabla^{2}A \vert
+
\vert A \vert^{3} 
)
(\sigma,s)ds.
\end{gather*}
This shows in contradiction to (\ref{timecorollary1})
\begin{gather*}
\sup_{0 \leq t <t_{0}} 
(
\Vert \nabla^{2} A \Vert_{L^{\infty}_{\mu}(\Sigma)}
+
\Vert  A \Vert_{L^{\infty}_{\mu}(\Sigma)}
)
= \infty.
\end{gather*}
The additional estimates in theorem \ref{thm2} follow from
$T>t_{0}=\min(\lambda,t_{b})$
recalling the definitions of $t_{b}$ and $t_{0}$ taking $k(n):=3\Gamma(n).$
\end{proof}

\begin{proof}[Proof of theorem \ref{thm1.1}]
Rescaling we may assume $\rho=1.$ Let $x \in \mathbb{R}^{n}.$ \\
Using a suitable cut-off function and integration by parts one obtains
\begin{gather*}
\int_{B_{1}(x)} \vert \nabla A \vert^{2} d\mu  
\leq 
c \int_{B_{2}(x)} \vert A \vert^{2} d\mu 
+
c\int_{B_{2}(x)} \vert \nabla^{2} A \vert^{2} d\mu
\end{gather*}
for some constant $c=c(n)>0$. \\
Covering $B_{2}(x)\subseteq \cup_{i \in \Gamma=\Gamma(n)}B_{1}(x_{i})$
we obtain passing to the suprema
\begin{gather*}
\sup_{x \in \mathbb{R}^{n}}\int_{B_{1}(x)} \vert \nabla A \vert^{2} d\mu  
\leq 
c \sup_{x \in \mathbb{R}^{n}}\int_{B_{1}(x)} \vert A \vert^{2} d\mu 
+
c\sup_{x \in \mathbb{R}^{n}}\int_{B_{1}(x)} \vert \nabla^{2} A \vert^{2} d\mu.
\end{gather*}
Apply theorem \ref{thm2} with 
$
\alpha_{2}=\alpha,\;\alpha_{1}=c\delta+c\alpha_{2},\,
\rho=\frac{1}{2}.
$
\end{proof}
\begin{proof} [Proof of theorem \ref{thm1.2}]  
The proof is, due to the stronger assumptions made, a simpler version of the one of theorem \ref{thm2}. 
Rescaling we may assume $\rho=1.$ \\
Again we take $\delta:=\frac{\varepsilon}{3\Gamma}$ and  have for some $\Gamma=\Gamma(n)>0$ and arbitrary  $\lambda>0$ 
\begin{gather*}
\begin{split}
&
\varepsilon(t):=\sup_{x \in \mathbb{R}^{n}} \int_{B_{2}(x)} \vert A \vert^{2}d\mu \leq \Gamma
\sup_{x \in \mathbb{R}^{n}} \int_{B_{1}(x)} \vert A \vert^{2}d\mu,\\
&
t_{0}:=\sup \lbrace 0 \leq t < \min(T,\lambda) \mid \varepsilon(\tau)\leq 3\Gamma \delta \;\text{for all}\;0 \leq \tau < t \rbrace>0.
\end{split}
\end{gather*}
Taking $\lambda:=\frac{c(n)}{R^{8}+R^{4}}$ as in the proof of theorem \ref{thm2} we derive analogously
\begin{gather*}
t_{0}=\min(T,\frac{c(n)}{R^{8}+R^{4}}),
\end{gather*}
since by assumption 
\begin{gather*}
\sup_{0 \leq t < \min \lbrace T,\frac{c(n)}{R^{8}+R^{4}} \rbrace} \Vert f \Vert_{L^{\infty}_{\mu}(\Sigma)}\leq R.
\end{gather*}
If $t_{0}=\frac{c(n)}{R^{8}+R^{4}},$ then $T>c(n)\frac{1}{R^{8}+R^{4}},$ which is the required result.\\
If $t_{0}=T$, then we have
\begin{gather*}
\begin{split}
&
t_{0}\leq \frac{c(n)}{R^{8}+R^{4}}, \\
&
\sup_{0 \leq t <t_{0},\, x\in \mathbb{R}^{n}}\int_{B_{2}(x)}\vert A \vert^{2}d\mu \leq \varepsilon,
\\
&
\sup_{0 \leq t < t_{0}} \Vert f \Vert_{L^{\infty}_{\mu}(\Sigma)}\leq 2 R.
\end{split}
\end{gather*}
and the contradiction follows as in the case $t_{0}=T$ in the proof of \ref{thm2}.\\
Clearly we may choose $k(n)=3\Gamma(n).$

\end{proof}

\newpage

\newpage

\section{Appendix}

\begin{lemma} \label{Gronwall's inequality}(Gronwall's inequality)\\
Let $u,\alpha \in C^{0}([0,T),\mathbb{R})$ and $\beta \in C^{0}([0,T),\mathbb{R}_{\geq 0}).$ Then we have
\begin{gather*}
u(t) 
\leq 
\alpha(t)
+
\int^{t}_{0}\beta(s)u(s)\,ds
\quad 
\Longrightarrow 
\quad 
u(t)
\leq 
\alpha(t) 
+
\int^{t}_{0}\alpha(s) \beta(s) e^{\int^{t}_{s}\beta(\sigma)\,d\sigma}ds.
\end{gather*}
\end{lemma}
\begin{proof}
Consider $y(t):=e^{-\int^{t}_{0}\beta(\sigma)d\sigma}\int^{t}_{0}\beta(s)u(s)\,ds.$ By assumption
\begin{gather*}
\begin{split}
y'(t)
= &
-\beta(t) e^{-\int^{t}_{0}\beta(\sigma)d\sigma}\int^{t}_{0}\beta(s)u(s)\,ds
+
\beta(t)u(t) e^{-\int^{t}_{0}\beta(\sigma)d\sigma}\\
= &
[
u(t)
-
\int^{t}_{0}\beta(s)u(s)\,ds
\,]\,
\beta(t)e^{-\int^{t}_{0}\beta(\sigma)d\sigma}
\leq
\alpha(t)\beta(t)e^{-\int^{t}_{0}\beta(\sigma)d\sigma}.
\end{split}
\end{gather*}
Hence by integration 
\begin{gather*}
\begin{split}
e^{-\int^{t}_{0}\beta(\sigma)d\sigma}\int^{t}_{0}\beta(s)u(s)\,ds
=
y(t)-y(0)
\leq 
\int^{t}_{0} \alpha(s)\beta(s)e^{-\int^{s}_{0}\beta(\sigma)d\sigma}ds,
\end{split}
\end{gather*}
what proves the claim by multiplicating both sides with $e^{\int^{t}_{0}\beta(\sigma)d\sigma}.$
\end{proof}

Let in the following $f:\Sigma \longrightarrow \mathbb{R}^{n}$ be a smooth immersion.
\begin{lemma}
For $i \geq 1$ we have 
\begin{gather*}
\begin{split}
\nabla^{i}\Vert f \Vert^{2} 
= &
\sum_{2j+k=i-1}\langle f,\nabla f \rangle *P^{k}_{2j}(A)
+
\sum_{2n+m=i-2}\langle \nabla f , \nabla f \rangle *P^{m}_{2n}(A)\\
& +
\sum_{l+2q+p=i-2}\langle f , \nabla^{l} A \rangle *P^{p}_{2q}(A).
\end{split}
\end{gather*}
\end{lemma}
\begin{proof}
The claim holds obviously true for $i=1.$ Inductively
\begin{gather*}
\begin{split}
\nabla^{i+1} \Vert f & \Vert^{2}
= 
\nabla \nabla^{i} \Vert f \Vert^{2} \\
= &
\sum_{2j+k=i-1} 
[
\langle \nabla f, \nabla f \rangle *P^{k}_{2j}(A)
+
\langle f , A \rangle * P^{k}_{2j}(A)
+
\langle f , \nabla f \rangle * P^{k+1}_{2j}(A)
] \\
& +
\sum_{2n+m=i-2} 
[
\nabla \langle \nabla f , \nabla f \rangle *P^{m}_{2n}(A)
+
\langle \nabla f , \nabla f \rangle *P^{m+1}_{2n}(A)
]\\
& +
\sum_{l+2q+p=i-2}
[
\langle \nabla f , \nabla^{l}A \rangle *P^{p}_{2q}(A)
+
\langle f , \partial^{T} \nabla^{l}A+\nabla^{l+1}A \rangle + P^{p}_{2q+1}\\
& \hspace{165pt}+
\langle f , \nabla^{l} A \rangle + P^{p+1}_{2q}(A)
],
\end{split}
\end{gather*}
where $\partial^{2}f=A+\Gamma \partial f$ was used. \\
Clearly
$\nabla \langle \nabla f , \nabla f \rangle=\nabla g=0$ and 
$\langle \nabla f, \nabla^{l}A \rangle=0.$ Hence
\begin{gather*}
\begin{split}
\nabla^{i+1}\Vert f \Vert^{2} 
= &
\sum_{2j+k=i}\langle f,\nabla f \rangle * P^{k}_{2j}(A)
+
\sum_{2n+m=i-1}\langle \nabla f , \nabla f \rangle * P^{m}_{2n}(A)\\
& +
\sum_{l+2q+p=i-1}\langle f , \nabla^{i} A \rangle * P^{p}_{2q}(A)
+
\sum_{l+2q+p=i-2}\langle f, \partial^{T} \nabla^{l}A\rangle * P^{p}_{2q}(A).
\end{split}
\end{gather*}
We have
\begin{gather*}
\begin{split}
\langle f , \partial^{T}_{k} \nabla^{l}A \rangle
= &
-\langle \partial_{k}f^{T},\nabla^{l}A \rangle 
=
-\langle \partial_{k}(g^{n,m}\langle f, \partial_{n}f \rangle \partial_{m}f)
, \nabla^{l}A \rangle \\
= &
-g^{n,m} \langle f, \partial_{n}f \rangle 
\langle \partial_{k}\partial_{m}f, \nabla^{l} A \rangle
=
-g^{n,m} \langle f, \partial_{n}f \rangle
\langle A_{k,m},\nabla^{l} A \rangle\\
= &
\langle f , \nabla f \rangle * P^{l}_{2}(A)
\end{split}
\end{gather*}
and therefore
\begin{gather*}
\begin{split}
\sum_{l+2q+p=i-2}\langle f, \partial^{T} \nabla^{l}A\rangle * P^{p}_{2q}(A)
= &
\sum_{l+2q+p=i-2} \langle f, \nabla f \rangle *
P^{p+l}_{2q+2}(A) \\
= &
\sum_{2j+k=i}\langle f, \nabla f \rangle * P^{k}_{2j}(A).
\end{split}
\end{gather*}
\end{proof}
Let $P^{l}_{k}(\vert \nabla^{\cdot} A \vert)$ denote any term of the type
$\Pi^{k}_{i=1}\vert \nabla^{l_{i}} A \vert$ with $\sum^{k}_{i=1}l_{i}=l$,

where we freely use $k=0 \Longrightarrow l=0$ and define $P^{0}_{0}(\vert \nabla^{\cdot}A \vert)=1.$
\begin{corollary} \label{derivatives1}
For $i \geq 1 $ we have 
\begin{gather*}
\vert \nabla^{i} \Vert f \Vert^{2} \vert
\leq 
\Vert f \Vert \sum_{l+k=i-1}P^{l}_{k}(\vert \nabla^{\cdot} A \vert)
+
\sum_{2n+m=i-2}P^{m}_{2n}(\vert \nabla^{\cdot} A \vert).
\end{gather*}
\end{corollary}
\begin{proof}
Clearly $\vert \nabla f \vert^{2}=g^{i,j}\langle \partial_{i}f,\partial_{j}f \rangle=g^{i,j}g_{i,j}=\delta^{i}_{i}=dim(\Sigma)$
and
\begin{gather*}
\sum_{l+2q+p=i-2}\langle f , \nabla^{l} A \rangle * P^{p}_{2q}(A)
=
\sum_{l+2q+p=i-2} f*P^{p+l}_{2q+1}(A)
=
\sum_{l+k=i-1}f*P^{l}_{k}(A)
.
\end{gather*}
\end{proof}
From this one obtains
\begin{corollary} \label{tabel}
We have for $i \geq 1$ 
\begin{gather*}
\begin{split}
\vert \nabla^{i} \Vert f \Vert^{8} \vert
\leq &
\sum^{8}_{m=1}\Vert f \Vert^{8-m}\sum_{l+k=i-m}P^{l}_{k}(\vert \nabla^{\cdot} A \vert) \\
\leq & \; c(i) 
\Vert f \Vert^{7}
[\hspace{1pt}
\vert \nabla^{i-2} A \vert
+
\vert \nabla^{i-3} A \vert \vert A \vert
+
\vert \nabla^{i-4} A \vert(\vert \nabla A \vert + \vert A \vert^{2})
+
\ldots
]\\
& \hspace{4pt}+\hspace{3pt}
\Vert f \Vert^{6}
[ \hspace{54pt}
\vert \nabla^{i-3}A \vert
\hspace{16pt} + 
\vert \nabla^{i-4} A \vert \;\,\vert A \vert
\hspace{49pt} + \ldots
]\\
& \hspace{4pt} + \hspace{3pt}
\Vert f \Vert^{5}
[
\hspace{124pt} 
\vert \nabla^{i-4}A \vert
\hspace{70pt} + \ldots
]\\
& \hspace{4pt}+ \hspace{4pt}\ldots \;\; ,
\end{split},
\end{gather*}
where we freely use $k=0 \Longrightarrow l=0$ and define $P^{0}_{0}(\vert \nabla^{\cdot}A \vert)=1.$
\end{corollary}

We calculate the scaling behaviour of some geometric objects.
\begin{lemma} \label{A2}
Let $f:\Sigma \longrightarrow \mathbb{R}^{n}$ be a smooth immersion,\,$d:=dim(\Sigma).$ 
\begin{enumerate}
\item 
$g_{\rho f} = \rho^{2}g_{f} $ and $g^{-}_{\rho f} = \rho^{-2}g^{-}_{f} $ for the metric,
\item 
$\Gamma_{\rho f} = \Gamma_{f} $ for the Christoffel symbols,
\item 
$d\mu_{\rho f}=\rho^{d}d\mu_{f}$ for the area measure,
\item 
$\vert \nabla^{m}_{\rho f}A_{\rho f} \vert^{2}_{g_{_{\rho f}}} = \rho^{-2(m+1)} \vert \nabla^{m}_{f}A_{f} \vert^{2}_{g_{f}}  $ ,
\item 
$\sum_{(i,j,k) \in I(m),\, j<m+4} \nabla^{i}_{\rho f}\Vert \rho f \Vert^{8}* P^{j}_{k}(A_{\rho f})*\nabla^{m}_{\rho f}A_{\rho f} $

\quad \quad\quad$=\rho^{-2(m-1)}\sum_{(i,j,k) \in I(m),\, j<m+4} \nabla^{i}_{f}\Vert f \Vert^{8}* P^{j}_{k}(A_{f})*\nabla^{m}_{f}A_{f} $.
 \end{enumerate}
\end{lemma}

\begin{proof}
 (1) is clear, (2) follows from (1) and 
 \begin{gather*}
  \Gamma^{k}_{i,j} = \frac{1}{2} g^{k,m}(\partial_{i}g_{j,m} + \partial_{j}g_{i,m} - \partial_{m}g_{i,j}).                               
 \end{gather*}
 (3) follows from 
 \begin{gather*}
  J(\rho f) = \sqrt{det(d(\rho f)^{T}\circ d(\rho f)})
  =\sqrt{det(\rho^{2}df^{T}df)}=\rho^{d} Jf.       
 \end{gather*}
To check (4) we derive 
$\nabla^{m}_{\rho f}A_{\rho f} = \rho\nabla^{m}_{f}A_{f}$ by $\nabla_{\rho f}=\nabla_{f},$ the scaling invariance of the covariant derivative and a simple induction argument. Hence 
\begin{gather*}
 \hspace{-320pt}
 \vert \nabla^{m}_{\rho f}A_{\rho f} \vert^{2}_{g_{\rho f}} \\
\begin{split}
 &=\sum_{i_{1},\ldots,i_{m+2} \in \lbrace 1,2 \rbrace} \langle (\nabla^{m}_{\rho f}A_{\rho f})(e^{\rho f}_{i_{1}},\ldots,e^{\rho f}_{i_{m+2}}) ,
 (\nabla^{m}_{\rho f}A_{\rho f})(e^{\rho f}_{i_{1}},\ldots,e^{\rho f}_{i_{m+2}}) \rangle \\
 &=g^{i_{1},i'_{1}}_{\rho f} \ldots g^{i_{m+2},i_{m+2}'}_{\rho f}
  \langle (\nabla^{m}_{\rho f}A_{\rho f})(\partial_{i_{1}},\ldots,\partial_{i_{m+2}}) ,
 (\nabla^{m}_{\rho f}A_{\rho f})(\partial_{i'_{1}},\ldots,\partial_{i'_{m+2}}) \rangle \\
&=\rho^{-2(m+2) +2}g^{i_{1},i'_{1}}_{f} \ldots g^{i_{m+2},i_{m+2}'}_{f} \\
  &\hspace{100pt}\langle (\nabla^{m}_{f}A_{f})(\partial_{i_{1}},\ldots,\partial_{i_{m+1}}) ,
 (\nabla^{m}_{f}A_{f})(\partial_{i'_{1}},\ldots,\partial_{i_{m+1}'}) \rangle \\
& =\rho^{-2(m+1)}  \vert \nabla^{m}_{f}A_{f} \vert^{2}_{g_{f}}.
\end{split}
\end{gather*}
Likewise we obtain for  $(i,j,k) \in I(m)$ recalling (\ref{I(m)=}),
\begin{gather*}
\begin{split}
\nabla^{i}_{\rho f}\Vert \rho f \Vert^{8}* &
 P^{j}_{k}(A_{\rho f})*\nabla^{m}_{\rho f}A_{\rho f} \\
= &
\nabla^{i}_{\rho f}\Vert \rho f \Vert^{8}*
\nabla^{n_{1}}_{\rho f}A_{\rho f}*\ldots*\nabla^{n_{k}}_{\rho f}
A_{\rho f}*\nabla^{m}_{\rho f}A_{\rho f} \\
= &
\rho^{-i} \nabla^{i}_{f} \rho^{8}\Vert f \Vert^{8}*
\rho^{-(n_{1}+1)}\nabla^{n_{1}}_{f}A_{f}*\ldots\\
& \hspace{136pt}
\ldots*\rho^{-(n_{k}+1)}\nabla^{n_{k}}_{\rho f}A_{\rho f}
*\rho^{-(m+1)}\nabla^{m}_{\rho f}A_{\rho f} \\
= &
\rho^{-i+8-(n_{1}+1)-\ldots-(n_{k}+1)-(m+1)}\nabla^{i}_{f}\Vert f \Vert^{8}* P^{j}_{k}(A_{f})*\nabla^{m}_{f}A_{f} \\
= &
\rho^{-2(m-1)}\nabla^{i}_{f}\Vert f \Vert^{8}*
2P^{j}_{k}(A_{f})*\nabla^{m}_{f}A_{f},
\end{split}
\end{gather*}
since $n_{1}+\ldots+n_{k}=j,\quad i+j+k=m+5.$ This proves (5).
\end{proof}
The Willmore functional is invariant under inversion for closed surfaces.
\begin{proposition} \label{invariance}
Let $f\in C^{2}(\Sigma,\mathbb{R}^{n} \setminus \lbrace \mathbb{0} \rbrace)$ be a closed immersed surface, 

$I:\mathbb{R}^{n} \setminus \lbrace \mathbb{0} \rbrace \longrightarrow \mathbb{R}^{n} \setminus \lbrace \mathbb{0} \rbrace:x \rightarrow \frac{x}{\Vert x \Vert^{2}}.$ 
Then we have
\begin{gather*}
 W(f):=\frac{1}{4}\int_{\Sigma}\vert H_{f} \vert^{2}d\mu_{f}
=\frac{1}{4}\int_{\Sigma}\vert H_{I_{\sharp}f} \vert^{2}d\mu_{I_{\sharp}f}
=W(I_{\sharp}f),
\end{gather*}

where $I_{\sharp}f(x):=I(f(x))$ for $x \in \Sigma.$
\end{proposition}
\begin{proof}
For the metric we have
\begin{gather*}
\begin{split}
g^{I_{\sharp}f}_{i,j}
& =
\langle \partial_{i} \frac{f}{\Vert f \Vert^{2}},
\partial_{j} \frac{f}{\Vert f \Vert^{2}} \rangle
=
\langle \frac{\partial_{i} f}{\Vert f \Vert^{2}}
- 2\frac{\langle \partial_{i} f , f \rangle}{\Vert f \Vert^{4}}f,
\frac{\partial_{j} f}{\Vert f \Vert^{2}}
- 2\frac{\langle \partial_{j} f , f \rangle}{\Vert f \Vert^{4}}f \rangle \\
& =
\frac{1}{\Vert f \Vert^{4}} g^{f}_{i,j},
\end{split}
\end{gather*}
hence $g^{i,j}_{I_{\sharp}f}=\Vert f \Vert^{4}g^{i,j}_{f}.$ 
To calculate the mean curvature 
\begin{gather*}
H_{I_{\sharp}f}=g^{i,j}_{I_{\sharp}f}A^{I_{\sharp}f}_{i,j},
\quad
A^{I_{\sharp}f}_{i,j}
=
\partial^{\perp_{I_{\sharp}f}}_{i}\partial_{j}I_{\sharp}f
=\partial_{i}\partial_{j}I_{\sharp}f-P^{T_{I_{\sharp}f}}(\partial _{i}\partial_{j}I_{\sharp}f)
\end{gather*}
we derive
\begin{gather*}
\begin{split}
\partial_{i}\partial_{j}I_{\sharp}f
= & 
\partial_{i}(\frac{\partial_{j} f}{\Vert f \Vert^{2}}
- 2\frac{\langle \partial_{j} f , f \rangle}{\Vert f \Vert^{4}}f)\\
= &
\frac{\partial_{i }\partial_{j} f}{\Vert f \Vert^{2}}
- 2\frac{\langle \partial_{i} f , f \rangle}{\Vert f \Vert^{4}}\partial_{j}f
- 2\frac{\langle \partial_{j} f , f \rangle}{\Vert f \Vert^{4}}\partial_{i}f \\
& -
2\frac{\langle \partial_{i} \partial_{j} f , f \rangle}{\Vert f \Vert^{4}}f
- 2\frac{\langle \partial_{i} f , \partial_{j} f \rangle}{\Vert f \Vert^{4}}f
+ 8\frac
{\langle \partial_{i} f , f \rangle 
\langle \partial_{j} f , f \rangle}
{\Vert f \Vert^{6}}f,
\end{split}
\end{gather*}
and 
\begin{gather*}
\begin{split}
P^{T_{I_{\sharp}f}}(\partial_{i}\partial_{j} I_{\sharp}f)
= &
g^{k,l}_{I_{\sharp}f}\langle \partial_{i}\partial_{j} I_{\sharp}f,\partial_{k} \frac{f}{\Vert f \Vert^{2}} \rangle \partial_{l} \frac{f}{\Vert f \Vert^{2}} \\
= &
\Vert f \Vert^{4} g^{k,l}_{f} \langle \partial_{i}\partial_{j} I_{\sharp}f,
\frac{\partial_{k} f}{\Vert f \Vert^{2}}
- 2\frac{\langle \partial_{k} f , f \rangle}{\Vert f \Vert^{4}}f \rangle
(\frac{\partial_{l} f}{\Vert f \Vert^{2}}
- 2\frac{\langle \partial_{l} f , f \rangle}{\Vert f \Vert^{4}}f) \\
= &
g^{k,l}_{f}\langle \partial_{i} \partial_{j} I_{\sharp}f,\partial_{k}f \rangle \partial_{l} f \\
& -
\frac{2}{\Vert f \Vert^{2}}g^{k,l}_{f}
\langle \partial_{i} \partial_{j} I_{\sharp}f,\partial_{k}f \rangle 
\langle \partial_{l} f ,f \rangle f \\
& -
\frac{2}{\Vert f \Vert^{2}} g^{k,l}_{f}
\langle \partial_{i} \partial_{j} I_{\sharp}f, f \rangle \langle \partial_{k}f,f\rangle
\partial_{l}f \\
& +
\frac{4}{\Vert f \Vert^{4}}g^{k,l}_{f}
\langle \partial_{i}\partial_{j}I_{\sharp}f,f \rangle \langle \partial_{k}f,f \rangle
\langle \partial_{l} f ,f \rangle f.
\end{split}
\end{gather*}
For the several summands of the mean curvature we obtain
\begin{enumerate}
\item
\begin{gather*}
\begin{split}
g^{i,j}_{I_{\sharp}f}\partial_{i}\partial_{j}I_{\sharp}f 
= &
\Vert f \Vert^{2} g^{i,j}_{f}\partial_{i}\partial_{j}f
-4f^{T_{f}}
-2\langle g^{i,j}_{f}\partial_{i}\partial_{j}f,f \rangle f
-4f
+8\frac{\Vert f^{T_{f}} \Vert^{2}}{\Vert f \Vert^{2}}f,
\end{split}
\end{gather*}
\item
\begin{gather*}
\begin{split}
g^{i,j}_{I_{\sharp}f}g^{k,l}_{f} \langle & \partial_{i}\partial_{j}I_{\sharp}f, \partial_{k}f \rangle \partial_{l}f \\
= &
P^{T_{f}}
(
\Vert f \Vert^{2} g^{i,j}_{f}\partial_{i}\partial_{j}f
-4f^{T_{f}}
-2\langle g^{i,j}_{f}\partial_{i}\partial_{j}f,f \rangle f
-4f
+8\frac{\Vert f^{T_{f}} \Vert^{2}}{\Vert f \Vert^{2}}f
) \\
= &
\Vert f \Vert^{2} g^{i,j}_{f} \partial^{T_{f}}_{i}\partial_{j}f
-8f^{T_{f}}
-2\langle g^{i,j}_{f}\partial_{i}\partial_{j}f,f \rangle f^{T_{f}}
+ 8\frac{\Vert f^{T_{f}} \Vert^{2}}{\Vert f \Vert^{2}}f^{T_{f}},
\end{split}
\end{gather*}
\item
\begin{gather*}
\begin{split}
-\frac{2}{\Vert f   \Vert^{2}} & g^{i,j}_{I_{\sharp}f} 
g^{k,l}_{f} 
\langle \partial_{i} \partial_{j}  I_{\sharp}f , \partial_{k}f \rangle 
\langle \partial_{l} f ,f \rangle f \\
= &
-\frac{2}{\Vert f \Vert^{2}}\langle 
\Vert f \Vert^{2} g^{i,j}_{f}\partial_{i}\partial_{j}f
-4f^{T_{f}} \\
& \hspace{43pt} -
2\langle g^{i,j}_{f}\partial_{i}\partial_{j}f,f \rangle f
-4f
+8\frac{\Vert f^{T_{f}} \Vert^{2}}{\Vert f \Vert^{2}}f
,f^{T_{f}} \rangle f \\
= &
-2\langle g^{i,j}_{f}\partial_{i} \partial_{j}f,f^{T_{f}} \rangle f
+4\frac{\Vert f^{T_{f}} \Vert^{2}}{\Vert f \Vert^{2}}\langle g^{i,j}_{f}\partial_{i} \partial_{j}f,f \rangle  f \\
&
+16\frac{\Vert f^{T_{f}} \Vert^{2}}{\Vert f \Vert^{2}} f
\hspace{-2.5pt}
-16\frac{\Vert f^{T_{f}} \Vert^{4}}{\Vert f \Vert^{4}}f,
\end{split}
\end{gather*}
\item
\begin{gather*}
\begin{split}
-\frac{2}{\Vert f \Vert^{2}} g^{i,j}_{I_{\sharp}f}g^{k,l}_{f} &
\langle \partial_{i} \partial_{j} I_{\sharp}f, f \rangle \langle \partial_{k}f,f\rangle
\partial_{l}f \\
= &
-\frac{2}{\Vert f \Vert^{2}}
\langle 
\Vert f \Vert^{2} g^{i,j}_{f}\partial_{i}\partial_{j}f
-4f^{T_{f}} \\
& \hspace{43pt} -
2\langle g^{i,j}_{f}\partial_{i}\partial_{j}f,f \rangle f
-4f
+8\frac{\Vert f^{T_{f}} \Vert^{2}}{\Vert f \Vert^{2}}f,f \rangle f^{T_{f}} \\
= &
2\langle g^{i,j}_{f} \partial_{i} \partial_{j}f,f \rangle f^{T_{f}}
+8\frac{\Vert f^{T_{f}} \Vert^{2}}{ \Vert f \Vert^{2}}f^{T_{f}}
+8f^{T_{f}}
-16\frac{\Vert f^{T_{f}} \Vert^{2}}{\Vert f \Vert^{2}}f^{T_{f}},
\end{split}
\end{gather*}
\item
\begin{gather*}
\begin{split}
\frac{4}{\Vert f \Vert^{4}} g^{i,j}_{I_{\sharp}f}g^{k,l}_{f} &
\langle \partial_{i}\partial_{j}I_{\sharp}f,f \rangle \langle \partial_{k}f,f \rangle
\langle \partial_{l} f ,f \rangle f \\
= &
\frac{4}{\Vert f \Vert^{4}} \langle
\Vert f \Vert^{2} g^{i,j}_{f}\partial_{i}\partial_{j}f
-4f^{T_{f}} \\
& \hspace{29pt} -
2\langle g^{i,j}_{f}\partial_{i}\partial_{j}f,f \rangle f
-4f
+8\frac{\Vert f^{T_{f}} \Vert^{2}}{\Vert f \Vert^{2}}f,f\rangle
\Vert f^{T_{f}} \Vert^{2}f \\
= &
-4\frac{\Vert f^{T_{f}} \Vert^{2}}{\Vert f \Vert^{2}} \langle g^{i,j}_{f} \partial_{i}\partial_{j}f,f \rangle f
+16\frac{\Vert f^{T_{f}} \Vert^{4}}{\Vert f \Vert^{4}}f
-16\frac{\Vert f^{T_{f}} \Vert^{2}}{\Vert f \Vert^{2}}f.
\end{split}
\end{gather*}
\end{enumerate}
Collecting terms we conclude
\begin{gather*}
\begin{split}
H_{I_{\sharp}f}
= &
g^{i,j}_{I_{\sharp}f}A^{I_{\sharp}f}_{i,j}
= 
g^{i,j}_{I_{\sharp}f}\partial_{i}\partial_{j}I_{\sharp}f
-g^{i,j}_{I_{\sharp}f}P^{T_{I_{\sharp}f}}(\partial_{i}\partial_{j} I_{\sharp}f) \\
= &
\Vert f \Vert^{2} g^{i,j}_{f}\partial_{i}\partial_{j}f
-4f^{T_{f}}
-2\langle g^{i,j}_{f}\partial_{i}\partial_{j}f,f \rangle f
-4f
+8\frac{\Vert f^{T_{f}} \Vert^{2}}{\Vert f \Vert^{2}}f \\
& -
(\Vert f \Vert^{2} g^{i,j}_{f} \partial^{T_{f}}_{i}\partial_{j}f
-8f^{T_{f}}
-2\langle g^{i,j}_{f}\partial_{i}\partial_{j}f,f \rangle f^{T_{f}}
+ 8\frac{\Vert f^{T_{f}} \Vert^{2}}{\Vert f \Vert^{2}}f^{T_{f}})\\
& -(-2\langle g^{i,j}_{f}\partial_{i} \partial_{j}f,f^{T_{f}} \rangle f
+4\frac{\Vert f^{T_{f}} \Vert^{2}}{\Vert f \Vert^{2}}\langle g^{i,j}_{f}\partial_{i} \partial_{j}f,f \rangle  f \\
& \hspace{116pt}
+16\frac{\Vert f^{T_{f}} \Vert^{2}}{\Vert f \Vert^{2}} f
\hspace{-2.5pt}
-16\frac{\Vert f^{T_{f}} \Vert^{4}}{\Vert f \Vert^{4}}f) \\
& -
(2\langle g^{i,j}_{f} \partial_{i} \partial_{j}f,f \rangle f^{T_{f}}
+8\frac{\Vert f^{T_{f}} \Vert^{2}}{ \Vert f \Vert^{2}}f^{T_{f}}
+8f^{T_{f}}
-16\frac{\Vert f^{T_{f}} \Vert^{2}}{\Vert f \Vert^{2}}f^{T_{f}}) \\
& -
(-4\frac{\Vert f^{T_{f}} \Vert^{2}}{\Vert f \Vert^{2}} \langle g^{i,j}_{f} \partial_{i}\partial_{j}f,f \rangle f 
+16\frac{\Vert f^{T_{f}} \Vert^{4}}{\Vert f \Vert^{4}}f
-16\frac{\Vert f^{T_{f}} \Vert^{2}}{\Vert f \Vert^{2}}f) \\
= &
\Vert f \Vert^{2} H_{f} 
-2\langle g^{i,j}_{f}\partial_{i}\partial_{j}f,f^{\perp_{f}}\rangle f
-4f^{T_{f}}-4f+8\frac{\Vert f^{T_{f}} \Vert^{2}}{\Vert f \Vert^{2}}f \\
= &
\Vert f \Vert^{2} H_{f} 
-2\langle H_{f},f\rangle f
-4f^{T_{f}}-4f+8\frac{\Vert f^{T_{f}} \Vert^{2}}{\Vert f \Vert^{2}}f. \\
\end{split}
\end{gather*}
Hence we obtain for the squared absolute value 
\begin{gather*}
\begin{split}
\vert H_{I_{\sharp}f}\vert^{2}
= &
\Vert f \Vert^{4} \vert H_{f} \vert^{2}
-2\Vert f \Vert^{2} \langle H_{f},f \rangle^{2}
-4\Vert f \Vert^{2}\langle H_{f},f \rangle
+8\Vert f^{T_{f}}\Vert^{2} \langle H_{f},f \rangle \\
& -
2\Vert f \Vert^{2} \langle H_{f},f \rangle^{2}
+4\Vert f \Vert^{2}\langle H_{f},f \rangle^{2}
+8\Vert f^{T_{f}} \Vert^{2}\langle H_{f},f \rangle \\
&
+8\Vert f \Vert^{2}\langle H_{f},f \rangle 
-16\Vert f^{T_{f}} \Vert^{2} \langle H_{f},f \rangle \\
& +
8\Vert f^{T_{f}} \Vert^{2}\langle H_{f},f \rangle
+16 \Vert f^{T_{f}}\Vert^{2}
+16 \Vert f^{T_{f}}\Vert^{2}
-32 \frac{\Vert f^{T_{f}} \Vert^{4}}{\Vert f \Vert^{2}} \\
& 
-4\Vert f \Vert^{2}\langle H_{f},f \rangle
+8\Vert f \Vert^{2}\langle H_{f},f \rangle 
+16 \Vert f^{T_{f}} \Vert^{2}
+16 \Vert f \Vert^{2}
-32\Vert f^{T_{f}} \Vert^{2} \\
& +
8\Vert f^{T_{f}}\Vert^{2} \langle H_{f},f \rangle
-16\Vert f^{T_{f}} \Vert^{2} \langle H_{f},f \rangle
-32 \frac{\Vert f^{T_{f}} \Vert^{4}}{\Vert f \Vert^{2}} \\
&
-32\Vert f^{T_{f}} \Vert^{2}
+64\frac{\Vert f^{T_{f}}\Vert^{4}}{\Vert f\Vert^{2}} \\
= &
\Vert f \Vert^{4} \vert H_{f} \vert^{2}
+8 \Vert f \Vert^{2} \langle H_{f} ,f \rangle
-16 \Vert f^{T_{f}} \Vert^{2}
+16 \Vert f \Vert^{2} \\
= &
\vert \Vert f \Vert^{2} H_{f} + 4 f^{\perp_{f}} \vert^{2}.
\end{split}
\end{gather*}
By  (\ref{eq0})  we have
\begin{gather*}
\begin{split}
\int_{\Sigma} \vert H_{I_{\sharp}f} \vert^{2}d\mu_{I_{\sharp}f} 
= &
\int_{\Sigma} \Vert f \Vert^{-4} \vert \,\Vert f \Vert^{2}H_{f}+4f^{\perp_{f}} \vert^{2}d\mu_{f} \\
= &
\int_{\Sigma} \vert H_{f} \vert^{2} d\mu_{f}
+8 \int_{\Sigma}\frac{\langle H_{f},f \rangle}{\Vert f \Vert^{2}}d\mu_{f}
+16 \int_{\Sigma}\frac{\vert f^{\perp_{f}}\vert^{2}}{\Vert f \Vert^{4}}d\mu_{f}
\end{split}
\end{gather*}
and
\begin{gather*}
\begin{split}
0
= &
\int_{\Sigma} g^{i,j}_{f}\nabla_{i}(\Vert f \Vert^{-2} \langle \partial_{j}f,f \rangle)d\mu_{f} \\
= &
-2\int_{\Sigma}g^{i,j}_{f}\frac{\langle \partial_{i}f,f \rangle \langle \partial_{j}f,f \rangle}{\Vert f \Vert^{4}}d\mu_{f}
+\int_{\Sigma}g^{i,j}_{f}\frac{\langle\partial_{i}\partial_{j}f,f\rangle}{\Vert f \Vert^{2}}d\mu_{f} \\
& +
\int_{\Sigma}g^{i,j}_{f}\frac{\langle \partial_{i}f,\partial_{j}f \rangle}{\Vert f \Vert^{2}}d\mu_{f}
-\int_{\Sigma}g^{i,j}_{f}\Gamma^{k}_{i,j}\frac{\langle \partial_{k}f,f \rangle}{\Vert f \Vert^{2}}d\mu_{f} \\
= &
-2\int_{\Sigma}\frac{\Vert f^{T_{f}} \Vert^{2}}{\Vert f \Vert^{4}}d\mu_{f}
+2\int_{\Sigma}\frac{1}{\Vert f \Vert^{2}}d\mu_{f}
+\int_{\Sigma}g^{i,j}_{f}\frac{\langle \partial_{i}\partial_{j}f-\Gamma^{k}_{i,j}\partial_{k}f,f \rangle}{\Vert f \Vert^{2}}d\mu_{f} \\
= &
2 \int_{\Sigma} \frac{\vert f^{\perp_{f}}\vert^{2}}{\Vert f \Vert^{4}}d\mu_{f}
+\int_{\Sigma}\frac{\langle H_{f},f \rangle}{\Vert f \Vert^{2}}d\mu_{f},
\end{split}
\end{gather*}
which concludes the proof.
\end{proof}
\begin{lemma} \label{A3}
Let $f:\Sigma \longrightarrow \mathbb{R}^{n}$ be a smooth immersion. 

Then we have
 \begin{gather*}
  (\Delta H + Q(A^{0})H)_{\rho f} = \rho^{-3}(\Delta H + Q(A^{0})H)_{f}.
 \end{gather*}
\end{lemma}
\begin{proof}
From lemma \ref{A2} we infer 
\begin{gather*}
\begin{split}
H_{\rho f}
&= A_{\rho f}(e^{\rho f}_{i} , e^{\rho f}_{i})
 =g^{i,j}_{\rho f}A_{\rho f}(\partial_{i} , \partial_{j})
 =\rho ^{-2} g^{i,j}_{f} \rho A_{f}(\partial_{i} , \partial_{j}) \\
&=\rho^{-1}H_{f}.
\end{split}
\end{gather*}
Therefore the Lapace-Beltrami of the mean curvature becomes
\begin{gather*}
\begin{split}
 (\Delta H)_{\rho f}
&=\rho^{-1} \Delta_{\rho f} (H_{f}) 
 =\rho^{-1}g^{i,j}_{\rho f} (\nabla^{\rho f}_{\partial_{i}} 
  \nabla^{\rho f}_{\partial_{j}} H_{f}  
                                             - \nabla^{\rho f}_{\nabla^{\rho f}_{\partial_{i}}\partial_{j}}H_{f} ) \\
&=\rho^{-3}g^{i,j}_{f} (\nabla^{f}_{\partial_{i}} \nabla^{f}_{\partial_{j}} H_{f}   
                                             - \nabla^{f}_{\nabla^{f}_{\partial_{i}}\partial_{j}}H_{f} ) \\
 &=\rho^{-3} (\Delta H)_{f}  .                                \end{split}
\end{gather*}
Next we derive for $Q(A^{0}),$ cf (\ref{formula13}), 
\begin{gather*}
 \begin{split}
 (Q(A^{0})H)_{\rho f} 
 &=  A^{0}_{\rho f}(e^{\rho f}_{i} , e^{\rho f}_{j}) \langle A^{0}_{\rho f}(e^{\rho f}_{i} , e^{\rho f}_{j}) , H_{\rho f} \rangle \\
 &=\rho^{-1} g^{k,l}_{\rho f}g^{n,m}_{\rho f} A^{0}_{\rho f}(\partial_{k}, \partial_{n}) \langle A^{0}_{\rho f}(\partial_{l}, \partial_{m}) , H_{f} \rangle \\
&=\rho^{-5}g^{k,l}_{f}g^{n,m}_{f}  A^{0}_{\rho f}(\partial_{k}, \partial_{n}) \langle A^{0}_{\rho f}(\partial_{n}, \partial_{m}) , H_{f} \rangle.
\end{split}
\end{gather*}
Since the tracefree part of the second fundamental form satisfies by (\ref{formula12})
\begin{gather*}
 \begin{split}
A^{0}_{\rho f}(\partial_{i} , \partial_{j}) 
= & 
A_{\rho f}(\partial_{i}  , \partial_{j}) - \frac{1}{2}g_{\rho f}( \partial_{i} , \partial_{j})H_{\rho f}\\
= &
\partial_{i}\partial_{j}(\rho f) - _{\rho f}\Gamma^{k}_{i,j}\partial_{k}(\rho f) - \frac{1}{2}\langle \partial_{i}(\rho f) , \partial_{j}(\rho f) \rangle \rho^{-1} H_{f}\\
= &
\rho(\partial_{i}\partial_{j}f - _{f}\Gamma^{k}_{i,j}\partial_{k}f -\frac{1}{2}g_{f}(\partial_{i} , \partial_{j})H_{f}) \\
= &
\rho A^{0}_{f}(\partial_{i} , \partial_{j}),
\end{split}
\end{gather*}
we conclude
\begin{gather*}
(Q(A^{0})H)_{\rho f} = \rho^{-3}g^{k,l}_{f}g^{n,m}_{f}  A^{0}_{ f}(\partial_{k}, \partial_{n}) \langle A^{0}_{f}(\partial_{l}, \partial_{m}) , H_{f} \rangle =\rho^{-3}(Q(A^{0})H)_{f}.
\end{gather*}
\end{proof}
We state two Sobolev-type inequalities.
\begin{lemma} \label{A4}
Let $f:\Sigma \rightarrow \mathbb{R}^{n}$ be a smooth, immersed surface.\\ Then we have for $u \in C^{1}_{0}(\Sigma)$ 
 \begin{gather*}
 ( \int_{\Sigma} u^{2} d\mu)^{\frac{1}{2}} \leq c \;( \int_{\Sigma}  \vert \nabla u \vert d\mu + \int_{\Sigma} \vert H \vert \vert u \vert d\mu ),
 \end{gather*}
 where $\mu=\mu_{f}$ and $c>0$.
\end{lemma} 
\begin{proof}
Cf. \cite{ref2}.
\end{proof}
\begin{lemma} \label{A5}
Let $f:\Sigma \rightarrow \mathbb{R}^{n}$ be a smooth, immersed surface. Furthermore 
\begin{center}
$2<p\leq \infty,0 \leq m\leq \infty$ and $0<\alpha \leq 1$
with $\frac{1}{\alpha}=(\frac{1}{2}-\frac{1}{p})\,m+1.$             \end{center}
Then we have for $u \in C^{1}_{0}(\Sigma)$
\begin{gather*}
\Vert u \Vert_{L^{\infty}_{\mu}(\Sigma)}\leq c \Vert u \Vert^{1-\alpha}_{L^{m}_{\mu}(\Sigma)}(\Vert \nabla u \Vert_{L^{p}_{\mu}(\Sigma)}+\Vert H u \Vert_{L^{p}_{\mu}(\Sigma)})^{\alpha},
\end{gather*}
where $c=c(n,m,p)>0.$
\end{lemma}
\begin{proof}
Cf. Theorem 5.6 in \cite{ref1}.
\end{proof}
\newpage
We prove some localized interpolation inequalities.\\
Let  $f:\Sigma \rightarrow \mathbb{R}^{n}$ be a smooth immersion, $d=dim(\Sigma)$, $\Phi $ a $l$-linear form along $f$ and $\gamma \in C^{1}_{0}(\Sigma)$ with
$0 \leq \gamma\leq 1,\,\Vert \nabla \gamma \Vert_{L^{\infty}_{\mu}(\Sigma)} \leq \Lambda.$
\begin{proposition} \label{A6}
Let 
$1 \leq p,q,r <\infty,\;\frac{1}{p}+\frac{1}{q}=\frac{1}{r},
\;\alpha + \beta=\eta+\theta=1.$ 
For 
\begin{center}
$k \in \lbrace 0,\mathbb{R}_{\geq \eta p,\theta q}\rbrace,\,s \geq \alpha p,\beta q$  
and
$-\frac{s}{q}\leq u_{1},u_{2} \leq \frac{s}{p},
\;-\frac{k}{q}\leq v_{1},v_{2} \leq \frac{k}{p}$                    \end{center}
we have
\begin{gather*}
\begin{split}
(\int_{\Sigma}\Vert f \Vert & ^{k}  \vert \nabla \Phi  \vert^{2r}\gamma^{s}d\mu)^{\frac{1}{r}}\\
\leq & 
ck(\int_{[\gamma>0]}\Vert f \Vert^{k-\eta p}\vert \Phi \vert^{p}\gamma^{s-u_{1}p}d\mu)^{\frac{1}{p}}
(\int_{[\gamma>0]}\Vert f \Vert^{k-\theta q}\vert \nabla \Phi \vert^{q}\gamma^{s+u_{1}q}d\mu)^{\frac{1}{q}}\\
&+
c(\int_{[\gamma>0]}\Vert f \Vert^{k-v_{2} p}\vert \Phi \vert^{p}\gamma^{s-u_{2}p}d\mu)^{\frac{1}{p}}
(\int_{[\gamma>0]}\Vert f \Vert^{k+v_{2}q}\vert \nabla^{2} \Phi \vert^{q} \gamma^{s+u_{2}q}d\mu)^{\frac{1}{q}}\\
&+
c s \Lambda(\int_{[\gamma>0]}\Vert f \Vert^{k-v_{1}p}\vert \Phi \vert^{p}\gamma^{s-\alpha p}d\mu)^{\frac{1}{p}}
(\int_{[\gamma>0]}\Vert f \Vert^{k+v_{1}q}\vert \nabla \Phi \vert^{q}\gamma^{s-\beta q}d\mu)^{\frac{1}{q}},
\end{split}
\end{gather*}
where $0\cdot \infty:=0,\,c=c(r,d)$ and $\Vert \nabla \gamma \Vert_{L^{\infty}_{\mu}(\Sigma)}\leq \Lambda.$
 \end{proposition}
\begin{proof}
Let $k\geq \eta p,\theta q.$ By integration by parts and H\"older's inequality 
\begin{gather*}
\begin{split}
\int_{\Sigma}\Vert f \Vert & ^{k} 
\vert \nabla \Phi \vert^{2r}\gamma^{s}d\mu \\
\leq &
ck\int_{\Sigma}\Vert f \Vert^{k-1}\vert \Phi \vert \vert \nabla \Phi \vert^{2r-1}\gamma^{s}d\mu
+
c\int_{\Sigma}\Vert f \Vert^{k}\vert \Phi \vert\vert \nabla \Phi \vert^{2r-2}\vert \nabla^{2}\Phi \vert\gamma^{s}d\mu \\
& +
c s \Lambda \int_{\Sigma}\Vert f \Vert^{k}\vert \Phi \vert \vert \nabla \Phi \vert^{2r-1}\gamma^{s-1}d\mu\\
=&
ck\int_{\Sigma}\Vert f \Vert^{\frac{k}{p}-\eta}\vert \Phi \vert \gamma^{\frac{s}{p}-u_{1}}
\Vert f \Vert^{k\frac{r-1}{r}}\vert \nabla \Phi \vert^{2r-2}\gamma^{s\frac{r-1}{r}}
\Vert f \Vert^{\frac{k}{q}-\theta}\vert \nabla \Phi \vert \gamma^{\frac{s}{q}+u_{1}}d\mu\\
&+
c\int_{\Sigma}\Vert f \Vert^{\frac{k}{p}-v_{2}}\vert \Phi \vert\gamma^{\frac{s}{p}-u_{2}}
\Vert f \Vert^{k\frac{r-1}{r}}\vert \nabla \Phi \vert^{2r-2}\gamma^{s\frac{r-1}{r}}
\Vert f \Vert^{\frac{k}{q}+v_{2}}\vert \nabla^{2}\Phi \vert\gamma^{\frac{s}{q}+u_{2}}d\mu\\
&+
c s \Lambda\int_{\Sigma}\Vert f \Vert^{\frac{k}{p}-v_{1}}\vert \Phi \vert \gamma^{\frac{s}{p}-\alpha}
\Vert f \Vert^{k\frac{r-1}{r}}\vert \nabla \Phi \vert^{2r-2}\gamma^{s\frac{r-1}{r}}
\Vert f \Vert^{\frac{k}{q}+v_{1}}\vert \nabla \Phi \vert \gamma^{\frac{s}{q}-\beta}d\mu\\
\leq &
ck(\int_{[\gamma>0]}\Vert f \Vert^{k-\eta p}\vert \Phi \vert^{p}\gamma^{s-u_{1}p}d\mu)^{\frac{1}{p}}\\
&\hspace{51pt}
(\int_{\Sigma}\Vert f \Vert^{k} \vert \nabla \Phi \vert^{2r}\gamma^{s}d\mu)^{\frac{r-1}{r}}
(\int_{[\gamma>0]}\Vert f \Vert^{k-\theta q}\vert \nabla \Phi \vert^{q}\gamma^{s+u_{1}q}d\mu)^{\frac{1}{q}}\\
&+
c(\int_{[\gamma>0]}\Vert f \Vert^{k-v_{2} p}\vert \Phi \vert^{p}\gamma^{s-u_{2}p}d\mu)^{\frac{1}{p}}\\
& \hspace{51pt}
(\int_{\Sigma}\Vert f \Vert^{k} \vert \nabla \Phi \vert^{2r}\gamma^{s}d\mu)^{\frac{r-1}{r}}
(\int_{[\gamma>0]}\Vert f \Vert^{k+v_{2} q}\vert \nabla^{2} \Phi \vert^{q}\gamma^{s+u_{2}q}d\mu)^{\frac{1}{q}}\\
&+
c s \Lambda(\int_{[\gamma>0]}\Vert f \Vert^{k-v_{1}p}\vert \Phi \vert^{p}\gamma^{s-\alpha p}d\mu)^{\frac{1}{p}}\\
& \hspace{51pt}
(\int_{\Sigma}\Vert f \Vert^{k} \vert \nabla \Phi \vert^{2r}\gamma^{s}d\mu)^{\frac{r-1}{r}}
(\int_{[\gamma>0]}\Vert f \Vert^{k+v_{1}q}\vert \nabla \Phi \vert^{q} \gamma^{s-\beta q}d\mu)^{\frac{1}{q}}.
\end{split}
\end{gather*}
The case $k=0$ follows analogously.
\end{proof}
\newpage
\begin{corollary} \label{A7}
Let $s \geq p \geq 2,\,k \in \lbrace 0,\mathbb{R}_{\geq p} \rbrace$ 
,
$
-\frac{k}{p}\leq u \leq \frac{k}{p},
\;
-\frac{s}{p}\leq v \leq \frac{s}{p}.
$ \\
Then we have
\begin{gather*}
\begin{split}
\int_{\Sigma} \Vert  f & \Vert^{k}  \vert   \nabla \Phi  
\vert^{p}\gamma^{s}d\mu \\
\leq &
\varepsilon \int_{[\gamma>0]} \Vert f \Vert^{k+up}\vert \nabla^{2} \Phi \vert^{p}
\gamma^{s+vp} d\mu \\
& \hspace{-3pt}+
c(\varepsilon,d,p) 
\hspace{-3pt}
\int_{[\gamma>0]}
[\,
k^{p} \Vert f \Vert^{k-p}\gamma^{s}
+
\Vert f \Vert^{k-up} \gamma^{s-vp}
+
\hspace{-2pt}
s^{p}\Lambda^{p} \Vert f \Vert^{k}\gamma^{s-p}
]
\vert \Phi \vert^{p} d\mu.
\end{split}
\end{gather*} 
\end{corollary}
\begin{proof}
Applying proposition \ref{A6} with
\begin{gather*}
p=q=2r,\,\alpha=\eta=1,\,\beta=\theta=0,\,u_{1}=v_{1}=0  
\end{gather*}
we obtain for
$
-\frac{k}{q}\leq u:=v_{2} \leq \frac{k}{p},
\;
-\frac{s}{q}\leq v:=u_{2} \leq \frac{s}{p}
$
, since $s>0,$
\begin{gather*}
\begin{split}
(\int_{\Sigma}\Vert f \Vert^{k} & \vert \nabla \Phi \vert^{p}\gamma^{s}d\mu)^{\frac{2}{p}} \\
\leq &
ck(\int_{[\gamma>0]}\Vert f \Vert^{k-p}\vert \Phi \vert^{p}\gamma^{s}d\mu)^{\frac{1}{p}}
(\int_{\Sigma}\Vert f \Vert^{k}\vert \nabla \Phi \vert^{p} \gamma^{s}d\mu)^{\frac{1}{p}}\\
& +
c(\int_{[\gamma>0]} \Vert f \Vert^{k-up}\vert \Phi \vert^{p}\gamma^{s-vp}d\mu)^{\frac{1}{p}}
(\int_{[\gamma>0]}\Vert f \Vert^{k+up} 
\vert \nabla^{2} \Phi \vert^{p}\gamma^{s+vp}d\mu)^{\frac{1}{p}}\\
& +
cs \Lambda (\int_{[\gamma>0]}\Vert f \Vert^{k}\vert \Phi \vert^{p} \gamma^{s-p}d\mu)^{\frac{1}{p}}
(\int_{\Sigma}\Vert f \Vert^{k}\vert \nabla \Phi \vert^{p} \gamma^{s}d\mu)^{\frac{1}{p}}.
\end{split}
\end{gather*}
Monotonicity and convexity of $t \rightarrow t^{\frac{p}{2}},\,t \geq 0,$ and Young's inequality
yield
\begin{gather*}
\begin{split}
\int_{\Sigma}\Vert f & \Vert^{k}  \vert \nabla \Phi \vert^{p}\gamma^{s}d\mu \\
\leq &
\varepsilon \int_{\Sigma}\Vert f \Vert^{k}
\vert \nabla \Phi \vert^{p} \gamma^{s}d\mu
+
\varepsilon \int_{[\gamma>0]}\Vert f \Vert^{k+up}
\vert \nabla^{2} \Phi \vert^{p}\gamma^{s+vp}d\mu \\
& +
c(\varepsilon,n,p) \int_{[\gamma>0]}
[\,
k^{p}\Vert f \Vert^{k-p}\gamma^{s}
+
\Vert f \Vert^{k-up}\gamma^{s-vp}
+
\Lambda^{p} \Vert f \Vert^{k}\gamma^{s-p}
\,]
\vert \Phi \vert^{p}d\mu.
\end{split}
\end{gather*}
Absorption and replacing $(1-\varepsilon)^{-1}\varepsilon$ by $\varepsilon$
prove the claim.
\end{proof}

\end{document}